\documentclass[10pt,oneside,reqno,fleqn]{amsart}

\newtheorem{thm}{Theorem}

\newtheorem{lem}{Lemma}
\newtheorem*{mydef}{Definition}
\newtheorem*{Hyp H}{Hypothesis (H)}
\newtheorem{rem}{Remark}
\newtheorem{prop}{Proposition}
\numberwithin{equation}{section}

\usepackage{setspace}
\usepackage[left=2.5cm,right=2.5cm,bottom=2cm]{geometry}
\usepackage{newtxtext,newtxmath}
\usepackage{mathrsfs}
\usepackage{accents}
\usepackage{bigints}
\usepackage{dsfont}

\begin{document}

\title{On the Distribution of the Number of Lattice Points in Norm Balls on the Heisenberg Groups}

\author{YOAV A. GATH}
\address{Department of Mathematics\\ Technion - Israel Institute of Technology\\ Haifa 3200003\\ Israel}
\curraddr{}
\email{yoavgat@campus.technion.ac.il}
\thanks{}



\date{\today}

\dedicatory{}

\begin{abstract}
We investigate the fluctuations in the number of integral lattice points on the Heisenberg groups which lie inside a Cygan-Kor{\'a}nyi norm ball of large radius. Let $\mathcal{E}_{q}(x)=\big|\mathbb{Z}^{2q+1}\cap\delta_{x}\mathcal{B}\big|-\textit{vol}\big(\mathcal{B}\big)x^{2q+2}$ denote the error term which occurs for this lattice point counting problem on the Heisenberg group $\mathbb{H}_{q}$, where $\mathcal{B}$ is the unit ball in the Cygan-Kor{\'a}nyi norm and $\delta_{x}$ is the Heisenberg-dilation by $x>0$. The characteristic behavior of the error term $\mathcal{E}_{q}(x)$ may only be one of two types, depending on whether $q=1$ or $q>1$. It is the higher dimensional case that is the most challenging, and we shall confine ourselves to the case $q\geq3$. To that end, for $q\geq3$ we consider the suitably normalized error term $\mathcal{E}_{q}(x)/x^{2q-1}$, and prove it has a limiting value distribution which is absolutely continuous with respect to the Lebesgue measure. We show that the defining density for this distribution, denoted by $\mathcal{P}_{q}(\alpha)$, can be extended to the whole complex plane $\mathbb{C}$ as an entire function of $\alpha$ and satisfies for any non-negative integer $j\geq0$ and any $\alpha\in\mathbb{R}$, $|\alpha|>\alpha_{q,j}$, the bound:
\begin{equation*}
\begin{split}
\big|\mathcal{P}^{(j)}_{q}(\alpha)\big|\leq\exp{\Big(-|\alpha|^{4-\beta/\log\log{|\alpha|}}\Big)}
\end{split}
\end{equation*}
where $\beta>0$ is an absolute constant. In addition, we prove that $\int_{-\infty}^{\infty}\alpha\mathcal{P}_{q}(\alpha)\textit{d}\alpha=0$ and $\int_{-\infty}^{\infty}\alpha^{3}\mathcal{P}_{q}(\alpha)\textit{d}\alpha<0$, and we give an explicit formula for the $j$-th integral moment of the density $\mathcal{P}_{q}(\alpha)$ for any integer $j\geq1$. Finally, we show that for $j=1,2$, the $j$-th integral moment of $\mathcal{E}_{q}(x)/x^{2q-1}$ is equal to $\int_{-\infty}^{\infty}\alpha^{j}\mathcal{P}_{q}(\alpha)\textit{d}\alpha$, and more generally:
\begin{equation*}
\lim\limits_{X\to\infty}\frac{1}{X}\int\limits_{ X}^{2X}\big|\mathcal{E}_{q}(x)/x^{2q-1}\big|^{\lambda}\textit{d}x=\int\limits_{-\infty}^{\infty}|\alpha|^{\lambda}\mathcal{P}_{q}(\alpha)\textit{d}\alpha
\end{equation*}
for any $0<\lambda\leq2$.
\end{abstract}
\maketitle
\tableofcontents
\section{Introduction, notation and statement of results}
\subsection{Introduction}
We investigate the fluctuations in the number of integral lattice points on the Heisenberg groups which lie inside a Cygan-Kor{\'a}nyi norm ball of expanding radius. This lattice point counting problem arises naturally in the context of the Heisenberg groups, and may be viewed as a non-commutative analogue of the classical lattice point counting problem for Euclidean balls. The latter problem, namely the Euclidean case, has been extensively studied and there is vast body of work dedicated to this subject. Of particular relevance to us is the planer case, the so called Gauss circle problem, which is the problem of determining how many integral lattice points there are in a circle of large radius $x>0$ centered at the origin. Denoting the unit ball with respect to the Euclidean norm by $\mathscr{O}=\big\{u\in\mathbb{R}^{2}:|u|_{2}\leq1\big\}$, and the error term by $E(x)=\big|\mathbb{Z}^{2}\cap x\mathscr{O}\big|-\pi x^{2}$, one seeks to determine the exponent $\omega>0$ defined by:
\begin{equation*}
\omega=\sup\big\{\alpha>0:\big|E(x)\big|\ll x^{2-\alpha}\big\}\,.
\end{equation*}
The determination of the exponent $\omega$ is a famous open problem. Gauss gave the first lower bound $\omega\geq1$. This lower bound has been improved many times over, the landmark results being: $\omega\geq\frac{4}{3}$ Vorono\"{i} \cite{voronoi1904fonction} and Sierpi\'{n}ski \cite{sierpinski1906pewnem} ; $\omega\geq\frac{67}{50}$ Van der Corput \cite{van1923neue} ; $\omega\geq\frac{580}{429}$ Kolesnik \cite{kolesnik1985method} ; $\omega\geq\frac{15}{11}$ Iwaniec and Mozzochi \cite{iwaniec1988divisor} ; $\omega\geq\frac{285}{208}$ Huxley \cite{huxley2003exponential} ; $\omega\geq\frac{1131}{824}$ Bourgain and Watt \cite{bourgain2017mean}. It is conjectured that $\omega=\frac{3}{2}$, which is supported by following moment estimates:
\begin{equation*}
\lim\limits_{X\to\infty}\frac{1}{X}\bigintssss\limits_{X}^{2X}\Big(E(x)/x^{1/2}\Big)^{k}\textit{d}x=(-1)^{k}c_{k}\quad;\qquad k=2,3,\ldots,9
\end{equation*}
where $c_{k}>0$ are positive constants. For $k=2$ this was proved by Cram\'{e}r \cite{cramer1922zwei}, $k=3,4$ by Tsang \cite{Tsang}, and for $k=5,\ldots,9$ by Wenguang \cite{Wenguang}. It is natural to ask: what can be said about the nature in which $\big|\mathbb{Z}^{2}\cap x\mathscr{O}\big|$ fluctuates around its expected value $\pi x^{2}$? Wintner \cite{Wintner} proved that the suitably normalized error term $E(x)/x^{1/2}$ has a limiting value distribution in the sense that, there is a measure $\textit{d}\nu$ such that for any interval $\mathcal{I}\subseteq\mathbb{R}$: 
\begin{equation*}
\lim_{X\to\infty}\frac{1}{X}\textit{meas}\big\{x\in[X,2X]:E(x)/x^{1/2}\in\mathcal{I}\big\}=\bigintssss\limits_{ \mathcal{I}}\textit{d}\nu(\alpha)\,.
\end{equation*}
Heath-Brown \cite{heath1992distribution} showed that the limiting value distribution
is absolutely continuous with respect to the Lebesgue measure, that is $\textit{d}\nu(\alpha)=f(\alpha)\textit{d}\alpha$, and that the density $f(\alpha)$ can be extended to the whole complex plane $\mathbb{C}$ as an entire function of $\alpha$; in particular, $f(\alpha)$ is supported on all of the real line. He also estimated the tails showing that $f(\alpha)$ decays roughly like $\exp{\big(-|\alpha|^{4-\epsilon}\big)}$, and in particular the is non-Gaussian. The underlying idea behind Heath-Brown's proof is based on the fact that $E(x)/x^{1/2}$ can be very well approximated in the
mean by a very short initial segment of its Vorono\"{i} series. In fact, the method developed in \cite{heath1992distribution} applies to a rather general type of functions $F(x)$ satisfying the following hypothesis.
\begin{Hyp H}
There exist continuous real-valued functions $a_{1}(x), a_{2}(x),\ldots$ of period $1$, and real numbers $\lambda_{1},\lambda_{2},\ldots$ which are linearly independent over $\mathbb{Q}$, such that:
\begin{equation*}
\lim_{M\to\infty}\limsup_{X\to\infty}\frac{1}{X}\bigintssss\limits_{X}^{2X}\textit{min}\bigg\{1,\Big|F(x)-\sum_{n\leq N}a_{n}(\lambda_{n}x)\Big|\bigg\}\textit{d}x=0
\end{equation*}
where the functions $a_{n}(x)$ satisfy:
\begin{equation*}
\begin{split}
&\textit{(i)}\quad\bigintssss\limits_{0}^{1}a_{n}(x)\textit{d}x=0\quad\textit{(ii)}\quad\sum_{n=1}^{\infty}\bigintssss\limits_{0}^{1}a^{2}_{n}(x)\textit{d}x<\infty\\
&\textit{(iii)}\quad\underset{x\in[0,1]}{\textit{max}}|a_{n}(x)|\ll n^{1-c}\quad\textit{(iv)}\quad\lim_{n\to\infty}n^{c}\bigintssss\limits_{0}^{1}a^{2}_{n}(x)\textit{d}x=\infty\quad;\quad\textit{for some constant }c>1\,.
\end{split}
\end{equation*}
\end{Hyp H}
\noindent
In the case of the Gauss circle problem, it can be shown that $E(x)/x^{1/2}$ satisfies Hypothesis (H) whith:
\begin{equation*}
a_{n}(x)=\frac{|\mu(n)|}{\pi}\sum_{k=1}^{\infty}\frac{r_{2}(nk^{2})}{(nk^{2})^{3/4}}\cos{\Big(2\pi kx-\frac{3\pi}{4}\Big)}\quad;\quad\lambda_{n}=\sqrt{n}
\end{equation*}
and one can take $c=\frac{5}{3}$, for example. Here $\mu(\cdot)$ is the M{\"o}bius function and $r_{2}(\cdot)$ is the counting function for the number of representations of an integer as a sum of two squares. There are other well known error terms in analytic number theory, as well as error terms arising from counting problem in other areas, which are known to satisfy Hypothesis (H). We refer the reader to \cite{bleher2} for more details and references.\\\\  
We now proceed to introduce the lattice point counting problem on the Heisenberg groups. We mentioned at the beginning of the introduction that this problem may be viewed as a non-commutative analogue of the lattice point counting problem for Euclidean balls. More specifically, we shall replace the usual vector addition in $\mathbb{R}^{2q+1}$ by a non-commutative group law and replace the Euclidean dilation by Heisenberg dilation which will be compatible with our new group structure. The analogue of the Euclidean norm will be the Cygan-Kor{\'a}nyi norm, which is the \textit{canonical} norm on our new group. The lattice point counting problem we shall then consider is that of estimating the number of integral lattice points that are contained in a Heisenberg dilated Cygan-Kor{\'a}nyi norm ball of large radius. Our goal will be to prove that the error term, properly normalized, has a limiting value distribution. As we shall see, the various analytic results we shall prove for this limiting distribution correspond precisely to the known results for the limiting distribution in the case of the Gauss circle problem, and in fact, from the perspective of these results the two distributions are indistinguishable. Nevertheless, there are some subtle yet profound differences between the two, which is due to the fact the dilation process we are using is no longer the Euclidean one. Also, Hypothesis (H) fails to hold in our case for the exact same reason.    
\subsection{The lattice point counting problem on the Heisenberg group $\mathbb{H}_{q}$}
Let $q\geq1$ be an integer. We endow the space $\mathbb{R}^{2q+1}\equiv\mathbb{R}^{2q}\times\mathbb{R}$ with the following homogeneous structure:
\begin{equation*}
\begin{split}
&\big(v,w\big)\ast\big(v',w'\big)=\big(v+v',w+w'+2\langle\text{J}v,v'\rangle\big)\quad;\quad\text{J}=\begin{pmatrix}0&I_{q}\\-I_{q}& 0\end{pmatrix}\\
&\delta_{x}\big(v,w\big)=\big(xv,x^{2}w\big)\qquad\qquad\quad\qquad\qquad\quad\quad;\quad x\in\mathbb{R}_{+}
\end{split}
\end{equation*}
where $\langle\,,\, \rangle$ stands for standard inner product on $\mathbb{R}^{2q}$. The identity element is $(0,0)$ and the inverse is given by $(v,w)^{-1}=\big(-v,-w)$. We shall write $\mathbb{H}_{q}=\big(\mathbb{R}^{2q+1},\ast\big)$ and refer to this group as the $q$-th Heisenberg group. It is a 2-step nilpotent group with a 1-dimensional center, and it is easily checked that the Heisenberg dilations $\delta_{x}$ satisfy:
\begin{equation*}
\begin{split}
&\delta_{x}\big((v,w)\ast(v',w')\big)=\delta_{x}\big(v,w\big)\ast\delta_{x}\big(v',w'\big)\\
&\delta_{x'}\big(\delta_{x}(v,w)\big)=\delta_{x'x}\big(v,w\big)\,.
\end{split}
\end{equation*}
Hence $\big\{  \delta_{x}:x\in\mathbb{R}_{+}\big\}$ forms a group of automorphisms of $\mathbb{H}_{q}$, called the dilation group. The group $\mathbb{H}_{q}$ carries a canonical norm known as the Cygan-Kor{\'a}nyi norm:
\begin{equation*}
\mathcal{N}(v,w)=\Big(|v|_{2}^{4}+w^{2}\Big)^{1/4}\quad;\quad\quad|\cdot|_{2}=\text{Euclidean norm on } \mathbb{R}^{2q}\,.
\end{equation*}
It is homogeneous with respect to the action of the dilation group $\mathcal{N}\circ\delta_{x}=x\mathcal{N}$, and is sub-additive in the sense that:
\begin{equation*}
\mathcal{N}\big((v,w)\ast(v',w')\big)\leq \mathcal{N}(v,w)+\mathcal{N}(v',w')\,.
\end{equation*}
Consequently, the Cygan-Kor{\'a}nyi norm defines a left invariant homogeneous distance on $\mathbb{H}_{q}$. This norm has been extensively studied by Cygan \cite{cygan1978wiener}, \cite{cygan1981subadditivity} and Kor{\'a}nyi \cite{koranyi1985geometric}. It appears in the expression defining the fundamental
solution of a natural sublaplacian on $\mathbb{H}_{q}$ and in other natural kernels, see \cite{stein2016harmonic} and \cite{cowling2010unitary}. It is the analogue of the Euclidean norm on the abelian homogeneous group $\big(\mathbb{R}^{2q+1},+\big)$, and it is for this canonical norm that we shall consider the lattice point counting problem on the $q$-th Heisenberg group. Let:  
\begin{equation*}
\mathcal{B}=\big\{(v,w)\in\mathbb{R}^{2q}\times\mathbb{R}:\mathcal{N}(v,w)\leq1\big\}
\end{equation*}
be the Cygan-Kor{\'a}nyi norm ball in $\mathbb{H}_{q}$ of unit radius. We consider the problem of estimating the number of points in the integral lattice $\mathbb{Z}^{2q+1}$ which lie inside the dilated ball $\delta_{x}\mathcal{B}$ of radius $x>0$:
\begin{equation*}
\big|\mathbb{Z}^{2q+1}\cap\delta_{x}\mathcal{B}\big|=\big|\big\{(z,z')\in\mathbb{Z}^{2q}\times\mathbb{Z}:\mathcal{N}(z,z')\leq x\big\}\big|\,.
\end{equation*}
We expect that the number of such points should be well approximated by $\textit{vol}\big(\mathcal{B}\big)x^{2q+2}$, where $\textit{vol}(\cdot)$ is the $(2q+1)$-dimensional volume. We shall therefore consider the error term that arises from this approximation.
\begin{mydef}
Let $q\geq1$ be an integer. For $x>0$ define:
\begin{align*}
\mathcal{E}_{q}(x)=\big|\mathbb{Z}^{2q+1}\cap\delta_{x}\mathcal{B}\big|-\textit{vol}\big(\mathcal{B}\big)x^{2q+2}
\end{align*}
and set $\kappa_{q}=\sup\big\{\alpha>0:\big|\mathcal{E}_{q}(x)\big|\ll x^{2q+2-\alpha}\big\}$.
\end{mydef}
\noindent
Let us remark that the Gaussian curvature of the enclosing surface $\partial\mathcal{B}$ vanishes at both points of intersection of $\mathcal{B}$ with the $w$-axis, namely the north and south poles. In fact, all of the $2q$ principal curvatures vanish at these two points. We now proceed to describe the current state of knowledge regarding the error term $\mathcal{E}_{q}(x)$.
\subsection{Upper bound, mean-square and $\Omega$-results for $\mathcal{E}_{q}(x)$: A short survey}
Garg, Nevo and Taylor \cite{garg2015lattice} have obtained the lower bound $\kappa_{q}\geq2$ for any value of $q\geq1$. In our previous work \cite{gath2017best}, this lower bound was proved to be sharp in the case of $q=1$, namely $\kappa_{1}=2$, and so the problem is resolved for $\mathbb{H}_{1}$. In the higher dimensional case the behavior of the error term is of an entirely different nature, and is closely related both in shape and form to the error term in the Gauss circle problem as soon as $q\geq3$, while $q=2$ marks the transition point. The case $q=2$ will not be dealt with in the current paper as it requires a different approach, and the corresponding results will appear in a separate paper.\\\\
From now on we assume that $q\geq3$. In \cite{gath2019analogue} we succeeded in establishing the upper bound $|\mathcal{E}_{q}(x)|\ll x^{2q-2/3}$, which translates to the lower bound $\kappa_{q}\geq\frac{8}{3}=2.666...$. This improves upon the lower bound obtained by Garg, Nevo and Taylor. We also have the following $\Omega$-result \cite{gath2019analogue}:
\begin{equation}\label{eq:1.2}
\mathcal{E}_{q}(x)=\Omega\Big(x^{2q-1}\big(\log{x}\big)^{1/4}W(x)\Big)\qquad;\qquad W(x)=\big(\log{\log{x}}\big)^{1/8}\,.
\end{equation}
Consequently, $\kappa_{q}\leq3$. In regards to what should be the true order of magnitude of $\mathcal{E}_{q}(x)$ we conjecture that $|\mathcal{E}_{q}(x)|\ll_{\epsilon}x^{2q-1+\epsilon}$, or in other words  $\kappa_{q}=3$.
This conjecture is supported by the second moment estimate established in \cite{gath2019analogue}. For $q\equiv0\,(2)$ we have:
\begin{equation}\label{eq:1.3}
\begin{split}
\frac{1}{X}\bigintssss\limits_{ X}^{2X}\Big|\mathcal{E}_{q}(x)/x^{2q-1}\Big|^{2}\textit{d}x=\frac{1}{2}\Bigg(\frac{\pi^{q-1}}{2\Gamma(q)}\Bigg)^{2}
\Bigg\{\underset{\,\,(d,2m)=1}{\sum_{d,m=1}^{\infty}}\frac{r^{2}_{2}\big(m,d;q\big)}{m^{3/2}d^{2q-3}}+2^{2q}\underset{d\equiv0(4)}{\underset{\,\,(d,m)=1}{\sum_{d,m=1}^{\infty}}}\frac{r^{2}_{2}\big(m,d;q\big)}{m^{3/2}d^{2q-3}}\Bigg\}+O\Big( X^{-1}\log^{2}{X}\Big)
\end{split}
\end{equation}
and for $q\equiv1\,(2)$:
\begin{equation}\label{eq:1.4}
\begin{split}
\frac{1}{X}\bigintssss\limits_{X}^{2X}\Big|\mathcal{E}_{q}(x)/x^{2q-1}\Big|^{2}\textit{d}x=\frac{1}{2}\Bigg(\frac{\pi^{q-1}}{2\Gamma(q)}\Bigg)^{2}\Bigg\{\underset{\,\,(d,2m)=1}{\sum_{d,m=1}^{\infty}}\frac{r^{2}_{2}\big(m,d;q\big)}{m^{3/2}d^{2q-3}}+2^{2q}\underset{d\equiv0(4)}{\underset{\,\,(d,m)=1}{\sum_{d,m=1}^{\infty}}}\frac{r^{2}_{2,\chi}\big(m,d;q\big)}{m^{3/2}d^{2q-3}}\Bigg\}+O\Big( X^{-1}\log^{2}{X}\Big)\,.
\end{split}
\end{equation}
The arithmetical functions $r_{2}\big(m,d;q\big)$ and $r_{2,\chi}\big(m,d;q\big)$ will feature prominently throughout the paper, and so it is appropriate at this point to introduce them.
\begin{mydef}
Let $q\geq3$ be an integer. For integers $m,d\geq1$ define:
\begin{equation*}
\begin{split}
&r_{2}\big(m,d;q\big)=\underset{b\equiv0(d)}{\sum_{a^{2}+b^{2}=m}}\bigg(\frac{|a|}{\sqrt{m}}\bigg)^{q-1}\quad;\quad r_{2,\chi}\big(m,d;q\big)=\underset{b\equiv0(d)}{\sum_{a^{2}+b^{2}=m}}\chi\big(|a|\big)\bigg(\frac{|a|}{\sqrt{m}}\bigg)^{q-1}
\end{split}
\end{equation*}
where the representation runs over $a,b\in\mathbb{Z}$, and $\chi$ denotes the non-trivial Dirichlet character (mod 4).
\end{mydef}
\noindent
We now proceed to present the main results of this paper. Our goal is to investigate the nature in which $\big|\mathbb{Z}^{2q+1}\cap\delta_{x}\mathcal{B}\big|$ fluctuates around its expected value $\textit{vol}\big(\mathcal{B}\big)x^{2q+2}$. To that end, let $\textit{d}\nu_{X,q}(\alpha)$ be the distribution of the random variable $x\mapsto\mathcal{E}_{q}(x)/x^{2q-1}$, $x$ uniformly distributed on a segment $[X,2X]$. We shall establish the weak convergence of the distributions  $\textit{d}\nu_{X,q}(\alpha)$, as $X\to\infty$, to an absolutely continuous distribution $\textit{d}\nu_{q}(\alpha)=\mathcal{P}_{q}(\alpha)\textit{d}\alpha$, and obtain decay estimates for its defining density $\mathcal{P}_{q}(\alpha)$.
\subsection{Statement of the main results}
\begin{thm}
Let $q\geq3$ be an integer. Then the normalized error term $\mathcal{E}_{q}(x)/x^{2q-1}$ has a limiting value distribution in the sense that, there exists a probability density $\mathcal{P}_{q}(\alpha)$ such that for any bounded (piecewise)-continuous function $\mathcal{F}$:
\begin{equation}\label{eq:2.1}
\lim\limits_{X\to\infty}\frac{1}{X}\bigintssss\limits_{ X}^{2X}\mathcal{F}\Big(\mathcal{E}_{q}(x)/x^{2q-1}\Big)\textit{d}x=\bigintssss\limits_{-\infty}^{\infty}\mathcal{F}(\alpha)\mathcal{P}_{q}(\alpha)\textit{d}\alpha\,.
\end{equation}
The density $\mathcal{P}_{q}(\alpha)$ can be extended to the whole complex plane $\mathbb{C}$ as an entire function of $\alpha$, and in particular is supported on all of the real line. It satisfies for any non-negative integer $j\geq0$ and any $\alpha\in\mathbb{R}$, $|\alpha|>\alpha_{q,j}$, the bound:
\begin{equation}\label{eq:2.2}
\begin{split}
\big|\mathcal{P}^{(j)}_{q}(\alpha)\big|\leq\exp{\bigg(-|\alpha|^{4-\beta/\log\log{|\alpha|}}\bigg)}
\end{split}
\end{equation}
where $\beta>0$ is an absolute constant.
\end{thm}
\noindent
Theorem 1 therefore establishes the weak convergence of the distributions  $\textit{d}\nu_{X,q}(\alpha)$ to an absolutely continuous distribution $\textit{d}\nu_{q}(\alpha)=\mathcal{P}_{q}(\alpha)\textit{d}\alpha$ as $X\to\infty$. Our next result gives an explicit formula for all the integral moments of the density $\mathcal{P}_{q}(\alpha)$.
\begin{thm}
Let $q\geq3$ be an integer. The $j$-th integral moment of $\mathcal{P}_{q}(\alpha)$ is given by:
\begin{equation}\label{eq:2.3}
\begin{split}
&\bigintssss\limits_{-\infty}^{\infty}\alpha^{j}\mathcal{P}_{q}(\alpha)\textit{d}\alpha=\sum_{s=1}^{j}\underset{\,\,\ell_{1},\,\ldots\,,\ell_{s}\geq1}{\sum_{\ell_{1}+\cdots+\ell_{s}=j}}\,\,\frac{j!}{\ell_{1}!\cdots\ell_{s}!}\underset{\,\,\,m_{s}>\cdots>m_{1}}{\sum_{m_{1},\,\ldots\,,m_{s}=1}^{\infty}}\prod_{i=1}^{s}\mathcal{Q}_{q}(m_{i},\ell_{i})
\end{split}
\end{equation}
where the series on the RHS of \eqref{eq:2.3} converges absolutely. For integers $m,\ell\geq1$, the term $\mathcal{Q}_{q}(m,\ell)$ is given by:
\begin{equation}\label{eq:2.4}
\mathcal{Q}_{q}(m,\ell)=(-1)^{\ell}\Bigg(\frac{\pi^{q-1}}{4\Gamma(q)}\Bigg)^{\ell}\frac{\mu^{2}(m)}{m^{3\ell/4}}\sum_{a\,(8)}\cos{\Big(\frac{a\pi}{4}\Big)}\underset{\overset{\ell}{\underset{i=1}{\sum}}\textit{e}_{i}\equiv a\,(8)}{\sum_{\textit{e}_{1},\ldots,\textit{e}_{\ell}=\pm1}}\,\,\,\underset{\overset{\ell}{\underset{i=1}{\sum}}\textit{e}_{i}\epsilon_{q}(d_{i})\frac{k_{i}}{d_{i}}=0}{\sum_{d_{1},k_{1},\ldots,d_{\ell},k_{\ell}=1}^{\infty}}\,\,\prod_{i=1}^{\ell}\frac{\mathfrak{r}\big(mk^{2}_{i},d_{i}\,;q\big)}{d_{i}^{q-3/2}k_{i}^{3/2}}
\end{equation}
where for $q\equiv0\,(2)$ we define $\epsilon_{q}(d)=1$, and:
\begin{equation*}
\mathfrak{r}\big(mk^{2},d\,;q\big)=\left\{
        \begin{array}{ll}
            0& ;\,(k,d)\neq1\text{ or } d\equiv2\,(4)\\\\
            r_{2}\big(mk^{2},d;q\big)& ;\, (d,2)=1\\\\
            (-1)^{\frac{q}{2}}2^{q}r_{2}\big(mk^{2},d;q\big)& ;\, d\equiv0\,(4)
        \end{array}
    \right.
\end{equation*}
while for $q\equiv1\,(2)$ we define $\epsilon_{q}(d)=1$ if $(d,2)=1$ and $\epsilon_{q}(d)=-1$ otherwise, and:
\begin{equation*}
\mathfrak{r}\big(mk^{2},d\,;q\big)=\left\{
        \begin{array}{ll}
            0& ;\,(k,d)\neq1\text{ or } d\equiv2\,(4)\\\\
            \chi(d)r_{2}\big(mk^{2},d;q\big)& ;\, (d,2)=1\\\\
            (-1)^{\frac{q-1}{2}}2^{q}r_{2,\chi}\big(mk^{2},d;q\big)& ;\, d\equiv0\,(4)\,.
        \end{array}
    \right.
\end{equation*}
Note that for $\ell=1$ the sum in \eqref{eq:2.4} is void, so by definition $\mathcal{Q}_{q}(m,1)=0$. In addition, the first moment of $\mathcal{P}_{q}(\alpha)$ vanishes while its third moment is strictly negative:
\begin{equation}\label{eq:2.5}
\begin{split}
\bigintssss\limits_{-\infty}^{\infty}\alpha\mathcal{P}_{q}(\alpha)\textit{d}\alpha=0\quad;\quad\bigintssss\limits_{-\infty}^{\infty}\alpha^{3}\mathcal{P}_{q}(\alpha)\textit{d}\alpha<0\,.
\end{split}
\end{equation}
\end{thm}
\noindent
From Theorem 2 it follows that $\mathcal{P}_{q}(\alpha)$ is asymmetric, and is skewed towards negative values of $\alpha$. It can also be shown that the maximum of $\mathcal{P}_{q}(\alpha)$ occurs at positive $\alpha$. Our third and final result is:
\begin{thm}
Let $q\geq3$ be an integer. Then for $0<\lambda\leq2$ we have:
\begin{equation}\label{eq:2.6}
\lim\limits_{X\to\infty}\frac{1}{X}\bigintssss\limits_{ X}^{2X}\big|\mathcal{E}_{q}(x)/x^{2q-1}\big|^{\lambda}\textit{d}x=\bigintssss\limits_{-\infty}^{\infty}|\alpha|^{\lambda}\mathcal{P}_{q}(\alpha)\textit{d}\alpha
\end{equation}
and, in the case where $\lambda=1$:
\begin{equation}\label{eq:2.7}
\lim\limits_{X\to\infty}\frac{1}{X}\bigintssss\limits_{ X}^{2X}\mathcal{E}_{q}(x)/x^{2q-1}\textit{d}x=\bigintssss\limits_{-\infty}^{\infty}\alpha\mathcal{P}_{q}(\alpha)\textit{d}\alpha\,.
\end{equation}
\end{thm}
\noindent
\textbf{\textit{On the method of proof}}.
Our main idea of the proof is to show, using the results obtained in \cite{gath2019analogue}, that $\mathcal{E}_{q}(x)/x^{2q-1}$ can be approximated in the mean by certain oscillating series with incommensurable frequencies. Our problem is then to establish the existence of a limiting distribution for almost periodic functions of a special from. These functions are given by a conditionally convergent square-summable series $\sum_{m=1}^{\infty}\upphi_{q,m}\big(\gamma_{m}t\big)$ with linearly independent frequencies $\gamma_{m}$, where $\upphi_{q,m}(\cdot)$ are real valued continuous functions satisfying certain conditions. Heath-Brown \cite{heath1992distribution}, see also  \cite{bleher1993distribution} and \cite{bleher1992distribution}, has proved some general theorems concerning the limiting distribution of certain series similar to ours. The results however can not be applied to our case (at least not directly), since to begin with the functions $\upphi_{q,m}(\cdot)$ are not periodic. Nevertheless, the principle ideas developed in \cite{heath1992distribution} can be made to work in our case. As will be apparent, the lack of periodicity will not be the only issue we shall be faced with in the analysis stage of $\upphi_{q,m}(\cdot)$ and there will be additional difficulties, more intrinsic in nature, that we shall need to overcome.\\\\ 
\textbf{\textit{Notation and conventions}}. Throughout this paper, $q\geq3$ is an arbitrary fixed integer. The following notation will occur repeatedly in this paper:
\begin{equation*}
\begin{split}
1.\,\,\,\,\chi=\text{\,the non-trivial Dirichlet character (mod 4)}
\end{split}
\end{equation*}
\begin{equation*}
\begin{split}
2.\,\,\,\,\mu=\text{\,the M{\"o}bius function}
\end{split}
\end{equation*}
\begin{equation*}
\begin{split}
3.\,\,\,\,\gamma_{m}=\sqrt{m}\quad;\quad m\in\mathbb{N}
\end{split}
\end{equation*}
\begin{equation*}
\begin{split}
4.\,\,\,\,\xi\big(d;q\big)=\mathds{1}_{d\equiv1(2)}+(-1)^{\frac{q}{2}+1}\mathds{1}_{d\equiv0(2)}+(-1)^{\frac{q}{2}}2^{q}\mathds{1}_{d\equiv0(4)}\quad;\quad q\equiv0\,(2)
\end{split}
\end{equation*}
\begin{equation*}
\begin{split}
5.\,\,\,\,\varrho_{q}=\frac{\pi^{q}}{(1-2^{-q})\Gamma(q)\zeta(q)}\quad;\quad\zeta(s)=\sum_{n=1}^{\infty}\frac{1}{n^{s}}\,\,,\,\,\Re(s)>1
\end{split}
\end{equation*}
\begin{equation*}
\begin{split}
6.\,\,\,\,\varrho_{\chi,q}=\frac{\pi^{q}}{2^{q-1}\Gamma(q)L(q,\chi)}\quad;\quad L(s,\chi)=\sum_{n=1}^{\infty}\frac{\chi(n)}{n^{s}}\,\,,\,\,\Re(s)>1
\end{split}
\end{equation*}
\begin{equation*}
\begin{split}
7.\,\,\,\,r_{2}(m)=\big|\big\{\big(a,b\big)\in\mathbb{Z}^{2}:a^{2}+b^{2}=m\big\}\big|
\end{split}
\end{equation*}
\section{Almost periodicity of $\mathcal{E}_{q}(x)/x^{2q-1}$}
\subsection{Statement of Theorem 4}
In this section we show that the normalized error term $\mathcal{E}_{q}(x)/x^{2q-1}$ can be approximated, in a suitable sense, by means of certain oscillating series. This notion is made precise in the following theorem:
\begin{thm}
We have:
\begin{equation}\label{eq:4.1}
\lim_{M\to\infty}\limsup_{X\to\infty}\frac{1}{X}\bigintssss\limits_{X}^{2X}\Big|\mathcal{E}_{q}(x)/x^{2q-1}-\sum_{m\leq M}\upphi_{q,m}\big(\gamma_{m}x^{2}\big)\Big|^{2}\textit{d}x=0
\end{equation}
where the functions $\upphi_{q,m}(\cdot)$ are defined in accordance with the parity of $q$. For $q\equiv0\,(2)$: 
\begin{equation*}
\upphi_{q,m}(t)=\frac{\varrho_{q}}{2\pi}\frac{\mu^{2}(m)}{m^{3/4}}\sum_{d,k=1}^{\infty}\frac{\xi\big(d;q\big)r_{2}\big(mk^{2},d;q\big)}{d^{q-3/2}k^{3/2}}\sin{\Big(2\pi\frac{k}{d}t-\frac{\pi}{4}\Big)}
\end{equation*}
and for $q\equiv1\,(2)$: 
\begin{equation*}
\begin{split}
\upphi_{q,m}(t)=2^{q-2}\frac{\varrho_{\chi,q}}{\pi}\frac{\mu^{2}(m)}{m^{3/4}}\Bigg\{&\sum_{d,k=1}^{\infty}\frac{\chi(d)r_{2}\big(mk^{2},d;q\big)}{d^{q-3/2}k^{3/2}}\sin{\Big(2\pi\frac{k}{d}t-\frac{\pi}{4}\Big)}+\\
&+(-1)^{\frac{q+1}{2}}2^{q}\underset{\,\,d\equiv0\,(4)}{\sum_{d,k=1}^{\infty}}\frac{r_{2,\chi}\big(mk^{2},d;q\big)}{d^{q-3/2}k^{3/2}}\cos{\Big(2\pi \frac{k}{d}t-\frac{\pi}{4}\Big)}\Bigg\}\,.
\end{split}
\end{equation*}
\end{thm}
\noindent
The proof of Theorem 4 will be given at the end of this section. We now proceed to analyse the approximate expression for $\mathcal{E}_{q}(x)$ obtained in \cite{gath2019analogue}.
\subsection{The approximate expression for $\mathcal{E}_{q}(x)$}
In \cite{gath2019analogue} we obtained a Vorono\"{i}-type series expansion for the error term $\mathcal{E}_{q}(x)$. The derivation of this series expansion is quite involved and highly technical, and occupies a substantial prat of \cite{gath2019analogue}. Before we present it, we give a rough sketch of the main ideas. The first step is to execute the lattice point count in the following form:
\begin{equation*}
\begin{split}
\big|\mathbb{Z}^{2q+1}\cap\delta_{x}\mathcal{B}\big|=&2\sum_{0\,\leq\,m\,\leq\, \frac{x^{2}}{\sqrt{2}}}r_{2q}(m)\Big(\sqrt{x^{4}-m^{2}}-m\Big)+\sum_{0\,<\,m\,\leq\, \frac{x^{2}}{\sqrt{2}}}r_{2q}(m)+\sum_{0\,\leq\,|n|\,\leq\, \frac{x^{2}}{\sqrt{2}}}\,\,\sum_{|n|\,<\,m\,\leq\,\sqrt{x^{4}-n^{2}}}r_{2q}(m)-\\
&-2\sum_{0\,\leq\,m\,\leq\, \frac{x^{2}}{\sqrt{2}}}r_{2q}(m)2\psi\Big(\sqrt{x^{4}-m^{2}}\,\Big)
\end{split}
\end{equation*}
where $\psi(t)=t-[t]-1/2$, and $r_{2q}(m)$ is the counting function for the number of representation of the integer $m$ as a sum of $2q$-squarse. One then proceeds to extract the main term. The first sum is estimated using contour integration and gives a contribution to the main term plus a negligible error term. The third sum is estimated directly using the Euler–Maclaurin summation formula, which together with the second sum, gives a contribution to the main term plus a non-negligible error term in the form of a certain oscillating sum involving the function $\psi$. Having extracted the main term, the next step is to subject the oscillating sums involving the function $\psi$ to a transformation process whose end result is the Vorono\"{i}-type series expansion mentioned above. This process begins with an application of Vaaler's Lemma, which enables us to approximate the above $\psi$-sums by a certain type of exponential sums. In turn, these exponential sums are estimated using a sharp form of the \textit{B}-process of Van der Corput. The end result of this long process is: 
\begin{equation*}
\begin{split}
\mathcal{E}_{q}(x)/x^{2q-1}\approx&\sum_{d,m}\nu(d)\frac{r_{2}\big(m,d;q\big)}{d^{q-3/2}m^{3/4}}\sin{\Big(2\pi \frac{\sqrt{m}}{d}x^{2}-\frac{\pi}{4}\Big)}+\\
&+\mathds{1}_{q\equiv1\,(2)}\sum_{d,m}\eta(d)\frac{r_{2,\chi}\big(m,d;q\big)}{d^{q-3/2}m^{3/4}}\cos{\Big(2\pi \frac{\sqrt{m}}{d}x^{2}-\frac{\pi}{4}\Big)}+\textit{lower order terms}
\end{split}
\end{equation*}
where the range of summation depends on the variable $x$, and $\nu(\cdot)$, $\eta(\cdot)$ are certain bounded coefficients which depend on $q$. The fact of the matter is that the series expansion for $\mathcal{E}_{q}(x)$ is a bit more involved then the one sketched above. In order to state it, we first need to introduce some notation and definitions.
\begin{mydef}
Let $H\geq1$. For integers $m,d\geq1$ we define the following arithmetical functions:
\begin{equation*}
\begin{split}
\mathfrak{a}_{H}\big(m,d;q\big)=\frac{1}{m^{3/4}}\Bigg\{\underset{n\equiv0\,(d)}{\underset{0\leq n\leq h}{\underset{1\leq h\leq H}{\sideset{}{''}\sum_{\,\,\,\,\,\,n^{2}+h^{2}=m}}}}\tau\bigg(\frac{h}{[H]+1}\bigg)\bigg(\frac{h}{\sqrt{m}}\bigg)^{q-1}+\underset{h\equiv0\,(d)}{\underset{0\leq n\leq h}{\underset{1\leq h\leq H}{\sideset{}{''}\sum_{\,\,\,\,\,\,n^{2}+h^{2}=m}}}}\tau\bigg(\frac{h}{d[H/d]+d}\bigg)\bigg(\frac{n}{\sqrt{m}}\bigg)^{q-1}\Bigg\}
\end{split}
\end{equation*}
\begin{equation*}
\begin{split}
\mathfrak{a}^{\ast}_{H}\big(m,d;q\big)=\frac{1}{m^{3/4}}\Bigg\{\underset{n\equiv0\,(d)}{\underset{0\leq n\leq h}{\underset{1\leq h\leq H}{\sideset{}{''}\sum_{\,\,\,\,\,\,n^{2}+h^{2}=m}}}}\tau^{\ast}\bigg(\frac{h}{[H]+1}\bigg)\bigg(\frac{h}{\sqrt{m}}\bigg)^{q-1}+\underset{h\equiv0\,(d)}{\underset{0\leq n\leq h}{\underset{1\leq h\leq H}{\sideset{}{''}\sum_{\,\,\,\,\,\,n^{2}+h^{2}=m}}}}\tau^{\ast}\bigg(\frac{h}{d[H/d]+d}\bigg)\bigg(\frac{n}{\sqrt{m}}\bigg)^{q-1}\Bigg\}
\end{split}
\end{equation*}
\begin{equation*}
\begin{split}
\mathfrak{a}_{H,\chi}\big(m,d;q\big)=\frac{2}{m^{3/4}}\Bigg\{\underset{n\equiv0\,(d)}{\underset{0\leq n\leq h}{\underset{1\leq h\leq H}{\sideset{}{''}\sum_{\,\,\,\,\,\,n^{2}+h^{2}=m}}}}\chi(-h)\tau\bigg(\frac{h}{[H]+1}\bigg)\bigg(\frac{h}{\sqrt{m}}\bigg)^{q-1}+\underset{h\,\equiv\,0\,(d)}{\underset{0\,\leq\,n\,\leq\,h}{\underset{1\,\leq\, h\,\leq\, H}{\sideset{}{''}\sum_{n^{2}+h^{2}=m}}}}\chi(-n)\tau\bigg(\frac{h}{d[H/d]+d}\bigg)\bigg(\frac{n}{\sqrt{m}}\bigg)^{q-1}\Bigg\}
\end{split}
\end{equation*}
\begin{equation*}
\begin{split}
&\mathfrak{b}^{\ast}_{H}\big(m,d;q\big)=\frac{2}{m^{3/4}}\Bigg\{\underset{n\equiv0\,(d)}{\underset{0\leq n\leq h}{\underset{1\leq h\leq H}{\sideset{}{''}\sum_{\,\,\,\,\,\,n^{2}+h^{2}=m}}}}\lambda(h)\tau^{\ast}\bigg(\frac{h}{[H]+1}\bigg)\bigg(\frac{h}{\sqrt{m}}\bigg)^{q-1}+2\underset{n\equiv0\,(4)\,,\,h\equiv0\,(d)}{\underset{0\leq n\leq h}{\underset{1\leq h\leq H}{\sideset{}{''}\sum_{\,\,\,\,\,\,n^{2}+h^{2}=m}}}}\tau^{\ast}\bigg(\frac{h}{d[H/d]+d}\bigg)\bigg(\frac{n}{\sqrt{m}}\bigg)^{q-1}\Bigg\}
\end{split}
\end{equation*}
where the double-dash $''$ indicates that the terms $(n,h)$ corresponding to $n=0,h$ are multiplied by $1/2$, and $\lambda(h)=\mathds{1}_{h\equiv0(4)}-\mathds{1}_{h\equiv2(4)}$. The functions $\tau(t)$ and $\tau^{\ast}(t)$ are given by:
\begin{equation*}
\begin{split}
&\tau(t)=t(1-t)\cot{(\pi t)}+\pi^{-1}t\quad;\quad\,\,\,\,\,\,0<t<1\\
&\tau^{\ast}(t)=t(1-t)\quad\qquad\quad\qquad\,\,\,\,\,;\qquad0<t<1\,.
\end{split}
\end{equation*}
\end{mydef}
\begin{rem}
Note that for integers $m>2H^{2}$ the above arithmetical functions vanish identically. Let us also note that since $\tau(t)$ is bounded on the closed interval $[0,1]$, in fact $\tau(t)=\pi^{-1}+O\big(t^{2}\big)$ uniformly for $t\in[0,1]$, we have that:
\begin{equation*}
\big|\mathfrak{a}_{H,\chi}\big(m,d;q\big)\big|\leq\mathfrak{a}_{H}\big(m,d;q\big)\leq C\frac{r_{2}\big(m,d;q\big)}{m^{3/4}}\mathds{1}_{m\leq2H^{2}}
\end{equation*}
for any integers $m,d\geq1$ and any $H\geq1$, where $C>0$ is an absolute constant.
\end{rem}
\noindent
With the above arithmetical functions we construct the following trigonometric sums.
\begin{mydef}
Let $H\geq1$. For real $x>0$ we define the following trigonometric sum in accordance with the parity of $q$. For $q\equiv0\,(2)$:  
\begin{equation*}
\begin{split}
&\mathcal{S}_{q,H}(x)=2\varrho_{q}\underset{d\leq\sqrt{H}}{\sum_{\,m\leq 2H^{2}}}\,\frac{\xi\big(d;q\big)\mathfrak{a}_{H}\big(m,d;q\big)}{d^{q-3/2}}\sin{\Big(2\pi \frac{\sqrt{m}}{d}x^{2}-\frac{\pi}{4}\Big)}\\
&\mathcal{S}^{\ast}_{q,H}(x)=2\varrho_{q}\underset{d\leq\sqrt{H}}{\sum_{\,m\leq 2H^{2}}}\,\frac{|\xi(d;q)|\mathfrak{a}^{\ast}_{H}\big(m,d\big)}{d^{q-3/2}}\cos{\Big(2\pi \frac{\sqrt{m}}{d}x^{2}-\frac{\pi}{4}\Big)}\,.
\end{split}
\end{equation*}
For $q\equiv1\,(2)$:
\begin{equation*}
\begin{split}
\mathcal{S}_{q,H}(x)=&2^{q}\varrho_{\chi,q}\underset{d\leq\sqrt{H}}{\sum_{\,m\leq 2H^{2}}}\,\frac{\chi(d)\mathfrak{a}_{H}\big(m,d;q\big)}{d^{q-3/2}}\sin{\Big(2\pi \frac{\sqrt{m}}{d}x^{2}-\frac{\pi}{4}\Big)}+\\
&+(-1)^{\frac{q-1}{2}}2^{2q-1}\varrho_{\chi,q}\underset{\,d\equiv0\,(4)}{\underset{d\leq\sqrt{H}}{\sum_{\,m\leq 2H^{2}}}}\,\frac{\mathfrak{a}_{H,\chi}\big(m,d;q\big)}{d^{q-3/2}}\cos{\Big(2\pi \frac{\sqrt{m}}{d}x^{2}-\frac{\pi}{4}\Big)}\,.
\end{split}
\end{equation*}
\begin{equation*}
\begin{split}
\mathcal{S}^{\ast}_{q,H}(x)=2\varrho_{\chi,q}\underset{d\leq\sqrt{H}}{\sum_{\,m\leq 2H^{2}}}\,\frac{\mathfrak{d}^{\ast}_{H}\big(m,d;q\big)}{d^{q-3/2}}\cos{\Big(2\pi \frac{\sqrt{m}}{d}x^{2}-\frac{\pi}{4}\Big)}
\end{split}
\end{equation*}
\begin{equation*}
\begin{split}
\mathcal{T}^{H}_{q,\chi}(x)=2^{q-1}\varrho_{\chi,q}\underset{d>\sqrt{H}}{\sum_{hd\leq H}}\frac{\chi(d)}{d^{q-3/2}h^{3/2}}\tau\bigg(\frac{h}{[H/d]+1}\bigg)\sin{\Big(2\pi \frac{h}{d}x^{2}-\frac{\pi}{4}\Big)}
\end{split}
\end{equation*}
\begin{equation*}
\begin{split}
\mathcal{T}^{H,\chi}_{q}(x)=(-1)^{\frac{q+1}{2}}2^{2q-1}\varrho_{\chi,q}\underset{d\equiv0\,(4)}{\underset{d>\sqrt{H}}{\sum_{hd\leq H}}}\frac{\chi(h)}{d^{q-3/2}h^{3/2}}\tau\bigg(\frac{h}{[H/d]+1}\bigg)\cos{\Big(2\pi \frac{h}{d}x^{2}-\frac{\pi}{4}\Big)}
\end{split}
\end{equation*}
\begin{equation*}
\begin{split}
\mathcal{T}^{H}_{q}(x)=2\varrho_{\chi,q}\underset{d>\sqrt{H}}{\sum_{hd\leq H}}\frac{\lambda^{\ast}(h,d)}{d^{q-3/2}h^{3/2}}\tau^{\ast}\bigg(\frac{h}{[H/d]+1}\bigg)\cos{\Big(2\pi \frac{h}{d}x^{2}-\frac{\pi}{4}\Big)}
\end{split}
\end{equation*}
where $\mathfrak{d}^{\ast}_{H}\big(m,d;q\big)=2^{q-1}\mathfrak{a}^{\ast}_{H}\big(m,d;q\big)+2^{2q-2}\mathds{1}_{d\equiv0(4)}\mathfrak{b}^{\ast}_{H}\big(m,d;q\big)$, and $\lambda^{\ast}(h,d)=2^{q-2}+\mathds{1}_{d\equiv0(4)}4^{q-1}\lambda(h)$.
\end{mydef}
\noindent
We have the following approximate expression for $\mathcal{E}_{q}(x)$ (see \cite{gath2019analogue}, $\S$\textit{3.4} \textit{Proposition 3.1}).
\begin{prop}
Let $X>0$ be large, $H=H(X)=X^{2}/2$ and $X\leq x\leq2X$. Then:
\begin{equation}\label{eq:4.2}
\begin{split}
\mathcal{E}_{q}(x)/x^{2q-1}=\mathcal{S}_{q,H}(x)+\mathcal{R}^{H}_{q}(x)+\mathds{1}_{q=3}\bigg\{\mathcal{T}^{H}_{q,\chi}(x)+\mathcal{T}^{H,\chi}_{q}(x)\bigg\}+O\Big(x^{-1}\log^{2}{x}\Big)
\end{split}
\end{equation}
where for $q\equiv0\,(2)$:
\begin{equation}\label{eq:4.3}
\begin{split}
\big|\mathcal{R}^{H}_{q}(x)\big|\leq\mathcal{S}^{\ast}_{q,H}(x)+\tilde{c}_{q}x^{-1}\log^{2}{x}\quad\qquad\qquad\qquad;\quad \tilde{c}_{q}>0
\end{split}
\end{equation}
and for $q\equiv1\,(2)$:
\begin{equation}\label{eq:4.4}
\begin{split}
\big|\mathcal{R}^{H}_{q}(x)\big|\leq\mathcal{S}^{\ast}_{q,H}(x)
+\mathds{1}_{q=3}\mathcal{T}^{H}_{q}(x)+\tilde{c}_{q}x^{-1}\log^{2}{x}\quad;\quad \tilde{c}_{q}>0\,.
\end{split}
\end{equation}
\end{prop}
\begin{rem}
The above approximate expression was key in establishing the asymptotic estimate for the second moment of $\mathcal{E}_{q}(x)$ with a sharp bound on the remainder term (see \cite{gath2019analogue}, $\S1.3$ \textit{Theorem 2}). For our present purposes, namely proving Theorem 4, one could work with a weaker form of the approximate expression.          
\end{rem}
\subsection{Approximating $\mathcal{E}_{q}(x)$ in the mean by a short initial segment of $\mathcal{S}_{q,H}(x)$} Our first objective is to dispose of the remainder terms appearing in \eqref{eq:4.2}. We have the following estimate (see \cite{gath2019analogue} $\S$\textit{4.2} \textit{Proposition 4.5}).
\begin{lem}
Let $X>0$ be large, $H=H(X)=X^{2}/2$. If $\mathcal{K}(x)=\mathcal{K}_{q,H}(x)$ denotes any of the following trigonometric sums: $\mathcal{S}^{\ast}_{q,H}(x),\,\mathcal{T}^{H}_{q,\chi}(x),\,\mathcal{T}^{H,\chi}_{q}(x),\,\mathcal{T}^{H}_{q}(x)$, then:
\begin{equation}\label{eq:4.5}
\frac{1}{X}\bigintssss\limits_{ X}^{2X}\mathcal{K}^{2}(x)\textit{d}x\ll\Big(X^{-1}\log{X}\Big)^{2}\,.
\end{equation}
\end{lem}
\noindent
We immediately obtain:
\begin{prop}
For $X>0$ let $H=H(X)=X^{2}/2$. Then:
\begin{equation}\label{eq:4.6}
\begin{split}
\lim_{X\to\infty}\frac{1}{X}\bigintssss\limits_{ X}^{2X}\Big|\mathcal{E}_{q}(x)/x^{2q-1}-\mathcal{S}_{q,H}(x)\Big|^{2}\textit{d}x=0\,.
\end{split}
\end{equation}
\end{prop}
\begin{proof}
From Proposition 1 and Lemma 1 we obtain after applying Cauchy–Schwarz inequality to handle the cross-terms:
\begin{equation}\label{eq:4.7}
\begin{split}
\frac{1}{X}\bigintssss\limits_{ X}^{2X}\Big|\mathcal{E}_{q}(x)/x^{2q-1}-\mathcal{S}_{q,H}(x)\Big|^{2}\textit{d}x\ll X^{-2}\log^{4}X\,.
\end{split}
\end{equation}
Letting $X\to\infty$ concludes the proof.
\end{proof}
\noindent
Having disposed of the remainder terms in the approximate expression \eqref{eq:4.2}, our next objective is to reduce the number of terms appearing in $\mathcal{S}_{q,H}(x)$. We have the following lemma (see \cite{gath2019analogue}, $\S$\textit{4.1} \textit{Lemma 4.1}).
\begin{lem}
Let $X_{0}\geq1$ be a fixed parameter. Suppose that for $X>X_{0}$, there are arithmetical functions $\nu,\eta:\mathbb{N}\rightarrow\mathbb{R}$ and $\alpha,\beta:\mathbb{N}^{2}\rightarrow\mathbb{R}$, depending on the parameter $H=H(X)=X^{2}/2$, which satisfy the following two conditions:  
\begin{equation*}
\begin{split}
&C.1\quad |\nu(d)|,\,|\eta(d)|\leq w_{q}/d^{q-3/2}\,\mathds{1}_{d\leq\sqrt{H}}\\
&C.2\quad|\alpha(n,d)|,\,|\beta(n,d)|\leq C\frac{r_{2}\big(n,d;q\big)}{n^{3/4}}\mathds{1}_{n\leq2H^{2}}
\end{split}
\end{equation*}
where $C,\,w_{q}>0$ are absolute constants. For real $x>0$ define:
\begin{equation*}
\begin{split}
&\mathcal{J}^{\sin{}}_{q,H}\big(x;\nu,\alpha\big)=\sum_{d,n}\nu(d)\alpha\big(n,d\big)\sin{\Big(2\pi\frac{\sqrt{n}}{d}x^{2}-\frac{\pi}{4}\Big)}\\
&\mathcal{J}^{\cos{}}_{q,H}\big(x;\eta,\beta\big)=\sum_{d,n}\eta(d)\beta\big(n,d\big)\cos{\Big(2\pi\frac{\sqrt{n}}{d}x^{2}-\frac{\pi}{4}\Big)}\,.
\end{split}
\end{equation*}
Then for all large $X>X_{0}$ :
\begin{equation}\label{eq:4.8}
\begin{split}
&\frac{1}{X}\bigintssss\limits_{ X}^{2X}\bigg\{\mathcal{J}^{\sin{}}_{q,H}\big(x;\nu,\alpha\big)+\mathcal{J}^{\cos{}}_{q,H}\big(x;\eta,\beta\big)\bigg\}^{2}\textit{d}x=\frac{1}{2}\Big\{\Xi_{H}\big(\nu,\alpha\big)+\Xi_{H}\big(\eta,\beta\big)\Big\}+O\Big(X^{-1}\log{X}\Big)
\end{split}
\end{equation}
where:
\begin{equation*}
\begin{split}
&\Xi_{H}\big(\nu,\alpha\big)=\underset{(d,n)=1}{\sum_{\ell,d,n}}|\mu(\ell)|\bigg(\sum_{r}\nu(rd)\alpha\big((rn)^{2}\ell,rd\big)\bigg)^{2}\quad;\quad\Xi_{H}\big(\eta,\beta\big)=\underset{(d,n)=1}{\sum_{\ell,d,n}}|\mu(\ell)|\bigg(\sum_{r}\eta(rd)\beta\big((rn)^{2}\ell,rd\big)\bigg)^{2}
\end{split}
\end{equation*}
and the implied constant depends only on the absolute constants in conditions C.1 and C.2
\end{lem}
\noindent
With the aid of Lemma 2 we prove the following proposition:
\begin{prop}
For $X>0$ let $H=H(X)=X^{2}/2$. Then:
\begin{equation}\label{eq:4.9}
\lim_{M\to\infty}\limsup_{X\to\infty}\frac{1}{X}\bigintssss\limits_{X}^{2X}\Big|\mathcal{S}_{q,H}(x)-\mathcal{S}^{M}_{q,H}(x)\Big|^{2}\textit{d}x=0
\end{equation}
where for $q\equiv0\,(2)$:
\begin{equation*}
\begin{split}
&\mathcal{S}^{M}_{q,H}(x)=2\varrho_{q}\underset{d\leq\sqrt{H}}{\sum_{m\leq M}}\,\frac{\xi\big(d;q\big)\mathfrak{a}_{H}\big(m,d;q\big)}{d^{q-3/2}}\sin{\Big(2\pi \frac{\sqrt{m}}{d}x^{2}-\frac{\pi}{4}\Big)}
\end{split}
\end{equation*}
and for $q\equiv1\,(2)$:
\begin{equation*}
\begin{split}
\mathcal{S}_{q,H}(x)=&2^{q}\varrho_{\chi,q}\underset{d\leq\sqrt{H}}{\sum_{m\leq M}}\,\frac{\chi(d)\mathfrak{a}_{H}\big(m,d;q\big)}{d^{q-3/2}}\sin{\Big(2\pi \frac{\sqrt{m}}{d}x^{2}-\frac{\pi}{4}\Big)}+\\
&+(-1)^{\frac{q-1}{2}}2^{2q-1}\varrho_{\chi,q}\underset{\,d\equiv0\,(4)}{\underset{d\leq\sqrt{H}}{\sum_{m\leq M}}}\,\frac{\mathfrak{a}_{H,\chi}\big(m,d;q\big)}{d^{q-3/2}}\cos{\Big(2\pi \frac{\sqrt{m}}{d}x^{2}-\frac{\pi}{4}\Big)}\,.
\end{split}
\end{equation*}
\end{prop}
\begin{proof} 
Fix an integer $M\geq1$. We refer to Lemma 2 with $X_{0}=2M$. Let $X>X_{0}$, $H=H(X)=X^{2}/2$. For $q\equiv0\,(2)$ we define:
\begin{equation*}
\begin{split}
&\nu(d)=2\varrho_{q}\frac{\xi\big(d;q\big)}{d^{q-3/2}}\mathds{1}_{d\leq\sqrt{H}}\quad\quad\quad\quad\quad;\quad\alpha(m,d)=\mathfrak{a}_{H}\big(m,d;q\big)\mathds{1}_{m>M}\\
&\eta(d)=0\qquad\qquad\qquad\qquad\qquad\qquad\,;\quad\beta(m,d)=0\,.
\end{split}
\end{equation*}
For $q\equiv1\,(2)$ we define:
\begin{equation*}
\begin{split}
&\nu(d)=2^{q}\varrho_{\chi,q}\frac{\chi(d)}{d^{q-3/2}}\mathds{1}_{d\leq\sqrt{H}}\quad\quad\quad\,\,\,\,\,\,;\quad\alpha(m,d)=\mathfrak{a}_{H}\big(m,d;q\big)\mathds{1}_{m>M}\\
&\eta(d)=\frac{(-1)^{\frac{q-1}{2}}2^{2q-1}}{d^{q-3/2}}\mathds{1}_{d\equiv0\,(4)}\mathds{1}_{d\leq\sqrt{H}}\quad;\quad\beta(m,d)=\mathfrak{a}_{H,\chi}\big(m,d;q\big)\mathds{1}_{m>M}
\end{split}
\end{equation*}
With the above definitions we have for $X<x<2X$:
\begin{equation*}
\mathcal{S}_{q,H}(x)-\mathcal{S}^{M}_{q,H}(x)=\mathcal{J}^{\sin{}}_{q,H}\big(x;\nu,\alpha\big)+\mathcal{J}^{\cos{}}_{q,H}\big(x;\eta,\beta\big)\,.
\end{equation*}
Conditions C.1 and C.2 are clearly satisfied. Let us first estimate $\Xi_{H}\big(\nu,\alpha\big)$ and $\Xi_{H}\big(\eta,\beta\big)$. For integers $m,\ell,r,d\geq1$ we have:
\begin{equation*}
\begin{split}
&|\alpha\big((rm)^{2}\ell,rd\big)|,\,|\alpha\big((rm)^{2}\ell,rd\big)|\leq c\frac{r_{2}\big((rm)^{2}\ell,rd;q\big)}{(m^{2}\ell)^{3/4}r^{3/2}}\mathds{1}_{M<m^{2}\ell r^{2}\leq2H^{2}}\,\,; c>0 \textit{ an absolute constant}\\
&r_{2}\big((rm)^{2}\ell,rd;q\big)=r_{2}\big(m^{2}\ell,d;q\big)\leq r_{2}\big(m^{2}\ell\big)\,.
\end{split}
\end{equation*}
Thus:
\begin{equation}\label{eq:4.10}
\begin{split}
0\leq\Xi_{H}\big(\nu,\alpha\big)+ \Xi_{H}\big(\eta,\beta\big)\ll
\sum_{n}\frac{r^{2}_{2}(n)}{n^{3/2}}\Bigg(\,\,\sum_{nr^{2}>M}\frac{1}{r^{q}}\Bigg)^{2}\ll M^{-1/2}\log{2M}\,.
\end{split}
\end{equation}
Taking $\limsup$ we obtain from \eqref{eq:4.8} in Lemma $2$:
\begin{equation}\label{eq:4.11}
\limsup_{X\to\infty}\frac{1}{X}\bigintssss\limits_{X}^{2X}\Big|\mathcal{S}_{q,H}(x)-\mathcal{S}^{M}_{q,H}(x)\Big|^{2}\textit{d}x\ll M^{-1/2}\log{2M}\,.
\end{equation}
Letting $M\to\infty$ concludes the proof.
\end{proof}
\subsection{Proof of Theorem 4} Before presenting the proof of Theorem 4, we need the following lemma (see \cite{gath2019analogue}, $\S$\textit{4.3} \textit{Lemma 4.5}) which shows that for integers $m\geq1$ which are of moderate size with respect to the parameter $H$, $m^{3/4}\mathfrak{a}_{H}\big(m,d;q\big)$ and $m^{3/4}\mathfrak{a}_{H,\chi}\big(m,d;q\big)$ can be approximated by $r_{2}\big(m,d;q\big)$ and $r_{2,\chi}\big(m,d;q\big)$ with a reasonable error.
\begin{lem}
Let $H\geq1$. Suppose the $1\leq d\leq H$ and $1\leq m\leq Y$ for some $Y\leq H^{2}$, then:
\begin{equation}\label{eq:4.12}
\mathfrak{a}_{H}\big(m,d;q\big)=\frac{r_{2}\big(m,d;q\big)}{4\pi m^{3/4}}+O\Bigg(\frac{r_{2}\big(m,d;q\big)Y}{m^{3/4}H^{2}}\Bigg)
\end{equation}
\begin{equation}\label{eq:4.13}
\mathfrak{a}_{H,\chi}\big(m,d;q\big)=-\frac{r_{2,\chi}\big(m,d;q\big)}{2\pi m^{3/4}}+O\bigg(\frac{r_{2}\big(m,d;q\big)Y}{m^{3/4}H^{2}}\bigg)
\end{equation}
where the implied constant is absolute.
\end{lem}
\noindent
We now proceed to present the proof of Theorem 4.
\begin{proof}(Theorem 4.)
By Proposition 2 and Proposition 3 it suffices to show that:
\begin{equation}\label{eq:4.14}
\lim_{M\to\infty}\limsup_{X\to\infty}\frac{1}{X}\bigintssss\limits_{X}^{2X}\Big|\mathcal{S}^{M}_{q,H}(x)-\sum_{m\leq M}\upphi_{q,m}\big(\gamma_{m}x^{2}\big)\Big|^{2}\textit{d}x=0
\end{equation}
%
%
%
%
%
%
%
%
%
%
Fix an integer $M\geq1$. We refer to Lemma 2 with $X_{0}=2M$. Define:
\begin{equation*}
\mathcal{A}_{M}=\big\{M<n\leq M^{4}\,:\,n=mk^{2}\textit{ with } m>M\textit{ and }\mu^{2}(m)=1\big\}\,.
\end{equation*}
Let $X>X_{0}$, $H=H(X)=X^{2}/2$. For $q\equiv0\,(2)$ we define: 
\begin{equation*}
\begin{split}
&\nu(d)=\frac{\varrho_{q}}{2\pi}\frac{\xi\big(d;q\big)}{d^{q-3/2}}\mathds{1}_{d\leq\sqrt{H}}\qquad\qquad\qquad\qquad\quad\,\,\,\,\,;\quad\alpha(n,d)=\frac{r_{2}\big(n,d;q\big)}{n^{3/4}}\mathds{1}_{n\in\mathcal{A}_{M}}\\
&\eta(d)=0\qquad\qquad\qquad\qquad\qquad\qquad\qquad\qquad\,\,\,\,;\quad\beta(n,d)=0\,.
\end{split}
\end{equation*}
For $q\equiv1\,(2)$ we define:
\begin{equation*}
\begin{split}
&\nu(d)=2^{q-2}\frac{\varrho_{\chi,q}}{\pi}\frac{\chi(d)}{d^{q-3/2}}\mathds{1}_{d\leq\sqrt{H}}\qquad\qquad\qquad\quad\,\,\,\,;\quad\alpha(n,d)=\frac{r_{2}\big(n,d;q\big)}{n^{3/4}}\mathds{1}_{n\in\mathcal{A}_{M}}\\
&\eta(d)=(-1)^{\frac{q+1}{2}}2^{2q-2}\frac{\varrho_{\chi,q}}{\pi}\frac{1}{d^{q-3/2}}\mathds{1}_{d\equiv0\,(4)}\mathds{1}_{d\leq\sqrt{H}}\quad;\quad\beta(n,d)=\frac{r_{2,\chi}\big(n,d;q\big)}{n^{3/4}}\mathds{1}_{n\in\mathcal{A}_{M}}
\end{split}
\end{equation*}
Let $X<x<2X$. By Lemma 3 we have for $q\equiv0\,(2)$:
\begin{equation}\label{eq:4.15}
\begin{split}
&\mathcal{S}^{M}_{q,H}(x)-\sum_{m\leq M}\upphi_{q,m}\big(\gamma_{m}x^{2}\big)=\frac{\varrho_{q}}{2\pi}\underset{d\leq\sqrt{H}}{\sum_{m\leq M}}\,\frac{\xi\big(d;q\big)r_{2}\big(m,d;q\big)}{d^{q-3/2}m^{3/4}}\sin{\Big(2\pi \frac{\sqrt{m}}{d}x^{2}-\frac{\pi}{4}\Big)}-\\
&-\sum_{m\leq M}\upphi_{q,m}\big(\gamma_{m}x^{2}\big)+O\Big(M^{5/4}X^{-4}\Big)=
\frac{\varrho_{q}}{2\pi}\underset{d\leq\sqrt{H}}{\sum_{m\leq M}}\,\frac{\xi\big(d;q\big)r_{2}\big(m,d;q\big)}{d^{q-3/2}m^{3/4}}\sin{\Big(2\pi \frac{\sqrt{m}}{d}x^{2}-\frac{\pi}{4}\Big)}-\\
&-\frac{\varrho_{q}}{2\pi}\underset{d\leq\sqrt{H}}{\underset{m\leq M}{\sum_{mk^{2}\leq M^{4}}}}\,\mu^{2}(m)\frac{\xi\big(d;q\big)r_{2}\big(mk^{2},d;q\big)}{d^{q-3/2}\big(mk^{2}\big)^{3/4}}\sin{\Big(2\pi\frac{\sqrt{mk^{2}}}{d}x^{2}-\frac{\pi}{4}\Big)}+O\Big(M^{5/4}X^{-1/2}+M^{-1/2}\Big(\log{2M}\Big)^{3}\Big)=\\
&=\mathcal{J}^{\sin{}}_{q,H}\big(x;\nu,\alpha\big)+\mathcal{J}^{\cos{}}_{q,H}\big(x;\eta,\beta\big)+O\Big(M^{5/4}X^{-1/2}+M^{-1/2}\Big(\log{2M}\Big)^{3}\Big)\,.
\end{split}
\end{equation}
By Lemma 3 for $q\equiv1\,(2)$, the same arguments give:
\begin{equation}\label{eq:4.16}
\begin{split}
&\mathcal{S}^{M}_{q,H}(x)-\sum_{m\leq M}\upphi_{q,m}\big(\gamma_{m}x^{2}\big)=\mathcal{J}^{\sin{}}_{q,H}\big(x;\nu,\alpha\big)+\mathcal{J}^{\cos{}}_{q,H}\big(x;\eta,\beta\big)+O\Big(M^{5/4}X^{-1/2}+M^{-1/2}\big(\log{2M}\big)^{3}\Big)\,.
\end{split}
\end{equation}
From the definition of $\mathcal{A}_{M}$ it follows that:
\begin{equation*}
0\leq\Xi_{H}\big(\nu,\alpha\big)+\Xi_{H}\big(\eta,\beta\big)
\ll\sum_{n>M}r^{2}_{2}(n)/n^{3/2}\ll M^{-1/2}\log{2M}
\end{equation*}
and so by \eqref{eq:4.8} in Lemma 2 we have:
\begin{equation}\label{eq:4.17}
\begin{split}
&\limsup_{X\to\infty}\frac{1}{X}\bigintssss\limits_{X}^{2X}\Big|\mathcal{J}^{\sin{}}_{q,H}\big(x;\nu,\alpha\big)+\mathcal{J}^{\cos{}}_{q,H}\big(x;\eta,\beta\big)\Big|^{2}\textit{d}x\ll M^{-1/2}\log{2M}\,.
\end{split}
\end{equation}
Applying Cauchy–Schwarz inequality to handle the cross-terms we conclude from \eqref{eq:4.15}, \eqref{eq:4.16} and \eqref{eq:4.17} that:
\begin{equation}\label{eq:4.18}
\limsup_{X\to\infty}\frac{1}{X}\bigintssss\limits_{X}^{2X}\Big|\mathcal{S}^{M}_{q,H}(x)-\sum_{m\leq M}\upphi_{q,m}\big(\gamma_{m}x^{2}\big)\Big|^{2}\textit{d}x\ll M^{-1/2}\log{2M}\,.
\end{equation}
Letting $M\to\infty$ concludes the proof.
\end{proof}
\section{The probability density $\mathcal{P}_{q}(\alpha)$}
\subsection{Statement of Theorem 5 \& 6}
In this section we construct the probability density. The theorems we shall set out to prove are: 
\begin{thm}
Let $\alpha\in\mathbb{C}$. Then the limit:
\begin{equation}\label{eq:5.1}
\Phi_{q}(\alpha)\overset{\textit{def}}{=}\lim_{M\to\infty}\lim_{X\to\infty}\frac{1}{X}\bigintssss\limits_{X}^{2X}\exp{\bigg(2\pi i\alpha\sum_{m\leq M}\upphi_{q,m}\big(\gamma_{m}x^{2}\big)\bigg)}\textit{d}x
\end{equation}
exists, and defines an entire function of $\alpha$. It is given by:
\begin{equation*}
\Phi_{q}(\alpha)=\prod_{m=1}^{\infty}\mathcal{L}(\alpha,m)\quad;\quad\mathcal{L}(\alpha,m)=\lim_{n\to\infty}\frac{1}{n!}\bigintssss\limits_{0}^{n!}\exp{\Big(2\pi i\alpha\upphi_{q,m}(t)\Big)}\textit{d}t\,.
\end{equation*}
where $\mathcal{L}(\alpha,m)$ are entire functions of $\alpha$, and the infinite product converges absolutely and uniformly
on any compact subset of the plane. For large $|\alpha|$, $\alpha=\sigma+i\tau$, $\Phi_{q}(\alpha)$ satisfies the bound:
\begin{equation}\label{eq:5.2}
\begin{split}
\big|\Phi_{q}(\alpha)\big|\leq\exp{\bigg(-\Big(C^{-1}_{q}\sigma^{2}-C_{q}\tau^{2}\Big)\digamma^{-1/2}\big(|\alpha|\big)\log{|\alpha|}+C_{q}|\tau|\digamma^{1/4}\big(|\alpha|\big)\bigg)}
\end{split}
\end{equation}
where $C_{q}>1$ is constant, and $\digamma(x)=x^{\frac{1}{3/4-c/\log\log{x}}}$ with $c>0$ an absolute constant. In particular, for any non-negative integer $j\geq0$ and any $\sigma\in\mathbb{R}$, $|\sigma|>\sigma_{q,j}$, we have:
\begin{equation*}
\begin{split}
\big|\Phi_{q}^{(j)}(\sigma)\big|\leq\exp{\bigg(-K_{q}\sigma^{2}\digamma^{-1/2}\big(|\sigma|\big)\log{|\sigma|}\bigg)}\quad;\quad K_{q}=\frac{1}{2^{6}C_{q}}\,.
\end{split}
\end{equation*}
\end{thm}
\begin{thm}
Let $x\in\mathbb{R}$. Then the integral:
\begin{equation}\label{eq:5.3}
\mathcal{P}_{q}(x)\overset{\textit{def}}{=}\bigintssss\limits_{-\infty}^{\infty}\Phi_{q}(\sigma)\exp{\Big(-2\pi ix\sigma\Big)}\textit{d}\sigma
\end{equation}
converges absolutely, and defines a probability density. It satisfies for any non-negative integer $j\geq0$ and any $x\in\mathbb{R}$, $|x|>x_{q,j}$, the bound:
\begin{equation}\label{eq:5.4}
\begin{split}
\big|\mathcal{P}^{(j)}_{q}(x)\big|\leq\exp{\bigg(-|x|^{4-\beta/\log\log{|x|}}\bigg)}\quad;\quad\beta=48c\,.
\end{split}
\end{equation}
Moreover, the density $\mathcal{P}_{q}(x)$ can be extended to the whole complex plane $\mathbb{C}$ as an entire function of $x$, and in particular is supported on all of the real line.
\end{thm}
\noindent
The proofs will be given at the end of this section
\subsection{Mean value theorems}
In this section we establish certain mean value theorems related to the functions $\upphi_{q,m}(t)$. We make the following definition.
\begin{mydef}
For integers $m,n\geq1$, we define the function $\upphi_{q,m}(\cdot\,;n)$ in accordance with the parity of $q$. For $q\equiv0\,(2)$: 
\begin{equation*}
\upphi_{q,m}(t;n)=\frac{\varrho_{q}}{2\pi}\frac{\mu^{2}(m)}{m^{3/4}}\sum_{d\leq n}\sum_{k=1}^{\infty}\frac{\xi\big(d;q\big)r_{2}\big(mk^{2},d;q\big)}{d^{q-3/2}k^{3/2}}\sin{\Big(2\pi\frac{k}{d}t-\frac{\pi}{4}\Big)}
\end{equation*}
and for $q\equiv1\,(2)$: 
\begin{equation*}
\begin{split}
\upphi_{q,m}(t;n)=2^{q-2}\frac{\varrho_{\chi,q}}{\pi}\frac{\mu^{2}(m)}{m^{3/4}}\Bigg\{&\sum_{d\leq n}\sum_{k=1}^{\infty}\frac{\chi(d)r_{2}\big(mk^{2},d;q\big)}{d^{q-3/2}k^{3/2}}\sin{\Big(2\pi\frac{k}{d}t-\frac{\pi}{4}\Big)}+\\
&+(-1)^{\frac{q+1}{2}}2^{q}\underset{\,\,d\equiv0\,(4)}{\sum_{d\leq n}}\sum_{k=1}^{\infty}\frac{r_{2,\chi}\big(mk^{2},d;q\big)}{d^{q-3/2}k^{3/2}}\cos{\Big(2\pi \frac{k}{d}t-\frac{\pi}{4}\Big)}\Bigg\}\,.
\end{split}
\end{equation*}
\end{mydef}
\begin{lem}
For any integers $m,n\geq1$ and any $t\in\mathbb{R}$:
\begin{equation}\label{eq:5.5}
\begin{split}
\big|\upphi_{q,m}(t)\big|\leq a_{q}\frac{\mu^{2}(m)}{m^{3/4}}r_{2}(m)\quad;\quad\big|\upphi_{q,m}\big(t;n\big)\big|\leq a_{q}\frac{\mu^{2}(m)}{m^{3/4}}r_{2}(m)
\end{split}
\end{equation}
\begin{equation}\label{eq:5.6}
\begin{split}
\big|\upphi_{q,m}(t)-\upphi_{q,m}\big(t;n\big)\big|\leq\frac{a_{q}}{n^{1/2}}\frac{\mu^{2}(m)}{m^{3/4}}r_{2}(m)
\end{split}
\end{equation}
where $a_{q}>0$ is some constant.
\end{lem}
\begin{proof}
Let $m,n\geq1$ be any integers with $\mu^{2}(m)=1$. First, we note that $\big|r_{2,\chi}\big(\cdot\,,\cdot\,;q\big)\big|\leq r_{2}\big(\cdot\,,\cdot\,;q\big)\leq r_{2}\big(\cdot\big)$. Suppose $k,d\geq1$ are integers. If $m$ has a prime divisor $p\equiv3\,(4)$ then $r_{2,\chi}\big(mk^{2},d;q\big)=r_{2}\big(mk^{2},d;q\big)=0$ because $m$ is square-free, and so $\upphi_{q,m}(t)=\upphi_{q,m}\big(t;n\big)\equiv0$. We may thus suppose that if $p|m$ then either $p=2$ or $p\equiv1\,(4)$. In this case we have:
\begin{equation*}
\big|r_{2,\chi}\big(mk^{2},d;q\big)\big|\leq\big|r_{2}\big(mk^{2},d;q\big)\big|\leq r_{2}(mk^{2})\leq r_{2}(m)\bigg(\sum_{\ell|k}1\bigg)^{2}.
\end{equation*}
The lemma now follows by partial summation. 
\end{proof}
\begin{mydef}
Let $f:\mathbb{R}\rightarrow\mathbb{C}$ be a continuous function. For $X>0$, define the mean value $\mathfrak{M}_{X}\big(f\big)$ of $f$ to be:
\begin{equation*}
\mathfrak{M}_{X}\big(f\big)=\frac{1}{X}\bigintssss\limits_{X}^{2X}f(x)\textit{d}x\,.
\end{equation*}
If in addition, $\mathfrak{M}_{X}\big(f\big)$ converges as $X\to\infty$, we shall denote by $\mathcal{L}\big(f\big)$ the resulting limit.
\end{mydef}
\begin{mydef}
For $\alpha\in\mathbb{C}$, $\gamma\in\mathbb{R}$ and $m\geq1$ an integer, let the functions $\mathscr{F}_{\alpha,m}$ and $g_{\gamma}$ be defined by:
\begin{equation*}
\begin{split}
&\mathscr{F}_{\alpha,m}(x)=\exp{\bigg(2\pi i\alpha\upphi_{q,m}\big(\gamma_{m}x^{2}\big)\bigg)}\\
&g_{\gamma}(x)=\exp{\big(2\pi i\gamma x^{2}\big)}\,.
\end{split}
\end{equation*}
\end{mydef}
\begin{lem}
Fix a complex number $\alpha\in\mathbb{C}$. Let $\mathscr{A}$ be a finite set of integers, and $\gamma\in\mathbb{R}$ a real number. Then the limit:
\begin{equation*}
\mathcal{L}\big(\alpha,\mathscr{A};\gamma\big)\overset{\textit{def}}{=}\lim_{X\to\infty}\mathfrak{M}_{X}\Big(g_{\gamma}\prod_{m\in\mathscr{A}}\mathscr{F}_{\alpha,m}\Big)
\end{equation*}
exists. Moreover, if $\gamma \notin\mathscr{B}=\textit{span}_{\mathbb{Q}}\big\{\sqrt{m}:m\in\mathbb{N},\,\,\mu^{2}(m)=1\big\}$ then this limit is zero. In the particular case where $\gamma=0$ and $\mathscr{A}=\{m\}$, the limit is given by:
\begin{equation}\label{eq:5.7}
\mathcal{L}(\alpha,m)=\lim_{n\to\infty}\frac{1}{n!}\bigintssss\limits_{0}^{n!}\exp{\Big(2\pi i\alpha\upphi_{q,m}(t)\Big)}\textit{d}t\,.
\end{equation}
\end{lem}
\begin{proof}
The proof is by induction on $\big|\mathscr{A}\big|=\ell$. For $\ell=0$ and real $\gamma$ we have:
\begin{equation}\label{eq:5.8}
\mathfrak{M}_{X}\big(g_{\gamma}\big)=\frac{1}{X}\bigintssss\limits_{X}^{2X}\exp{\big(2\pi i\gamma x^{2}\big)}\textit{d}x\underset{X\to\infty}{\longrightarrow}\left\{
        \begin{array}{ll}
            1 & ; \gamma=0\\\\
            0 & ; \gamma\neq0\,.
        \end{array}
    \right.
\end{equation}
Suppose the lemma holds for a particular value of $\ell-1$, and let $\mathscr{A}=\big\{m_{1},\ldots,m_{\ell}\big\}=\mathscr{Q}\cup\big\{m_{\ell}\big\}$ be a set of $\ell$ distinct numbers. Fix a real number $\gamma$. If $\mu(m_{\ell})=0$ then $\upphi_{q,m_{\ell}}\big(x\big)\equiv0$, and so the lemma holds trivially for the value $\ell$ as well. Suppose now that $\mu^{2}(m_{\ell})=1$, and let $\epsilon>0$ be arbitrary. By Lemma 4 we may find $N=N_{\epsilon,\ell}$ such that for $n\geq N$ and $t\in\mathbb{R}$:
\begin{equation}\label{eq:5.9}
\Big|\exp{\Big(2\pi i\alpha\Big\{\upphi_{q,m_{\ell}}\big(t\big)-\upphi_{q,m_{\ell}}\big(t;n\big)\Big\}\Big)}-1\Big|\leq\frac{\epsilon}{6}\exp{\Big(-16\pi\ell a_{q}|\alpha|\Big)}\,.
\end{equation}
Thus for any $n\geq N$ and any $X>0$:
\begin{equation}\label{eq:5.10}
\Bigg|\mathfrak{M}_{X}\Big(g_{\gamma}\prod_{m\in\mathscr{A}}\mathscr{F}_{\alpha,m}\Big)-\frac{1}{X}\bigintssss\limits_{X}^{2X}\exp{\Big(2\pi i\alpha\upphi_{q,m_{\ell}}\big(\gamma_{m_{\ell}}x^{2};n\big)\Big)}g_{\gamma}(x)\prod_{m\in\mathscr{Q}}\mathscr{F}_{\alpha,m}(x)\textit{d}x\Bigg|\leq\frac{\epsilon}{6}
\end{equation}
Now, let us fix an integer $n\geq N$. The function $t\mapsto\exp{\Big(2\pi i\alpha\upphi_{q,m_{\ell}}\big(n!t;n\big)\Big)}$ is continuous and $1$-periodic, and has a Fourier series which
converges to it in the mean:
\begin{equation*}
\begin{split}
\lim_{H\to\infty}\bigintssss\limits_{0}^{1}\Big|\exp{\Big(2\pi i\alpha\upphi_{q,m_{\ell}}\big(n!t;n\big)\Big)}-\sum_{h=-H}^{H}c_{h,n}(m_{\ell})\exp{\big(2\pi iht\big)}\Big|\textit{d}t=0
\end{split}
\end{equation*}
where the Fourier coefficients are given by:
\begin{equation*}
c_{h,n}(m_{\ell})=\frac{1}{n!}\bigintssss\limits_{0}^{n!}\exp{\bigg(2\pi i\Big(\alpha\upphi_{q,m_{\ell}}(t;n)-\frac{h}{n!}t\Big)\bigg)}\textit{d}t\,.
\end{equation*}
In particular, for any $X\geq\sqrt{\frac{n!}{\gamma_{m_{\ell}}}}$ we have:
\begin{equation}\label{eq:5.11}
\begin{split}
&\Bigg|\frac{1}{X}\bigintssss\limits_{X}^{2X}\bigg\{\exp{\Big(2\pi i\alpha\upphi_{q,m_{\ell}}\big(\gamma_{m_{\ell}}x^{2};n\big)\Big)}-\sum_{h=-H}^{H}c_{h,n}(m_{\ell})\exp{\Big(2\pi i\frac{h}{n!}\gamma_{m_{\ell}}x^{2}\Big)}\bigg\}g_{\gamma}(x)\prod_{m\in\mathscr{Q}}\mathscr{F}_{\alpha,m}(x)\textit{d}x\Bigg|\leq\\
&\leq2\exp{\Big(16\pi(\ell-1) a_{q}|\alpha|\Big)}\bigintssss\limits_{0}^{1}\Big|\exp{\Big(2\pi i\alpha\upphi_{q,m_{\ell}}\big(n!t;n\big)\Big)}-\sum_{h=-H}^{H}c_{h,n}(m_{\ell})\exp{\big(2\pi iht\big)}\Big|\textit{d}t\underset{H\to\infty}{\longrightarrow}0\,.
\end{split}
\end{equation}
We may thus find some $H=H_{\epsilon,m_{\ell},n}$ such that for all $X\geq\sqrt{\frac{n!}{\gamma_{m_{\ell}}}}$ :
\begin{equation}\label{eq:5.12}
\bigg|\mathfrak{M}_{X}\Big(g_{\gamma}\prod_{m\in\mathscr{A}}\mathscr{F}_{\alpha,m}\Big)-\sum_{h=-H}^{H}c_{h,n}(m_{\ell})\mathfrak{M}_{X}\Big(g_{\gamma+\frac{h}{n!}\gamma_{m_{\ell}}}\prod_{m\in\mathscr{Q}}\mathscr{F}_{\alpha,m}\Big)\bigg|\leq\frac{\epsilon}{3}\,.
\end{equation}
Now, by our induction hypothesis:
\begin{equation}\label{eq:5.13}
\sum_{h=-H}^{H}c_{h,n}(m_{\ell})\mathfrak{M}_{X}\Big(g_{\gamma+\frac{h}{n!}\gamma_{m_{\ell}}}\prod_{m\in\mathscr{Q}}\mathscr{F}_{\alpha,m}\Big)\underset{X\to\infty}{\longrightarrow}\sum_{h=-H}^{H}c_{h,n}(m_{\ell})\mathcal{L}\Big(\alpha,\mathscr{Q};\gamma+\frac{h}{n!}\gamma_{m_{\ell}}\Big)=\mathscr{L}
\end{equation}
for some $\mathscr{L}\in\mathbb{C}$. If $\gamma\notin\mathscr{B}$ then $\gamma+\frac{h}{n!}\gamma_{m_{\ell}}\notin\mathscr{B}$, and so by our induction hypothesis:
\begin{equation*}
\mathcal{L}\Big(\alpha,\mathscr{Q};\gamma+\frac{h}{n!}\gamma_{m_{\ell}}\Big)=0
\end{equation*}
for each value of $-H\leq h\leq H$. Thus, if $\gamma\notin\mathscr{B}$ then $\mathscr{L}=0$. Now, by \eqref{eq:5.13} we may find $T=T_{\epsilon,m_{\ell},n,H}\geq\sqrt{\frac{n!}{\gamma_{m_{\ell}}}}$ such that for all
$X\geq T$:
\begin{equation}\label{eq:5.14}
\Bigg|\sum_{h=-H}^{H}c_{h,n}(m_{\ell})\mathfrak{M}_{X}\Big(g_{\gamma+\frac{h}{n!}\gamma_{m_{\ell}}}\prod_{m\in\mathscr{Q}}\mathscr{F}_{\alpha,m}\Big)-\mathscr{L}\Bigg|\leq\frac{\epsilon}{6}\,.
\end{equation}
In particular, from \eqref{eq:5.12} and \eqref{eq:5.14} we obtain that for any $X,Y\geq T$:
\begin{equation}\label{eq:5.15}
\bigg|\mathfrak{M}_{X}\Big(g_{\gamma}\prod_{m\in\mathscr{A}}\mathscr{F}_{\alpha,m}\Big)-\mathfrak{M}_{Y}\Big(g_{\gamma}\prod_{m\in\mathscr{A}}\mathscr{F}_{\alpha,m}\Big)\bigg|\leq\epsilon\,.
\end{equation}
Since $\epsilon$ in \eqref{eq:5.15} is arbitrary, we deduce by Cauchy's convergence criterion that the limit $\mathcal{L}\big(\alpha,\mathscr{A};\gamma\big)$ exists. Now suppose $\gamma\notin\mathscr{B}$. Then the limit $\mathscr{L}$ in \eqref{eq:5.13} satisfies $\mathscr{L}=0$, and so by letting $X\to\infty$ in \eqref{eq:5.12} we deduce that:
\begin{equation}\label{eq:5.16}
\Big|\mathcal{L}\big(\alpha,\mathscr{A};\gamma\big)\Big|\leq\frac{\epsilon}{3}
\end{equation}
for any $\epsilon>0$, which implies that $\mathcal{L}\big(\alpha,\mathscr{A};\gamma\big)=0$.\\\\
We now prove \eqref{eq:5.7} in the particular case $\mathscr{A}=\{m\}$ and $\gamma=0$ $m\geq1$. If $\mu(m)=0$, then $\upphi_{q,m}\big(t\big)\equiv0$, and so the claim is obvious. Suppose now that $\mu^{2}(m)=1$, and let $\epsilon>0$ be arbitrary. Then \eqref{eq:5.12} reads as follows. For any $n\geq N_{\epsilon}$, there exists $H(n)=H_{\epsilon,m,n}$ such that, for all $X\geq\sqrt{\frac{n!}{\gamma_{m}}}$: 
\begin{equation}\label{eq:5.17}
\bigg|\mathfrak{m}_{X}\Big(\mathscr{F}_{\alpha,m}\Big)-\sum_{h=-H(n)}^{H(n)}c_{h,n}(m)\mathfrak{M}_{X}\Big(g_{\frac{h}{n!}\gamma_{m}}\Big)\bigg|\leq\frac{\epsilon}{3}\,.
\end{equation}
By \eqref{eq:5.8} we have:
\begin{equation}\label{eq:5.18}
\underset{h\neq0}{\sum_{h=-H(n)}^{H(n)}}c_{h,n}(m)\mathfrak{M}_{X}\Big(g_{\frac{h}{n!}\gamma_{m}}\Big)\underset{X\to\infty}{\longrightarrow}0
\end{equation}
and since $\mathcal{L}\big({\alpha},m\big)$ exists, on letting $X\to\infty$ in \eqref{eq:5.17} we deduce that for any $n\geq N_{\epsilon}$:
\begin{equation}\label{eq:5.19}
\bigg|\mathcal{L}\big({\alpha},m\big)-c_{0,n}(m)\bigg|\leq\frac{\epsilon}{3}\,.
\end{equation}
Now, for $n\geq N_{\epsilon}$  we have by \eqref{eq:5.9}:
\begin{equation}\label{eq:5.20}
\Bigg|c_{0,n}(m)-\frac{1}{n!}\bigintssss\limits_{0}^{n!}\exp{\Big(2\pi i\alpha\upphi_{q,m}(t)\Big)}\textit{d}t\Bigg|\leq\frac{\epsilon}{6}
\end{equation}
and so, for any $n\geq N_{\epsilon}$ :
\begin{equation}\label{eq:5.21}
\Bigg|\mathcal{L}\big({\alpha},m\big)-\frac{1}{n!}\bigintssss\limits_{0}^{n!}\exp{\Big(2\pi i\alpha\upphi_{q,m}(t)\Big)}\textit{d}t\Bigg|\leq\frac{\epsilon}{2}\,.
\end{equation}
Since $\epsilon$ in \eqref{eq:5.21} is arbitrary, we deduce that:
\begin{equation}\label{eq:5.22}
\mathcal{L}\big({\alpha},m\big)=\lim_{n\to\infty}\frac{1}{n!}\bigintssss\limits_{0}^{n!}\exp{\Big(2\pi i\alpha\upphi_{q,m}(t)\Big)}\textit{d}t\,.
\end{equation}
This completes the proof of Lemma 5.
\end{proof}
\noindent
With the aid of Lemma 5 we can now prove:
\begin{prop}
Fix a complex number $\alpha\in\mathbb{C}$. Then with the notation as in Lemma 5, we have for $\ell\geq1$:
\begin{equation}\label{eq:5.23}
\mathcal{L}\Big(\prod_{m\leq\ell}\mathscr{F}_{\alpha,m}\Big)=\prod_{m\leq \ell}\mathcal{L}\big(\alpha,m\big)\,.
\end{equation}
\end{prop}
\vspace{0.1cm}
\begin{proof}
The proof is by induction on $\ell$, the results being trivial for $\ell=0$. Suppose the claim holds for a particular value $\ell-1$, and write $\mathscr{A}=\big\{1,2,\ldots,\ell\big\}=\mathscr{Q}\cup\big\{\ell\big\}$. If $\mu(\ell)=0$ then $\upphi_{q,\ell}\big(t\big)\equiv0$, and so the claim holds trivially for the value $\ell$ as well. Suppose now that $\mu^{2}(\ell)=1$, and let $\epsilon>0$ be arbitrary. By Lemma 4 we may find $N=N_{\epsilon,\ell}$ such that for $n\geq N$ and $t\in\mathbb{R}$:
\begin{equation}\label{eq:5.24}
\Big|\exp{\Big(2\pi i\alpha\Big\{\upphi_{q,\ell}\big(t\big)-\upphi_{q,\ell}\big(t;n\big)\Big\}\Big)}-1\Big|\leq\frac{\epsilon}{3}\exp{\Big(-16\pi\ell a_{q}|\alpha|\Big)}\,.
\end{equation}
Thus, for any $n\geq N$ and any $X>0$:
\begin{equation}\label{eq:5.25}
\Bigg|\mathfrak{M}_{X}\Big(\prod_{m\in\mathscr{A}}\mathscr{F}_{\alpha,m}\Big)-\frac{1}{X}\bigintssss\limits_{X}^{2X}\exp{\Big(2\pi i\alpha\upphi_{q,\ell}\big(\gamma_{\ell}x^{2};n\big)\Big)}\prod_{m\in\mathscr{Q}}\mathscr{F}_{\alpha,m}(x)\textit{d}x\Bigg|\leq\frac{\epsilon}{3}
\end{equation}
From Lemma 5 with the particular value $\gamma=0$ we know that $\mathcal{L}({\alpha},\mathscr{A};0)$ exists, and so we may find $T=T_{\epsilon,\ell}$ such that for all $X\geq T$ :
\begin{equation}\label{eq:5.26}
\bigg|\mathfrak{M}_{X}\Big(\prod_{m\in\mathscr{A}}\mathscr{F}_{\alpha,m}\Big)-\mathcal{L}({\alpha},\mathscr{A};0)\bigg|\leq\frac{\epsilon}{3}\,.
\end{equation}
Hence, for any $n\geq N$ and any $X\geq T$ :
\begin{equation}\label{eq:5.27}
\bigg|\mathcal{L}({\alpha},\mathscr{A};0)-\frac{1}{X}\bigintssss\limits_{X}^{2X}\exp{\Big(2\pi i\alpha\upphi_{q,\ell}\big(\gamma_{\ell}x^{2};n\big)\Big)}\prod_{m\in\mathscr{Q}}\mathscr{F}_{\alpha,m}(x)\textit{d}x\bigg|\leq\frac{2\epsilon}{3}\,.
\end{equation}
Now, let us fix an integer $n\geq N$. The function $t\mapsto\exp{\Big(2\pi i\alpha\upphi_{q,\ell}\big(n!t;n\big)\Big)}$ is continuous and $1$-periodic, and has a Fourier series which
converges to it in the mean:
\begin{equation*}
\begin{split}
\lim_{H\to\infty}\bigintssss\limits_{0}^{1}\Big|\exp{\Big(2\pi i\alpha\upphi_{q,\ell}\big(n!t;n\big)\Big)}-\sum_{h=-H}^{H}c_{h,n}(\ell)\exp{\big(2\pi iht\big)}\Big|\textit{d}t=0
\end{split}
\end{equation*}
where the Fourier coefficients are given by:
\begin{equation*}
c_{h,n}(\ell)=\frac{1}{n!}\bigintssss\limits_{0}^{n!}\exp{\bigg(2\pi i\Big(\alpha\upphi_{q,\ell}(t;n)-\frac{h}{n!}t\Big)\bigg)}\textit{d}t\,.
\end{equation*}
In particular, for any $X\geq\sqrt{\frac{n!}{\gamma_{\ell}}}$ we have:
\begin{equation}\label{eq:5.28}
\begin{split}
&\Bigg|\frac{1}{X}\bigintssss\limits_{X}^{2X}\bigg\{\exp{\Big(2\pi i\alpha\upphi_{q,\ell}\big(\gamma_{\ell}x^{2};n\big)\Big)}-\sum_{h=-H}^{H}c_{h,n}(\ell)\exp{\Big(2\pi i\frac{h}{n!}\gamma_{\ell}x^{2}\Big)}\bigg\}\prod_{m\in\mathscr{Q}}\mathscr{F}_{\alpha,m}(x)\textit{d}x\Bigg|\leq\\
&\leq2\exp{\Big(8\pi(\ell-1) a_{q}|\alpha|\Big)}\bigintssss\limits_{0}^{1}\Big|\exp{\Big(2\pi i\alpha\upphi_{q,\ell}\big(n!t;n\big)\Big)}-\sum_{h=-H}^{H}c_{h,n}(\ell)\exp{\big(2\pi iht\big)}\Big|\textit{d}t\underset{H\to\infty}{\longrightarrow}0\,.
\end{split}
\end{equation}
We may thus find some $H=H_{\epsilon,\ell,n}$ such that for all $X\geq\textit{max}\Big\{T,\sqrt{\frac{n!}{\gamma_{\ell}}}\Big\}$:
\begin{equation}\label{eq:5.29}
\bigg|\mathcal{L}({\alpha},\mathscr{A};0)-\sum_{h=-H}^{H}c_{h,n}(\ell)\mathfrak{M}_{X}\Big(g_{\frac{h}{n!}\gamma_{\ell}}\prod_{m\in\mathscr{Q}}\mathscr{F}_{\alpha,m}\Big)\bigg|\leq\epsilon\,.
\end{equation}
Now, for $h\neq0$ we have that $\frac{h}{n!}\gamma_{\ell}\notin\mathscr{B}$ so by Lemma 5 with the particular value $\gamma=\frac{h}{n!}\gamma_{\ell}$:
\begin{equation}\label{eq:5.30}
\underset{h\neq0}{\sum_{h=-H}^{H}}c_{h,n}(\ell)\mathfrak{M}_{X}\Big(g_{\frac{h}{n!}\gamma_{\ell}}\prod_{m\in\mathscr{Q}}\mathscr{F}_{\alpha,m}\Big)\underset{X\to\infty}{\longrightarrow}0\,.
\end{equation}
Thus, on letting $X\to\infty$ in \eqref{eq:5.29}, and referring once more to Lemma 5 this time with the particular value $\gamma=0$, we obtain:
\begin{equation}\label{eq:5.31}
\bigg|\mathcal{L}({\alpha},\mathscr{A};0)-c_{0,n}(\ell)\mathcal{L}({\alpha},\mathscr{Q};0)\bigg|\leq\epsilon\,.
\end{equation}
By our induction hypothesis:
\begin{equation}\label{eq:5.32}
\mathcal{L}({\alpha},\mathscr{Q};0)=\prod_{m\in\mathscr{Q}}\mathcal{L}({\alpha},m)\,.
\end{equation}
Letting $n\to\infty$ in \eqref{eq:5.31}, we obtain by \eqref{eq:5.7} in Lemma 5:
\begin{equation}\label{eq:5.33}
\bigg|\mathcal{L}\Big(\prod_{m\leq\ell}\mathscr{F}_{\alpha,m}\Big)-\prod_{m\leq \ell}\mathcal{L}\big(\alpha,m\big)\bigg|\leq\epsilon\,.
\end{equation}
Since $\epsilon$ is arbitrary, we obtain \eqref{eq:5.23}. This concludes the proof.
\end{proof}
\noindent
With the aid of Lemma 5, and repeating the exact same arguments used to prove Proposition 4, one obtains:
\begin{prop}
For any collection of distinct integers $m_{1},\ldots,m_{s}\geq1$, and any collection of (not necessarily distinct) integers $\ell_{1},\ldots,\ell_{s}\geq1$, one has:
\begin{equation}\label{eq:5.34}
\lim_{X\to\infty}\frac{1}{X}\bigintssss\limits_{X}^{2X}\prod_{i=1}^{s}\upphi^{\ell_{i}}_{q,m_{i}}\big(\gamma_{m}x^{2}\big)\textit{d}x=\prod_{i=1}^{s}\mathcal{Q}_{q}(m_{i},\ell_{i})
\end{equation}
where for integers $m,\ell\geq1$:
\begin{equation*}
\mathcal{Q}_{q}(m,\ell)\overset{\textit{def}}{=}\lim_{n\to\infty}\frac{1}{n!}\bigintssss\limits_{0}^{n!}\upphi^{\ell}_{q,m}(t)\textit{d}t\,.
\end{equation*}
\end{prop}
\subsection{Auxiliary estimates}
In this subsection we establish several auxiliary estimates which will be needed for the proof of the main results in this paper.
\begin{lem}
Let $m\geq1$ be an integer. Then with the notation as in Proposition 5, we have:\\\\
\textit{(1)} $\upphi_{q,m}(\cdot)$ has limiting mean value equal to zero:
\begin{equation}\label{eq:5.35}
\mathcal{Q}_{q}(m,1)=0\,.
\end{equation}
\textit{(2)} The limiting mean value of $\upphi^{2}_{q,m}(\cdot)$ in the case where $q\equiv0\,(2)$ is given by:
\begin{equation}\label{eq:5.36}
\begin{split}
\mathcal{Q}_{q}(m,2)=\frac{1}{2}\Bigg(\frac{\pi^{q-1}}{2\Gamma(q)}\Bigg)^{2}\Bigg\{\underset{\,\,\,(d,2mk^{2})=1}{\sum_{d,k=1}^{\infty}}\frac{r^{2}_{2}\big(mk^{2},d;q\big)}{d^{2q-3}k^{3}}+2^{2q}\underset{\,\,d\equiv0\,(4)}{\underset{\,\,(d,mk^{2})=1}{\sum_{d,k=1}^{\infty}}}\frac{r^{2}_{2}\big(mk^{2},d;q\big)}{d^{2q-3}k^{3}}\Bigg\}\frac{\mu^{2}(m)}{m^{3/2}}
\end{split}
\end{equation}
and in the case where $q\equiv1\,(2)$ it is given by:
\begin{equation}\label{eq:5.37}
\begin{split}
\mathcal{Q}_{q}(m,2)=\frac{1}{2}\Bigg(\frac{\pi^{q-1}}{2\Gamma(q)}\Bigg)^{2}\Bigg\{\underset{\,\,\,(d,2mk^{2})=1}{\sum_{d,k=1}^{\infty}}\frac{r^{2}_{2}\big(mk^{2},d;q\big)}{d^{2q-3}k^{3}}+2^{2q}\underset{\,\,d\equiv0\,(4)}{\underset{\,\,(d,mk^{2})=1}{\sum_{d,k=1}^{\infty}}}\frac{r^{2}_{2,\chi}\big(mk^{2},d;q\big)}{d^{2q-3}k^{3}}\Bigg\}\frac{\mu^{2}(m)}{m^{3/2}}\,.
\end{split}
\end{equation}
\textit{(3)} In addition, we have the lower and upper bound estimates:
\begin{equation}\label{eq:5.38}
\begin{split}
\sum_{m\geq\ell}\mathcal{Q}_{q}(m,2)\geq \mathrm{A}_{q}\ell^{-1/2}\log{2\ell}
\end{split}
\end{equation}
\begin{equation}\label{eq:5.39}
\begin{split}
\sum_{m\geq\ell}\mathcal{Q}_{q}(m,2)\leq \mathrm{B}_{q}\ell^{-1/2}\log{2\ell}
\end{split}
\end{equation}
for some constants $A_{q},B_{q}>0$.
\end{lem}
\begin{proof}
\textit{(1)} If $\mu^{2}(m)=0$ then $\upphi_{q,m}(t)\equiv0$, and so the claim is obvious. Suppose then that $\mu^{2}(m)=1$. By Lemma 4, we have for any integer $n\geq1$:
\begin{equation}\label{eq:5.40}
\Bigg|\frac{1}{n!}\bigintssss\limits_{0}^{n!}\upphi_{q,m}(t)-\upphi_{q,m}(t;n)\textit{d}t\Bigg|\leq\frac{1}{n!}\bigintssss\limits_{0}^{n!}\Big|\upphi_{q,m}(t)-\upphi_{q,\ell}(t;n)\Big|\textit{d}t\leq a_{q}\frac{r_{2}(m)}{n^{1/2}m^{3/4}}\,.
\end{equation}
Since $n!/d\in\mathbb{N}$ for every integer $1\leq d\leq n$, we obtain in the case where $q\equiv0\,(2)$:
\begin{equation}\label{eq:5.41}
\frac{1}{n!}\bigintssss\limits_{0}^{n!}\upphi_{q,m}(t;n)\textit{d}t=\frac{\varrho_{q}}{2\pi m^{3/4}}\sum_{d\leq n}\sum_{k=1}^{\infty}\frac{\xi\big(d;q\big)r_{2}\big(mk^{2},d;q\big)}{d^{q-3/2}k^{3/2}}\bigintssss\limits_{0}^{1}\sin{\Big(2\pi k\frac{n!}{d}t-\frac{\pi}{4}\Big)}\textit{d}t=0
\end{equation}
and in the case where $q\equiv1\,(2)$:
\begin{equation}\label{eq:5.42}
\begin{split}
\frac{1}{n!}\bigintssss\limits_{0}^{n!}\upphi_{q,m}(t;n)\textit{d}t&=2^{q-2}\frac{\varrho_{\chi,q}}{\pi m^{3/4}}\sum_{d\leq n}\sum_{k=1}^{\infty}\frac{\chi(d)r_{2}\big(mk^{2},d;q\big)}{d^{q-3/2}k^{3/2}}\bigintssss\limits_{0}^{1}\sin{\Big(2\pi k\frac{n!}{d}t-\frac{\pi}{4}\Big)}\textit{d}t+\\
&+(-1)^{\frac{q+1}{2}}2^{2q-2}\frac{\varrho_{\chi,q}}{\pi m^{3/4}}\underset{\,\,d\equiv0\,(4)}{\sum_{d\leq n}}\sum_{k=1}^{\infty}\frac{r_{2,\chi}\big(mk^{2},d;q\big)}{d^{q-3/2}k^{3/2}}\bigintssss\limits_{0}^{1}\cos{\Big(2\pi k\frac{n!}{d}t-\frac{\pi}{4}\Big)}\textit{d}t=0
\end{split}
\end{equation}
Thus, the LHS of \eqref{eq:5.40} takes the form:
\begin{equation}\label{eq:5.43}
\Bigg|\frac{1}{n!}\bigintssss\limits_{0}^{n!}\upphi_{q,m}(t)\textit{d}t\Bigg|\leq a_{q}\frac{r_{2}(m)}{n^{1/2}m^{3/4}}\,.
\end{equation}
Letting $n\to\infty$ in \eqref{eq:5.43} we obtain \eqref{eq:5.35}.\\
\textit{(2)} We may clearly assume that $\mu^{2}(m)=1$, for otherwise $\upphi_{q,m}(t)\equiv0$ and the result is trivial. By Lemma 4, we have for any integer $n\geq1$:
\begin{equation}\label{eq:5.44}
\frac{1}{n!}\bigintssss\limits_{0}^{n!}\upphi^{2}_{q,m}(t)\textit{d}t=\frac{1}{n!}\bigintssss\limits_{0}^{n!}\upphi^{2}_{q,m}(t;n)\textit{d}t+O\Bigg(\frac{r^{2}_{2}(m)}{n^{1/2}m^{3/2}}\Bigg)\,.
\end{equation}
We note that $r_{2}\big(sh^{2},hd;q\big)=r_{2}\big(s,d;q\big)$ and $r_{2,\chi}\big(sh^{2},hd;q\big)=\chi(h)r_{2,\chi}\big(s,d;q\big)$ for integers $s,h,d\geq1$. In the case where $q\equiv0\,(2)$ we have:
\begin{equation*}
\begin{split}
\frac{1}{n!}\bigintssss\limits_{0}^{n!}\upphi^{2}_{q,m}(t;n)\textit{d}t=\frac{1}{2}\Big(\frac{\varrho_{q}}{2\pi}\Big)^{2}\Bigg\{\sum_{d\leq n}\underset{\,\,(d,k)=1}{\sum_{k=1}^{\infty}}\frac{r^{2}_{2}\big(mk^{2},d;q\big)}{d^{2q-3}k^{3}}\Bigg(\sum_{\,\,b\leq n/d}\frac{\xi\big(bd;q\big)}{b^{q}}\Bigg)^{2}\Bigg\}\frac{\mu^{2}(m)}{m^{3/2}}=
\end{split}
\end{equation*}
\begin{equation}\label{eq:5.45}
\begin{split}
=\frac{1}{2}\Bigg(\frac{\pi^{q-1}}{2\Gamma(q)}\Bigg)^{2}\Bigg\{\underset{\,\,\,(d,2mk^{2})=1}{\sum_{d,k=1}^{\infty}}\frac{r^{2}_{2}\big(mk^{2},d;q\big)}{d^{2q-3}k^{3}}+2^{2q}\underset{\,\,d\equiv0\,(4)}{\underset{\,\,(d,mk^{2})=1}{\sum_{d,k=1}^{\infty}}}\frac{r^{2}_{2}\big(mk^{2},d;q\big)}{d^{2q-3}k^{3}}\Bigg\}\frac{\mu^{2}(m)}{m^{3/2}}+O\Bigg(\frac{r^{2}_{2}(m)}{n^{2}m^{3/2}}\log{2n}\Bigg)\,.
\end{split}
\end{equation}
In \eqref{eq:5.45} we have extended the summation over $1\leq d\leq n$ and $1\leq b\leq n/d$ all the way to infinity, referred to the definition of $\varrho_{q}$, and used the following identity valid for any integer $d\geq1$:
\begin{equation*}
\begin{split}
\sum_{b=1}^{\infty}\frac{\xi\big(bd;q\big)}{b^{q}}=\Big(1-2^{-q}\Big)\zeta(q)\Big\{\mathds{1}_{d\equiv1(2)}+(-1)^{\frac{q}{2}}2^{q}\mathds{1}_{d\equiv0(4)}\Big\}\,.
\end{split}
\end{equation*}
Note that in the second line of \eqref{eq:5.45} we have restricted the summation over the variable $d$ to those satisfying the additional coprimality condition $(d,m)=1$. We may do so because if $mk^{2}=a^{2}+b^{2}$ with $b\equiv0(d)$ where $d,k\geq1$ are integers satisfying $(d,k)=1$, then since $m$ is square-free we must have that $(d,m)=1$. Letting $n\to\infty$ in \eqref{eq:5.44} we obtain from \eqref{eq:5.45} that for $q\equiv0\,(2)$:
\begin{equation}\label{eq:5.46}
\begin{split}
\mathcal{Q}_{q}(m,2)=\frac{1}{2}\Bigg(\frac{\pi^{q-1}}{2\Gamma(q)}\Bigg)^{2}\Bigg\{\underset{\,\,\,(d,2mk^{2})=1}{\sum_{d,k=1}^{\infty}}\frac{r^{2}_{2}\big(mk^{2},d;q\big)}{d^{2q-3}k^{3}}+2^{2q}\underset{\,\,d\equiv0\,(4)}{\underset{\,\,(d,mk^{2})=1}{\sum_{d,k=1}^{\infty}}}\frac{r^{2}_{2}\big(mk^{2},d;q\big)}{d^{2q-3}k^{3}}\Bigg\}\frac{\mu^{2}(m)}{m^{3/2}}\,.
\end{split}
\end{equation}
In the case $q\equiv1\,(2)$, we obtain in a similar way: 
\begin{equation}\label{eq:5.47}
\begin{split}
\frac{1}{n!}\bigintssss\limits_{0}^{n!}\upphi^{2}_{q,m}(t;n)\textit{d}t=&\frac{1}{2}\Bigg(\frac{\pi^{q-1}}{2\Gamma(q)}\Bigg)^{2}\Bigg\{\underset{\,\,\,(d,2mk^{2})=1}{\sum_{d,k=1}^{\infty}}\frac{r^{2}_{2}\big(mk^{2},d;q\big)}{d^{2q-3}k^{3}}+2^{2q}\underset{\,\,d\equiv0\,(4)}{\underset{\,\,(d,mk^{2})=1}{\sum_{d,k=1}^{\infty}}}\frac{r^{2}_{2,\chi}\big(mk^{2},d;q\big)}{d^{2q-3}k^{3}}\Bigg\}\frac{\mu^{2}(m)}{m^{3/2}}+\\
&+O\Bigg(\frac{r^{2}_{2}(m)}{n^{2}m^{3/2}}\log{2n}\Bigg)\,.
\end{split}
\end{equation}
Letting $n\to\infty$ in \eqref{eq:5.44}  we obtain from \eqref{eq:5.47} that for $q\equiv1\,(2)$::
\begin{equation}\label{eq:5.48}
\begin{split}
\mathcal{Q}_{q}(m,2)=\frac{1}{2}\Bigg(\frac{\pi^{q-1}}{2\Gamma(q)}\Bigg)^{2}\Bigg\{\underset{\,\,\,(d,2k)=1}{\sum_{d,k=1}^{\infty}}\frac{r^{2}_{2}\big(mk^{2},d;q\big)}{d^{2q-3}k^{3}}+2^{2q}\underset{\,\,d\equiv0\,(4)}{\underset{\,\,(d,k)=1}{\sum_{d,k=1}^{\infty}}}\frac{r^{2}_{2,\chi}\big(mk^{2},d;q\big)}{d^{2q-3}k^{3}}\Bigg\}\frac{\mu^{2}(m)}{m^{3/2}}\,.
\end{split}
\end{equation}
\textit{(3)} Since for integers $m\geq1$ we have:
\begin{equation*}
\begin{split}
r_{2}\big(m,1;q\big)=\sum_{a^{2}+b^{2}=m}\bigg(\frac{|a|}{\sqrt{m}}\bigg)^{q-1}\geq\underset{|a|\geq|b|}{\sum_{a^{2}+b^{2}=m}}\bigg(\frac{|a|}{\sqrt{m}}\bigg)^{q-1}\geq2^{-\frac{q+1}{2}}r_{2}(m)
\end{split}
\end{equation*}
it follows from \eqref{eq:5.46} and \eqref{eq:5.48} that:
\begin{equation}\label{eq:5.49}
\begin{split}
\sum_{m\geq\ell}\mathcal{Q}_{q}(m,2)\geq\frac{1}{2}\Bigg(\frac{\pi^{q-1}}{2\Gamma(q)}\Bigg)^{2}\sum_{m\geq\ell}\frac{\mu^{2}(m)}{m^{3/2}}r^{2}_{2}\big(m,1;q\big)\geq\frac{1}{2^{q+4}}\Bigg(\frac{\pi^{q-1}}{\Gamma(q)}\Bigg)^{2}\sum_{m>2\ell}\frac{\mu^{2}(m)}{m^{3/2}}r^{2}_{2}(m)\,.
\end{split}
\end{equation}
From the upper bound \eqref{eq:5.5} in Lemma 4 we have:
\begin{equation}\label{eq:5.50}
\begin{split}
\sum_{m\geq\ell}\mathcal{Q}_{q}(m,2)\leq a^{2}_{q}\sum_{m\geq\ell}\frac{\mu^{2}(m)}{m^{3/2}}r^{2}_{2}(m)\,.
\end{split}
\end{equation}
The desired results now follow by partial summation, together with the estimate:
\begin{equation*}
\begin{split}
\sum_{n\leq N}\mu^{2}(n)r^{2}_{2}(n)\sim hN\log{N}\quad;\quad\textit{as } N\to\infty
\end{split}
\end{equation*}
where $h>0$ is a positive constant. This concludes the proof.
\end{proof}
\begin{lem}
Let $m,\ell\geq1$ be integers. Then:
\begin{equation}\label{eq:5.51}
\mathcal{Q}_{q}(m,\ell)=(-1)^{\ell}\Bigg(\frac{\pi^{q-1}}{4\Gamma(q)}\Bigg)^{\ell}\frac{\mu^{2}(m)}{m^{3\ell/4}}\sum_{a\,(8)}\cos{\Big(\frac{a\pi}{4}\Big)}\underset{\overset{\ell}{\underset{i=1}{\sum}}\textit{e}_{i}\equiv a\,(8)}{\sum_{\textit{e}_{1},\ldots,\textit{e}_{\ell}=\pm1}}\,\,\,\underset{\overset{\ell}{\underset{i=1}{\sum}}\textit{e}_{i}\epsilon_{q}(d_{i})\frac{k_{i}}{d_{i}}=0}{\sum_{d_{1},k_{1},\ldots,d_{\ell},k_{\ell}=1}^{\infty}}\,\,\prod_{i=1}^{\ell}\frac{\mathfrak{r}\big(mk^{2}_{i},d_{i}\,;q\big)}{d_{i}^{q-3/2}k_{i}^{3/2}}
\end{equation}
where for $q\equiv0\,(2)$ we define $\epsilon_{q}(d)=1$, and:
\begin{equation*}
\mathfrak{r}\big(mk^{2},d\,;q\big)=\left\{
        \begin{array}{ll}
            0& ;\,(k,d)\neq1\text{ or } d\equiv2\,(4)\\\\
            r_{2}\big(mk^{2},d;q\big)& ;\, (d,2)=1\\\\
            (-1)^{\frac{q}{2}}2^{q}r_{2}\big(mk^{2},d;q\big)& ;\, d\equiv0\,(4)
        \end{array}
    \right.
\end{equation*}
while for $q\equiv1\,(2)$ we define $\epsilon_{q}(d)=1$ if $(d,2)=1$ and $\epsilon_{q}(d)=-1$ otherwise, and:
\begin{equation*}
\mathfrak{r}\big(mk^{2},d\,;q\big)=\left\{
        \begin{array}{ll}
            0& ;\,(k,d)\neq1\text{ or } d\equiv2\,(4)\\\\
            \chi(d)r_{2}\big(mk^{2},d;q\big)& ;\, (d,2)=1\\\\
            (-1)^{\frac{q-1}{2}}2^{q}r_{2,\chi}\big(mk^{2},d;q\big)& ;\, d\equiv0\,(4)\,.
        \end{array}
    \right.
\end{equation*}
Note that for $\ell=1$ the sum in \eqref{eq:5.51} is void, so by definition $\mathcal{Q}_{q}(m,1)=0$ which is consistent with \eqref{eq:5.35} in Lemma 6. In addition we have:
\begin{equation}\label{eq:5.52}
\sum_{m=1}^{\infty}\mathcal{Q}_{q}(m,3)<0\,.
\end{equation}
where the series on the LHS of \eqref{eq:5.52} converges absolutely.
\end{lem}
\begin{proof}
Fix some integers $m,\ell\geq1$. The case $\mathcal{Q}_{q}(m,\ell)$ where $\ell=1,2$ was treated in Lemma 6, so we may assume that $\ell\geq3$. We may clearly suppose that $\mu^{2}(m)=1$, for otherwise $\upphi_{q,m}(t)\equiv0$ and the result is trivial. By Lemma 4, we have for any integer $n\geq1$:
\begin{equation}\label{eq:5.53}
\frac{1}{n!}\bigintssss\limits_{0}^{n!}\upphi^{\ell}_{q,m}(t)\textit{d}t=\frac{1}{n!}\bigintssss\limits_{0}^{n!}\upphi^{\ell}_{q,m}(t;n)\textit{d}t+O\Bigg(\frac{r^{\ell}_{2}(m)}{n^{1/2}m^{3\ell/4}}\Bigg)\,.
\end{equation}
For $q\equiv0\,(2)$ we have:
\begin{equation*}
\begin{split}
\upphi^{\ell}_{q,m}(t;n)&=\bigg(\frac{\varrho_{q}}{2\pi}\bigg)^{\ell}\frac{\mu^{2}(m)}{m^{3\ell/4}}\Bigg\{\sum_{d\leq n}\sum_{k=1}^{\infty}\frac{\xi\big(d;q\big)r_{2}\big(mk^{2},d;q\big)}{d^{q-3/2}k^{3/2}}\sin{\Big(2\pi\frac{k}{d}t-\frac{\pi}{4}\Big)}\Bigg\}^{\ell}=\\
&=(-1)^{\ell}\bigg(\frac{\varrho_{q}}{4\pi}\bigg)^{\ell}\frac{\mu^{2}(m)}{m^{3\ell/4}}\sum_{a\,(8)}\,\,\underset{\overset{\ell}{\underset{i=1}{\sum}}\textit{e}_{i}\equiv a\,(8)}{\sum_{\textit{e}_{1},\ldots,\textit{e}_{\ell}=\pm1}}\,\,\underset{d_{1},\ldots,d_{\ell}\leq n}{\sum_{k_{1},\ldots,k_{\ell}=1}^{\infty}}\prod_{i=1}^{\ell}\frac{\xi\big(d;q\big)r_{2}\big(mk^{2},d;q\big)}{d^{q-3/2}k^{3/2}}\cos{\Bigg(2\pi t\sum_{i}\textit{e}_{i}\frac{k_{i}}{d_{i}}+\frac{a\pi}{4}\Bigg)}\,.
\end{split}
\end{equation*}
%
Since $r_{2}\big(mk^{2}s^{2},ds;q\big)=r_{2}\big(mk^{2},d;q\big)$ and:
\begin{equation*}
\begin{split}
\sum_{s=1}^{\infty}\frac{\xi\big(sd;q\big)}{s^{q}}=\Big(1-2^{-q}\Big)\zeta(q)\Big\{\mathds{1}_{d\equiv1(2)}+(-1)^{\frac{q}{2}}2^{q}\mathds{1}_{d\equiv0(4)}\Big\}
\end{split}
\end{equation*}
on letting $n\to\infty$ in \eqref{eq:5.53} we obtain:
\begin{equation}\label{eq:5.54}
\begin{split}
\mathcal{Q}_{q}(m,\ell)&=(-1)^{\ell}\bigg(\frac{\varrho_{q}}{4\pi}\bigg)^{\ell}\frac{\mu^{2}(m)}{m^{3\ell/4}}\lim_{n\to\infty}
\sum_{a\,(8)}\,\,\underset{\overset{\ell}{\underset{i=1}{\sum}}\textit{e}_{i}\equiv a\,(8)}{\underset{d_{1},\ldots,d_{\ell}\leq n}{\underset{k_{1},\ldots,k_{\ell}\geq1}{\sum_{\textit{e}_{1},\ldots,\textit{e}_{\ell}=\pm1}}}}\prod_{i=1}^{\ell}\frac{\xi\big(d;q\big)r_{2}\big(mk^{2},d;q\big)}{d^{q-3/2}k^{3/2}}\bigintssss\limits_{0}^{1}\cos{\Bigg(2\pi t\sum_{i}\textit{e}_{i}k_{i}\frac{n!}{d_{i}}+\frac{a\pi}{4}\Bigg)}\textit{d}t=\\
&=(-1)^{\ell}\Bigg(\frac{\pi^{q-1}}{4\Gamma(q)}\Bigg)^{\ell}\frac{\mu^{2}(m)}{m^{3\ell/4}}\sum_{a\,(8)}\cos{\Big(\frac{a\pi}{4}\Big)}\underset{\overset{\ell}{\underset{i=1}{\sum}}\textit{e}_{i}\equiv a\,(8)}{\sum_{\textit{e}_{1},\ldots,\textit{e}_{\ell}=\pm1}}\,\,\,\underset{\overset{\ell}{\underset{i=1}{\sum}}\textit{e}_{i}\frac{k_{i}}{d_{i}}=0}{\sum_{d_{1},k_{1},\ldots,d_{\ell},k_{\ell}=1}^{\infty}}\,\,\prod_{i=1}^{\ell}\frac{\mathfrak{r}\big(mk^{2}_{i},d_{i}\,;q\big)}{d_{i}^{q-3/2}k_{i}^{3/2}}\,.
\end{split}
\end{equation}
For $q\equiv1\,(2)$ we have: 
\begin{equation*}
\begin{split}
&\upphi^{\ell}_{q,m}(t;n)=(-1)^{\ell}\bigg(2^{q-3}\frac{\varrho_{\chi,q}}{\pi}\bigg)^{\ell}\frac{\mu^{2}(m)}{m^{3\ell/4}}\sum_{a\,(8)}\,\,\underset{\overset{\ell}{\underset{i=1}{\sum}}\textit{e}_{i}\equiv a\,(8)}{\sum_{\textit{e}_{1},\ldots,\textit{e}_{\ell}=\pm1}}\,\,\underset{d_{1},\ldots,d_{\ell}\leq n}{\sum_{k_{1},\ldots,k_{\ell}=1}^{\infty}}\prod_{i=1}^{\ell}\frac{\mathfrak{z}\big(mk^{2},d\,;q\big)}{d^{q-3/2}k^{3/2}}\cos{\Bigg(2\pi t\sum_{i}\textit{e}_{i}\epsilon_{q}(d_{i})\frac{k_{i}}{d_{i}}+\frac{a\pi}{4}\Bigg)}
\end{split}
\end{equation*}
where $\epsilon_{q}(d)=1$ if $(d,2)=1$ and $\epsilon_{q}(d)=-1$ otherwise, and:
\begin{equation*}
\mathfrak{z}\big(mk^{2},d\,;q\big)=\left\{
        \begin{array}{ll}
            0& ;\,d\equiv2\,(4)\\\\
            \chi(d)r_{2}\big(mk^{2},d;q\big)& ;\, (d,2)=1\\\\
            (-1)^{\frac{q-1}{2}}2^{q}r_{2,\chi}\big(mk^{2},d;q\big)& ;\, d\equiv0\,(4)\,.
        \end{array}
    \right.
\end{equation*}
Since $r_{2,\chi}\big(mk^{2}s^{2},ds;q\big)=\chi(s)r_{2,\chi}\big(mk^{2},d;q\big)$ for integers $s\geq1$, on letting $n\to\infty$ in \eqref{eq:5.53} we obtain:
\begin{equation*}
\begin{split}
\mathcal{Q}_{q}(m,\ell)=(-1)^{\ell}\bigg(2^{q-3}\frac{\varrho_{\chi,q}}{\pi}\bigg)^{\ell}\frac{\mu^{2}(m)}{m^{3\ell/4}}\sum_{a\,(8)}\cos{\Big(\frac{a\pi}{4}\Big)}\underset{\overset{\ell}{\underset{i=1}{\sum}}\textit{e}_{i}\equiv a\,(8)}{\sum_{\textit{e}_{1},\ldots,\textit{e}_{\ell}=\pm1}}\,\,\,\underset{\overset{\ell}{\underset{i=1}{\sum}}\textit{e}_{i}\epsilon_{q}(d_{i})\frac{k_{i}}{d_{i}}=0}{\sum_{d_{1},k_{1},\ldots,d_{\ell},k_{\ell}=1}^{\infty}}\,\,\prod_{i=1}^{\ell}\frac{\mathfrak{z}\big(mk^{2}_{i},d_{i}\,;q\big)}{d_{i}^{q-3/2}k_{i}^{3/2}}=
\end{split}
\end{equation*}
\begin{equation}\label{eq:5.55}
\begin{split}
\qquad\qquad=(-1)^{\ell}\Bigg(\frac{\pi^{q-1}}{4\Gamma(q)}\Bigg)^{\ell}\frac{\mu^{2}(m)}{m^{3\ell/4}}\sum_{a\,(8)}\cos{\Big(\frac{a\pi}{4}\Big)}\underset{\overset{\ell}{\underset{i=1}{\sum}}\textit{e}_{i}\equiv a\,(8)}{\sum_{\textit{e}_{1},\ldots,\textit{e}_{\ell}=\pm1}}\,\,\,\underset{\overset{\ell}{\underset{i=1}{\sum}}\textit{e}_{i}\epsilon_{q}(d_{i})\frac{k_{i}}{d_{i}}=0}{\sum_{d_{1},k_{1},\ldots,d_{\ell},k_{\ell}=1}^{\infty}}\,\,\prod_{i=1}^{\ell}\frac{\mathfrak{r}\big(mk^{2}_{i},d_{i}\,;q\big)}{d_{i}^{q-3/2}k_{i}^{3/2}}\,.
\end{split}
\end{equation}
Let us now prove \eqref{eq:5.52}. We only treat the case $q\equiv0\,(2)$, the proof in the case where $q\equiv1\,(2)$ being similar. Let $m\geq1$ be an integer. We have:
\begin{equation}\label{eq:5.56}
\begin{split}
\mathcal{Q}_{q}(m,3)=-\frac{3}{2^{5/2}}\Bigg(\frac{\pi^{q-1}}{2\Gamma(q)}\Bigg)^{3}\frac{\mu^{2}(m)}{m^{9/4}}\mathcal{S}_{q}(m,3)
\end{split}
\end{equation}
where:
\begin{equation}\label{eq:5.57}
\begin{split}
\mathcal{S}_{q}(m,3)=&\sideset{}{^\star}\sum_{d_{i},k_{i}}\underset{(\ell,k_{1}k_{2}k_{3})=1}{\sum_{\ell=1}^{\infty}}\frac{1}{\ell^{3q-9/2}}\prod_{i=1}^{3}\frac{\hat{\xi}\big(\ell d_{i};q\big)r_{2}\big(mk_{i}^{2},\ell d_{i};q\big)}{d_{i}^{q-3/2}k_{i}^{3/2}}\\
&\sideset{}{^\star}\sum_{d_{i},k_{i}}=\underset{(d_{1},d_{2},d_{3})=1}{\underset{d_{i},k_{i}\in\mathbb{N}\,;\,(d_{i},k_{i})=1}{\sum_{k_{1}/d_{1}+k_{2}/d_{2}=k_{3}/d_{3}}}}
\end{split}
\end{equation}
and for integers $n\geq1$ we define $\hat{\xi}(n;q)=\mathds{1}_{n\equiv1(2)}+(-1)^{\frac{q}{2}}2^{q}\mathds{1}_{n\equiv0(4)}$. Since $r_{2}\big(\cdot\,,\cdot\,;q\big)$ is non-negative and satisfies the inequality $r_{2}\big(n,hs;q\big)\leq r_{2}\big(n,s;q\big)$ for integers $n,h,s\geq1$, we obtain the following lower bound:
\begin{equation}\label{eq:5.58}
\begin{split}
&\underset{(\ell,k_{1}k_{2}k_{3})=1}{\sum_{\ell=1}^{\infty}}\frac{1}{\ell^{3q-9/2}}\prod_{i=1}^{3}\hat{\xi}\big(\ell d_{i};q\big)r_{2}\big(mk_{i}^{2},\ell d_{i};q\big)=\prod_{i=1}^{3}\hat{\xi}\big(d_{i};q\big)\underset{(\ell,2)=1}{\underset{(\ell,k_{1}k_{2}k_{3})=1}{\sum_{\ell=1}^{\infty}}}\frac{1}{\ell^{3q-9/2}}\prod_{i=1}^{3}r_{2}\big(mk_{i}^{2},\ell d_{i};q\big)+\\
&+(-1)^{\frac{q}{2}}2^{3q}\sum_{j=2}^{\infty}\frac{1}{2^{j(3q-9/2)}}\underset{(\ell,2)=1}{\underset{(2\ell,k_{1}k_{2}k_{3})=1}{\sum_{\ell=1}^{\infty}}}\frac{1}{\ell^{3q-9/2}}\prod_{i=1}^{3}r_{2}\big(mk_{i}^{2},2^{j}\ell d_{i};q\big)\geq\Bigg\{\prod_{i=1}^{3}\hat{\xi}\big(d_{i};q\big)-\eta_{q}\Bigg\}\mathcal{U}_{\underset{d_{1},d_{2},d_{3}}{k_{1},k_{2},k_{3}}}
\end{split}
\end{equation}
where:
\begin{equation*}
\begin{split}
\mathcal{U}_{\underset{d_{1},d_{2},d_{3}}{k_{1},k_{2},k_{3}}}=\underset{(\ell,2)=1}{\underset{(\ell,k_{1}k_{2}k_{3})=1}{\sum_{\ell=1}^{\infty}}}\frac{1}{\ell^{3q-9/2}}\prod_{i=1}^{3}r_{2}\big(mk_{i}^{2},\ell d_{i};q\big)\quad;\quad\eta_{q}=\frac{1}{2^{3q-10}}\,.
\end{split}
\end{equation*}
Hence:
\begin{equation}\label{eq:5.59}
\begin{split}
\mathcal{S}_{q}(m,3)&\geq
\Big(1-\eta_{q}\Big)\underset{(d_{1}d_{2}d_{3},2)=1}{\sideset{}{^\star}\sum_{d_{i},k_{i}}}\mathcal{U}_{\underset{d_{1},d_{2},d_{3}}{k_{1},k_{2},k_{3}}}-\eta_{q}\sum_{1\leq j<s\leq3}\underset{d_{j},d_{s}\equiv2\,(4)}{\sideset{}{^\star}\sum_{d_{i},k_{i}}}\mathcal{U}_{\underset{d_{1},d_{2},d_{3}}{k_{1},k_{2},k_{3}}}+\Big(2^{2q}-\eta_{q}\Big)\sum_{1\leq j<s\leq3}\underset{d_{j},d_{s}\equiv0\,(4)}{\sideset{}{^\star}\sum_{d_{i},k_{i}}}\mathcal{U}_{\underset{d_{1},d_{2},d_{3}}{k_{1},k_{2},k_{3}}}\geq\\
&\geq\Big(1-\eta_{q}\Big)\underset{(d_{1}d_{2}d_{3},2)=1}{\sideset{}{^\star}\sum_{d_{i},k_{i}}}\mathcal{U}_{\underset{d_{1},d_{2},d_{3}}{k_{1},k_{2},k_{3}}}-\eta_{q}\sum_{1\leq j<s\leq3}\underset{d_{j},d_{s}\equiv2\,(4)}{\sideset{}{^\star}\sum_{d_{i},k_{i}}}\mathcal{U}_{\underset{d_{1},d_{2},d_{3}}{k_{1},k_{2},k_{3}}}
\end{split}
\end{equation}
Fix distinct integers $1\leq w,j,s\leq3$ with $j<s$. We have:
\begin{equation}\label{eq:5.60}
\begin{split}
&\underset{d_{j},d_{s}\equiv2\,(4)}{\sideset{}{^\star}\sum_{d_{i},k_{i}}}\mathcal{U}_{\underset{d_{1},d_{2},d_{3}}{k_{1},k_{2},k_{3}}}=\frac{1}{2^{2q-9/2}}\underset{(k_{j}k_{s},2)=1\,;\,k_{w}\equiv0\,(2)}{\underset{(d_{1}d_{2}d_{3},2)=1}{\sideset{}{^\star}\sum_{d_{i},k_{i}}}}\,\,\underset{(\ell,2)=1}{\underset{(\ell,k_{1}k_{2}k_{3})=1}{\sum_{\ell=1}^{\infty}}}\frac{1}{\ell^{3q-9/2}}\frac{r_{2}\big(m\big(k_{w}/2\big)^{2},\ell d_{w};q\big)}{d_{w}^{q-3/2}k_{w}^{3/2}}\underset{i\neq w}{\prod_{i=1}^{3}}\frac{r_{2}\big(mk_{i}^{2},2\ell d_{i};q\big)}{d_{i}^{q-3/2}k_{i}^{3/2}}\,.
\end{split}
\end{equation}
Since for $k\equiv0\,(2)$:
\begin{equation*}
\begin{split}
&r_{2}\big(m\big(k/2\big)^{2},\ell d;q\big)=\underset{a\equiv0\,(2),\,b\equiv0\,(2\ell d)}{\sum_{a^{2}+b^{2}=mk^{2}}}\bigg(\frac{|a|}{\sqrt{mk^{2}}}\bigg)^{q-1}\leq r_{2}\big(mk^{2},2\ell d;q\big)
\end{split}
\end{equation*}
it follows that:
\begin{equation*}
\begin{split}
&\frac{r_{2}\big(m\big(k_{w}/2\big)^{2},\ell d_{w};q\big)}{d_{w}^{q-3/2}k_{w}^{3/2}}\underset{i\neq w}{\prod_{i=1}^{3}}\frac{r_{2}\big(mk_{i}^{2},2\ell d_{i};q\big)}{d_{i}^{q-3/2}k_{i}^{3/2}}\leq\prod_{i=1}^{3}\frac{r_{2}\big(mk_{i}^{2},2\ell d_{i};q\big)}{d_{i}^{q-3/2}k_{i}^{3/2}}\leq\prod_{i=1}^{3}\frac{r_{2}\big(mk_{i}^{2},\ell d_{i};q\big)}{d_{i}^{q-3/2}k_{i}^{3/2}}\,.
\end{split}
\end{equation*}
Hence:
\begin{equation}\label{eq:5.61}
\begin{split}
&\underset{d_{j},d_{s}\equiv2\,(4)}{\sideset{}{^\star}\sum_{d_{i},k_{i}}}\mathcal{U}_{\underset{d_{1},d_{2},d_{3}}{k_{1},k_{2},k_{3}}}\leq\frac{1}{2^{2q-9/2}}\underset{(k_{j}k_{s},2)=1\,;\,k_{w}\equiv0\,(2)}{\underset{(d_{1}d_{2}d_{3},2)=1}{\sideset{}{^\star}\sum_{d_{i},k_{i}}}}\,\,\underset{(\ell,2)=1}{\underset{(\ell,k_{1}k_{2}k_{3})=1}{\sum_{\ell=1}^{\infty}}}\frac{1}{\ell^{3q-9/2}}\prod_{i=1}^{3}\frac{r_{2}\big(mk_{i}^{2},\ell d_{i};q\big)}{d_{i}^{q-3/2}k_{i}^{3/2}}\leq\\
&\leq\frac{1}{2^{2q-9/2}}\underset{(d_{1}d_{2}d_{3},2)=1}{\sideset{}{^\star}\sum_{d_{i},k_{i}}}\,\,\underset{(\ell,2)=1}{\underset{(\ell,k_{1}k_{2}k_{3})=1}{\sum_{\ell=1}^{\infty}}}\frac{1}{\ell^{3q-9/2}}\prod_{i=1}^{3}\frac{r_{2}\big(mk_{i}^{2},\ell d_{i};q\big)}{d_{i}^{q-3/2}k_{i}^{3/2}}=\frac{1}{2^{2q-9/2}}\underset{(d_{1}d_{2}d_{3},2)=1}{\sideset{}{^\star}\sum_{d_{i},k_{i}}}\mathcal{U}_{\underset{d_{1},d_{2},d_{3}}{k_{1},k_{2},k_{3}}}\,.
\end{split}
\end{equation}
From \eqref{eq:5.59} we deduce that:
\begin{equation}\label{eq:5.62}
\begin{split}
\mathcal{S}_{q}(m,3)&\geq\bigg\{1-\Big(1+\frac{3}{2^{2q-9/2}}\Big)\eta_{q}\bigg\}\underset{(d_{1}d_{2}d_{3},2)=1}{\sideset{}{^\star}\sum_{d_{i},k_{i}}}\mathcal{U}_{\underset{d_{1},d_{2},d_{3}}{k_{1},k_{2},k_{3}}}\geq\frac{1}{2}\underset{(d_{1}d_{2}d_{3},2)=1}{\sideset{}{^\star}\sum_{d_{i},k_{i}}}\mathcal{U}_{\underset{d_{1},d_{2},d_{3}}{k_{1},k_{2},k_{3}}}\geq\\
&\geq\frac{1}{2}\,\,\underset{(k_{1},k_{2},k_{3})=1}{\underset{k_{1},k_{2},k_{3}\geq1}{\sum_{k_{1}+k_{2}=k_{3}}}}\prod_{i=1}^{3}\frac{1}{k_{i}^{3/2}}\sum_{k=1}^{\infty}\frac{1}{k^{9/2}}\prod_{i=1}^{3}r_{2}\big(m\big(kk_{i}\big)^{2},1;q\big)\,.
\end{split}
\end{equation}
For integers $k,s\geq1$ we have:
\begin{equation*}
\begin{split}
&r_{2}\big(m\big(ks\big)^{2},1;q\big)\geq\underset{a,b\equiv0\,(s)}{\sum_{a^{2}+b^{2}=mk^{2}s^{2}}}\,\,\bigg(\frac{|a|}{\sqrt{mk^{2}s^{2}}}\bigg)^{q-1}= r_{2}\big(mk^{2},1;q\big)\geq2^{-\frac{q+1}{2}}r_{2}\big(mk^{2}\big)
\end{split}
\end{equation*}
and so, by \eqref{eq:5.62}:
\begin{equation}\label{eq:5.63}
\begin{split}
\mathcal{S}_{q}(m,3)\geq\frac{1}{2^{\frac{3q+8}{2}}}\sum_{k=1}^{\infty}\frac{1}{k^{9/2}}r^{3}_{2}\big(mk^{2}\big)\,.
\end{split}
\end{equation}
Inserting this lower bound for $\mathcal{S}_{q}(m,3)$ into \eqref{eq:5.56}, we obtain:
\begin{equation}\label{eq:5.64}
\begin{split}
\mathcal{Q}_{q}(m,3)\leq-\Bigg(\frac{\pi^{q-1}}{2^{q+1}\Gamma(q)}\Bigg)^{3}\frac{\mu^{2}(m)}{m^{9/4}}\sum_{k=1}^{\infty}\frac{1}{k^{9/2}}r^{3}_{2}\big(mk^{2}\big)\,.
\end{split}
\end{equation}
Finally, summing over all $m\geq1$ we derive:
\begin{equation}\label{eq:5.65}
\sum_{m=1}^{\infty}\mathcal{Q}_{q}(m,3)\leq-\Bigg(\frac{\pi^{q-1}}{2^{q+1}\Gamma(q)}\Bigg)^{3}\sum_{m,k=1}^{\infty}\frac{r^{3}_{2}\big(mk^{2}\big)}{\big(mk^{2}\big)^{9/4}}\mu^{2}(m)=-\Bigg(\frac{\pi^{q-1}}{2^{q+1}\Gamma(q)}\Bigg)^{3}\sum_{n=1}^{\infty}\frac{r^{3}_{2}(n)}{n^{9/4}}<0\,.
\end{equation}
This concludes the proof.
\end{proof}
\subsection{Proof of Theorem 5 \& 6}
\begin{proof}(Theorem 5)
Let $\alpha\in\mathbb{C}$. By Proposition 4 and Lemma 5 we have:
\begin{equation}\label{eq:5.66}
\Phi_{q}(\alpha)=\prod_{m=1}^{\infty}\mathcal{L}(\alpha,m)\quad;\quad\mathcal{L}(\alpha,m)=\lim_{n\to\infty}\frac{1}{n!}\bigintssss\limits_{0}^{n!}\exp{\Big(2\pi i\alpha\upphi_{q,m}(t)\Big)}\textit{d}t\,.
\end{equation}
Since $\upphi_{q,m}(\cdot)$ is bounded, $\mathcal{L}(\alpha,m)$ defines an entire function of $\alpha$ where the limit converges uniformly on any compact subset of the plane. By Lemma 4 we may find some $M=M_{\alpha}$ such that for any $m\geq M$ and any $t\in\mathbb{R}$:
\begin{equation*}
\exp{\Big(2\pi i\alpha\upphi_{q,m}(t)\Big)}=1+2\pi i\alpha\upphi_{q,m}(t)+E_{\alpha,m}(t)\quad;\quad \big|E_{\alpha,m}(t)\big|\leq|\alpha|^{2}\frac{r^{2}_{2}(m)}{m^{3/2}}\,.
\end{equation*}
By \eqref{eq:5.35} in Lemma 6 it follows that for $m\geq M$:
\begin{equation}\label{eq:5.67}
\begin{split}
&\mathcal{L}(\alpha,m)=1+E_{m}(\alpha)\quad;\quad\big|E_{m}(\alpha)\big|\leq\lim_{n\to\infty}\frac{1}{n!}\bigintssss\limits_{0}^{n!}\big|E_{\alpha,m}(t)\big|\textit{d}t\leq|\alpha|^{2}\frac{r^{2}_{2}(m)}{m^{3/2}}\,.
\end{split}
\end{equation}
Since $\sum_{m=1}^{\infty}r^{2}_{2}(m)/m^{3/2}$ converges, the estimate \eqref{eq:5.67} implies the absolute convergence of the infinite product in \eqref{eq:5.66}, uniformly
on any compact set. Hence $\Phi_{q}(\alpha)$ is an entire function of $\alpha$. Now, let us now estimate $\Phi_{q}(\alpha)$. Let $c>0$ be an absolute constant such that:
\begin{equation*}
r_{2}(m)\leq m^{c/\log\log{m}}\quad:\quad m>2
\end{equation*}
and for real $x>\textit{e}$ we write $\theta(x)=3/4-c/\log\log{x}$ so that $\digamma(x)=x^{1/\theta(x)}$. Let $0<\epsilon<\frac{1}{64}$ be a small constant, $\epsilon=\epsilon_{q}$ depending on $q$, which will be specified later. In what follows, we write $\alpha=\sigma+i\tau$, and assume $|\alpha|$ is sufficiently large in terms of the absolute constant $c$ and the parameter $q$. Set:
\begin{equation*}
\ell=\ell(\alpha)=\Big[\Big(\epsilon^{-1}|\alpha|\Big)^{1/\theta(|\alpha|)}\Big]+1\,.
\end{equation*}
The product over $1\leq m\leq \ell-1$ in \eqref{eq:5.66} is estimated trivially by applying the upper bound \eqref{eq:5.5} in Lemma 4:
\begin{equation}\label{eq:5.68}
\prod_{m=1}^{\ell-1}\big|\mathcal{L}(\alpha,m)\big|\leq\exp{\bigg(a_{q}|\tau|\sum_{m=1}^{\ell-1}\frac{\mu^{2}(m)}{m^{3/4}}r_{2}(m)\bigg)}\leq\exp{\bigg(b_{q}\epsilon^{-1}|\tau|\digamma\big(|\alpha|\big)^{1/4}\bigg)}
\end{equation}
for some positive constant $b_{q}$. Suppose now $m\geq\ell$. Then:
\begin{equation}\label{eq:5.69}
\frac{r_{2}(m)}{m^{3/4}}|\alpha|\leq m^{-\theta(m)}|\alpha|\leq\big(\epsilon^{-1}|\alpha|\big)^{-\frac{\theta(m)}{\theta(|\alpha|)}}|\alpha|\leq\epsilon
\end{equation}
and it follows from \eqref{eq:5.35} in Lemma 6 wand the upper bound \eqref{eq:5.5} for $\upphi_{q,m}(t)$ in Lemma 4 that:
\begin{equation}\label{eq:5.70}
\begin{split}
&\mathcal{L}(\alpha,m)=1-\frac{\big(2\pi\alpha\big)^{2}}{2}\mathcal{Q}_{q}(m,2)+\mathcal{R}_{m}(\alpha)\quad;\quad\mathcal{R}_{m}(\alpha)=\sum_{j=3}^{\infty}\frac{\mathcal{Q}_{q}(m,j)}{j!}\big(2\pi i\alpha\big)^{j}\\
&\big|\mathcal{R}_{m}(\alpha)\big|\leq\bigg\{\frac{2}{3}\pi a_{q}\epsilon\exp{\Big(2\pi a_{q}\epsilon\Big)}\bigg\}\frac{\big(2\pi|\alpha|\big)^{2}}{2}\mathcal{Q}_{q}(m,2)\,.
\end{split}
\end{equation}
We now specify $\epsilon$. We choose $0<\epsilon<\frac{1}{64}$ such that:
\begin{equation*}
\begin{split}
2\pi a_{q}\epsilon^{1/2}\exp{\Big(2\pi a_{q}\epsilon\Big)}\leq1\,.
\end{split}
\end{equation*}
With this choice of $\epsilon$ we have:
\begin{equation}\label{eq:5.71}
\begin{split}
\big|\mathcal{R}_{m}(\alpha)\big|\leq\epsilon^{1/2}\frac{\big(2\pi|\alpha|\big)^{2}}{2}\mathcal{Q}_{q}(m,2)\,.
\end{split}
\end{equation}
Since $\big|\mathcal{L}(\alpha,m)-1\big|\leq\epsilon<\frac{1}{2}$, on rewriting \eqref{eq:5.70} in the form:
\begin{equation}\label{eq:5.72}
\mathcal{L}(\alpha,m)=\exp\bigg(-\frac{\big(2\pi\alpha\big)^{2}}{2}\mathcal{Q}_{q}(m,2)+\widetilde{\mathcal{R}}_{m}(\alpha)\bigg)
\end{equation}
we find that:
\begin{equation}\label{eq:5.73}
\big|\widetilde{\mathcal{R}}_{m}(\alpha)\big|\leq\epsilon^{1/2}\big(2\pi|\alpha|\big)^{2}\mathcal{Q}_{q}(m,2)\,.
\end{equation}
Hence:
\begin{equation}\label{eq:5.74}
\big|\mathcal{L}(\alpha,m)\big|\leq\exp\bigg(-\frac{\pi^{2}}{2}\Big(\sigma^{2}-3\tau^{2}\Big)\mathcal{Q}_{q}(m,2)\bigg)
\end{equation}
It follows that:
\begin{equation}\label{eq:5.75}
\prod_{m\geq\ell}\big|\mathcal{L}(\alpha,m)\big|\leq\exp\bigg(-\frac{\pi^{2}}{2}\Big(\sigma^{2}-3\tau^{2}\Big)\sum_{m\geq\ell}\mathcal{Q}_{q}(m,2)\bigg)\,.
\end{equation}
By \eqref{eq:5.38} and \eqref{eq:5.39} in Lemma 6:
\begin{equation*}
\begin{split}
&\frac{\pi^{2}}{2}\sum_{m\geq\ell}\mathcal{Q}_{q}(m,2)\geq \frac{\pi^{2}}{2}\mathrm{A}_{q}\ell^{-1/2}\log{2\ell}\geq D_{q}\digamma^{-1/2}\big(|\alpha|\big)\log{|\alpha|}\\
&\frac{3\pi^{2}}{2}\sum_{m\geq\ell}\mathcal{Q}_{q}(m,2)\leq\frac{3\pi^{2}}{2} \mathrm{B}_{q}\ell^{-1/2}\log{2\ell}\leq \widetilde{D}_{q}\digamma^{-1/2}\big(|\alpha|\big)\log{|\alpha|}
\end{split}
\end{equation*}
for some constants $\widetilde{D}_{q},D_{q}>0$. Taking $C_{q}= 1+\textit{max}\Big\{\widetilde{D}_{q},\, \frac{1}{D_{q}},\, b_{q}\epsilon^{-1}\Big\}$ we find that:
\begin{equation}\label{eq:5.76}
\begin{split}
\prod_{m\geq\ell}\big|\mathcal{L}(\alpha,m)\big|\leq\exp{\bigg(-\Big(C^{-1}_{q}\sigma^{2}-C_{q}\tau^{2}\Big)\digamma^{-1/2}\big(|\alpha|\big)\log{|\alpha|}\bigg)}
\end{split}
\end{equation}
which together with \eqref{eq:5.68} gives:
\begin{equation}\label{eq:5.77}
\begin{split}
\big|\Phi_{q}(\alpha)\big|=\prod_{m=1}^{\infty}\big|\mathcal{L}(\alpha,m)\big|\leq\exp{\bigg(-\Big(C^{-1}_{q}\sigma^{2}-C_{q}\tau^{2}\Big)\digamma^{-1/2}\big(|\alpha|\big)\log{|\alpha|}+C_{q}|\tau|\digamma^{1/4}\big(|\alpha|\big)\bigg)}
\end{split}
\end{equation}
as claimed. The estimates for the derivatives of $\Phi_{q}(\sigma)$ on the real axis are now straightforward. For $\sigma\in\mathbb{R}$, $|\sigma|$ large, it follows from \eqref{eq:5.77} that:
\begin{equation*}
\begin{split}
\underset{|\omega-\sigma|=1}{\textit{max}}\,\big|\Phi_{q}(\omega)\big|\leq\exp{\bigg(-2K_{q}\sigma^{2}\digamma^{-1/2}\big(|\sigma|\big)\log{|\sigma|}\bigg)}\quad;\quad K_{q}=\frac{1}{2^{6}C_{q}}
\end{split}
\end{equation*}
and so, by Cauchy’s integral formula for the $j$-th derivative we obtain:
\begin{equation*}
\begin{split}
\big|\Phi_{q}^{(j)}(\sigma)\big|=\frac{j!}{2\pi}\Bigg|\,\,\bigintssss\limits_{|\omega-\sigma|=1}\frac{\Phi_{q}(\omega)}{(\omega-\sigma)^{j+1}}\textit{d}\omega\Bigg|\leq j!\underset{|\omega-\sigma|=1}{\textit{max}}\,\big|\Phi_{q}(\omega)\big|\leq\exp{\bigg(-K_{q}\sigma^{2}\digamma^{-1/2}\big(|\sigma|\big)\log{|\sigma|}\bigg)}\,.
\end{split}
\end{equation*}
This completes the proof of Theorem 5.
\end{proof}
\begin{proof}
\textit{(Theorem 6)} Let $x\in\mathbb{C}$ be a complex number. By \eqref{eq:5.2} in Theorem 5 we have for $\sigma\in\mathbb{R}$, $|\sigma|$ large:
\begin{equation*}
\begin{split}
\big|\Phi_{q}(\sigma)\big|\leq\exp{\Bigg(-C^{-1}_{q}\sigma^{2}\digamma^{-1/2}\big(|\sigma|\big)\log{|\sigma|}\Bigg)}.
\end{split}
\end{equation*}
Since $\sigma^{2}\digamma^{-1/2}\big(|\sigma|\big)=|\sigma|^{\frac{4}{3}+O\big(1/\log\log{|\sigma|}\big)}$, it follows that the integral:
\begin{equation*}
\mathcal{P}_{q}(x)=\bigintssss\limits_{-\infty}^{\infty}\Phi_{q}(\sigma)\exp{\Big(-2\pi ix\sigma\Big)}\textit{d}\sigma
\end{equation*}
converges absolutely, and thus defines an entire function of $x$. Let us now estimate $\mathcal{P}^{(j)}_{q}(x)$ for large $|x|$, $x\in\mathbb{R}$ and $j\geq0$ a non-negative integer. We have:
\begin{equation}\label{eq:5.78}
\begin{split}
\mathcal{P}^{(j)}_{q}(x)=
\big(-2\pi i\big)^{j}\bigintssss\limits_{\ell(0)}^{}\alpha^{j}\Phi_{q}(\alpha)\exp{\Big(-2\pi ix\alpha\Big)}\textit{d}\alpha
\end{split}
\end{equation}
where for $\tau\in\mathbb{R}$ we set $\ell(\tau)=\big\{\sigma+i\tau:\sigma\in\mathbb{R}\big\}$. Since $\Phi_{q}(\alpha)$ is an entire function of $\alpha$, by Cauchy's theorem and the decay estimate \eqref{eq:5.2} in Theorem 5, we may shift the line of integration to:
\begin{equation}\label{eq:5.79}
\begin{split}
\mathcal{P}^{(j)}_{q}(x)=\big(-2\pi i\big)^{j}\exp{\Big(2\pi x\tau\Big)}\bigintssss\limits_{-\infty}^{\infty}\big(\sigma+i\tau\big)^{j}\Phi_{q}(\sigma+i\tau)\exp{\Big(-2\pi ix\sigma\Big)}\textit{d}\sigma
\end{split}
\end{equation}
for some $\tau=\tau_{x}$ to be determined later, which satisfies $\textit{sgn}(\tau)=-\textit{sgn}(x)$. Thus:
\begin{equation}\label{eq:5.80}
\begin{split}
\big|\mathcal{P}^{(j)}_{q}(x)\big|&\leq\Big(2\pi C_{q}|\tau|\Big)^{j+1}\exp{\Big(-2\pi |x||\tau|\Big)}\bigintssss\limits_{-\infty}^{\infty}\Big(\sigma^{2}+1\Big)^{j/2}\big|\Phi_{q}(C_{q}\tau\sigma+i\tau)\big|\textit{d}\sigma=\\
&=\Big(2\pi C_{q}|\tau|\Big)^{j+1}\exp{\Big(-2\pi |x||\tau|\Big)}\Bigg\{\bigintssss\limits_{|\sigma|\leq\sqrt{2}}\ldots\textit{d}\sigma+\bigintssss\limits_{|\sigma|>\sqrt{2}}\ldots\textit{d}\sigma\Bigg\}\,.
\end{split}
\end{equation}
In what follows, $|\tau|$ is assumed to be large in terms of the parameter $j$, and the constants $C_{q}$ and $c$ appearing in the statement of Theorem 5. In the range $|\sigma|\leq\sqrt{2}$ we have by Theorem 5:
\begin{equation*}
\begin{split}
\big|\Phi_{q}(\sigma+i\tau)\big|\leq\exp{\Bigg(2C_{q}|\tau|F^{1/4}\Big(\big|C_{q}\tau\sigma+i\tau\big|\Big)\Bigg)}\leq\exp{\Bigg(\big(2C_{q}\big)^{2}|\tau|\digamma^{1/4}\big(|\tau|\big)\Bigg)}
\end{split}
\end{equation*}
and so:
\begin{equation}\label{eq:5.81}
\begin{split}
\mathcal{I}_{1}=\bigintssss\limits_{|\sigma|\leq\sqrt{2}}\Big(\sigma^{2}+1\Big)^{j/2}\big|\Phi_{q}(C_{q}\tau\sigma+i\tau)\big|\textit{d}\sigma\leq\exp{\Bigg(\big(4C_{q}\big)^{2}|\tau|\digamma^{1/4}\big(|\tau|\big)\Bigg)}\,.
\end{split}
\end{equation}
Referring to Theorem 5 once again, we find that in the range $|\sigma|>\sqrt{2}$:
\begin{equation*}
\begin{split}
\big|\Phi_{q}(\sigma+i\tau)\big|\leq\exp{\Bigg(-2C^{2}_{q}\big||\tau|^{4\theta(|\tau|)+1}\sigma\big|^{\frac{1}{4\theta(|\tau|)}}\Big\{\big|G\big(|\tau|\big)\sigma\big|^{\frac{8\theta(|\tau|)-3}{4\theta(|\tau|)}}-1\Big\}\Bigg)}
\end{split}
\end{equation*}
where:
\begin{equation*}
\begin{split}
G\big(|\tau|\big)=\bigg(\frac{\log{|\tau|}}{8C^{2}_{q}}\bigg)^{\frac{4\theta(|\tau|)}{8\theta(|\tau|)-3}}|\tau|^{\frac{4\theta(|\tau|)-3}{8\theta(|\tau|)-3}}\quad;\quad\theta(|\tau|)=\frac{3}{4}-\frac{c}{\log\log{|\tau|}}\,.
\end{split}
\end{equation*}
Making a change of variables, we obtain:
\begin{equation}\label{eq:5.82}
\begin{split}
\mathcal{I}_{2}&=\bigintssss\limits_{|\sigma|>\sqrt{2}}\Big(\sigma^{2}+1\Big)^{j/2}\big|\Phi_{q}(C_{q}\tau\sigma+i\tau)\big|\textit{d}\sigma\leq\\
&\leq\Big(2G^{-1}\big(|\tau|\big)\Big)^{j+1}\bigintssss\limits_{0}^{\infty}\sigma^{j}\exp{\Bigg(-2C^{2}_{q}\Big(|\tau|^{4\theta(|\tau|)+1}G^{-1}\big(|\tau|\big)\sigma\Big)^{\frac{1}{4\theta(|\tau|)}}\Big\{\sigma^{\frac{8\theta(|\tau|)-3}{4\theta(|\tau|)}}-1\Big\}\Bigg)}\textit{d}\sigma=\\
&=\Big(2G^{-1}\big(|\tau|\big)\Big)^{j+1}\Bigg\{\bigintssss\limits_{0}^{y}\ldots\textit{d}\sigma+\bigintssss\limits_{y}^{\infty}\ldots\textit{d}\sigma\Bigg\}\quad;\quad y=2^{\frac{4\theta(|\tau|)}{8\theta(|\tau|)-3}}\,.
\end{split}
\end{equation}
In the range $|\sigma|\leq y$ we estimate trivially:
\begin{equation}\label{eq:5.83}
\begin{split}
\mathcal{I}_{3}=\bigintssss\limits_{0}^{y}\ldots\textit{d}\sigma\leq2^{2j+2}\exp{\Bigg(12C^{2}_{q}\Big(|\tau|^{4\theta(|\tau|)+1}G^{-1}\big(|\tau|\big)\Big)^{\frac{1}{4\theta(|\tau|)}}\Bigg)}\,.
\end{split}
\end{equation}
In the range $|\sigma|> y$, we make a change of variables obtaining:
\begin{equation}\label{eq:5.84}
\begin{split}
&\mathcal{I}_{4}=\bigintssss\limits_{y}^{\infty}\ldots\textit{d}\sigma\leq j!\Big(|\tau|^{-4\theta(|\tau|)-1}G\big(|\tau|\big)^{8\theta(|\tau|)-3}\Big)^{\frac{j+1}{8\theta(|\tau|)-2}}\,.
\end{split}
\end{equation}
We find that the integral $\mathcal{I}_{3}$ dominates, and so:
\begin{equation}\label{eq:5.85}
\begin{split}
\bigintssss\limits_{-\infty}^{\infty}\Big(\sigma^{2}+1\Big)^{j/2}\big|\Phi_{q}(C_{q}\tau\sigma+i\tau)\big|\textit{d}\sigma\leq\Big(2^{6}G^{-1}\big(|\tau|\big)\Big)^{j+1}\exp{\Bigg(12C^{2}_{q}\Big(|\tau|^{4\theta(|\tau|)+1}G^{-1}\big(|\tau|\big)\Big)^{\frac{1}{4\theta(|\tau|)}}\Bigg)}\,.
\end{split}
\end{equation}
By \eqref{eq:5.80} we arrive at:
\begin{equation}\label{eq:5.86}
\begin{split}
\big|\mathcal{P}^{(j)}_{q}(x)\big|\leq\Big(2^{7}C_{q}\pi|\tau|G^{-1}\big(|\tau|\big)\Big)^{j+1}\exp{\Bigg(-2\pi|\tau|\bigg\{|x|-6C^{2}_{q}\pi^{-1}\Big(|\tau|G^{-1}\big(|\tau|\big)\Big)^{\frac{1}{4\theta(|\tau|)}}\bigg\}\Bigg)}\,.
\end{split}
\end{equation}
We choose $|\tau|=|x|^{3-\frac{\beta}{\log\log{|x|}}}$, with $\beta=48c$. A simple calculation gives (remember that $|x|$ is assumed to large):
\begin{equation*}
\begin{split}
|x|-6C^{2}_{q}\pi^{-1}\Big(|\tau|G^{-1}\big(|\tau|\big)\Big)^{\frac{1}{4\theta(|\tau|)}}>\frac{|x|}{2}
\end{split}
\end{equation*}
and we obtain:
\begin{equation}\label{eq:5.87}
\begin{split}
\big|\mathcal{P}^{(j)}_{q}(x)\big|\leq\Big(2^{7}C_{q}\pi|\tau|G^{-1}\big(|\tau|\big)\Big)^{j+1}\exp{\Bigg(-\pi|x|^{4-\beta/\log\log{|x|}}\Bigg)}\leq\exp{\Bigg(-|x|^{4-\beta/\log\log{|x|}}\Bigg)}\,.
\end{split}
\end{equation}
It remains to show that $\mathcal{P}_{q}(x)$ defines a probability density. This will be a consequence of the proof of Theorem 1.
\end{proof}
\section{Proof of the main results: Theorem 1, 2 \& 3}
\noindent
We have everything in place for the proof of the main theorems. We begin with the proof of Theorem 1.
\begin{proof}
\textit{(Theorem 1)} We shall prove that \eqref{eq:2.1} holds with $\mathcal{P}_{q}(\alpha)$ defined as in \eqref{eq:5.3}. The analytic continuation of $\mathcal{P}_{q}(\alpha)$ and the decay estimates \eqref{eq:2.2} for its derivatives have already been established in Theorem 6, where it remained to show that $\mathcal{P}_{q}(\alpha)$ is a probability density. This will follow from our proof. Let us begin by showing that for any $\mathcal{F}\in C^{\infty}_{0}(\mathbb{R})$ the limit:
\begin{equation}\label{eq:6.1}
\mathscr{L}(\mathcal{F})\overset{\textit{def}}{=}\lim\limits_{X\to\infty}\frac{1}{X}\bigintssss\limits_{ X}^{2X}\mathcal{F}\Big(\mathcal{E}_{q}(x)/x^{2q-1}\Big)\textit{d}x
\end{equation}
exists. To that end, let $\mathcal{F}\in C^{\infty}_{0}(\mathbb{R})$. Since $\mathcal{F}$ is smooth and compactly supported, we have:
\begin{equation*}
\big|\mathcal{F}(w)-\mathcal{F}(y)\big|\leq c_{\mathcal{F}}|w-y|
\end{equation*}
for all $w,y\in\mathbb{R}$, where $c_{\mathcal{F}}>0$ is some constant. Thus, for any integer $M\geq1$ and any $X>0$:
\begin{equation}\label{eq:6.2}
\frac{1}{X}\bigintssss\limits_{ X}^{2X}\mathcal{F}\Big(\mathcal{E}_{q}(x)/x^{2q-1}\Big)\textit{d}x=\frac{1}{X}\bigintssss\limits_{ X}^{2X}\mathcal{F}\Big(\sum_{m\leq M}\upphi_{q,m}\big(\gamma_{m}x^{2}\big)\Big)\textit{d}x+\mathscr{E}_{\mathcal{F}}\big(X,M\big)
\end{equation}
where $\mathscr{E}_{\mathcal{F}}\big(X,M\big)$ satisfies the bound:
\begin{equation}\label{eq:6.3}
\big|\mathscr{E}_{\mathcal{F}}\big(X,M\big)\big|\leq c_{\mathcal{F}}\,\frac{1}{X}\bigintssss\limits_{X}^{2X}\Big|\mathcal{E}_{q}(x)/x^{2q-1}-\sum_{m\leq M}\upphi_{q,m}\big(\gamma_{m}x^{2}\big)\Big|\textit{d}x\,. 
\end{equation}
It follows from Theorem 4 that:
\begin{equation}\label{eq:6.4}
\lim_{M\to\infty}\limsup_{X\to\infty}\big|\mathscr{E}_{\mathcal{F}}\big(X,M\big)\big|=0\,.
\end{equation}
As $\mathcal{F}\in C^{\infty}_{0}(\mathbb{R})$, we have using the notation as in $\S3.2$:
\begin{equation}\label{eq:6.5}
\frac{1}{X}\bigintssss\limits_{ X}^{2X}\mathcal{F}\Big(\sum_{m\leq M}\upphi_{q,m}\big(\gamma_{m}x^{2}\big)\Big)\textit{d}x=\bigintssss\limits_{-\infty}^{\infty}\widehat{\mathcal{F}}(\alpha)\mathfrak{M}_{X}\Big(\prod_{m\leq M}\mathscr{F}_{\alpha,m}\Big)\textit{d}\alpha
\end{equation}
where $\widehat{\mathcal{F}}$ denotes the Fourier transform of $\mathcal{F}$. Letting $X\to\infty$ in \eqref{eq:6.5}, we have by Lemma 5 and Proposition 4 together with an application of Lebesgue’s Dominated Convergence Theorem:
\begin{equation}\label{eq:6.6}
\lim_{X\to\infty}\frac{1}{X}\bigintssss\limits_{ X}^{2X}\mathcal{F}\Big(\sum_{m\leq M}\upphi_{q,m}\big(\gamma_{m}x^{2}\big)\Big)\textit{d}x=\bigintssss\limits_{-\infty}^{\infty}\widehat{\mathcal{F}}(\alpha)\prod_{m\leq M}\mathcal{L}\big(\alpha,m\big)\textit{d}\alpha\,.
\end{equation}
Letting $M\to\infty$ in \eqref{eq:6.6}, we have by Theorem 5 and Lebesgue’s Dominated Convergence Theorem:
\begin{equation}\label{eq:6.7}
\begin{split}
\lim_{M\to\infty}\lim_{X\to\infty}\frac{1}{X}\bigintssss\limits_{ X}^{2X}\mathcal{F}\Big(\sum_{m\leq M}\upphi_{q,m}\big(\gamma_{m}x^{2}\big)\Big)\textit{d}x&=\lim_{M\to\infty}\bigintssss\limits_{-\infty}^{\infty}\widehat{\mathcal{F}}(\alpha)\prod_{m\leq M}\mathcal{L}\big(\alpha,m\big)\textit{d}\alpha=\\
&=\bigintssss\limits_{-\infty}^{\infty}\widehat{\mathcal{F}}(\alpha)\Phi_{q}(\alpha)\textit{d}\alpha=\bigintssss\limits_{-\infty}^{\infty}\mathcal{F}(\alpha)\widehat{\Phi}_{q}(\alpha)\textit{d}\alpha
\end{split}
\end{equation}
where in the last equality we made use of Parseval's theorem which is justified due to the decay estimates for $\Phi_{q}(\alpha)$ in Theorem 5. Since by definition $\mathcal{P}_{q}(\alpha)=\widehat{\Phi}_{q}(\alpha)$, it follows from \eqref{eq:6.7} that:
\begin{equation}\label{eq:6.8}
\begin{split}
\frac{1}{X}\bigintssss\limits_{ X}^{2X}\mathcal{F}\Big(\sum_{m\leq M}\upphi_{q,m}\big(\gamma_{m}x^{2}\big)\Big)\textit{d}x=\bigintssss\limits_{-\infty}^{\infty}\mathcal{F}(\alpha)\mathcal{P}_{q}(\alpha)\textit{d}\alpha+\mathscr{E}^{\flat}_{\mathcal{F}}\big(X,M\big)
\end{split}
\end{equation}
where $\mathscr{E}^{\flat}_{\mathcal{F}}\big(X,M\big)$ satisfies:
\begin{equation}\label{eq:6.9}
\lim_{M\to\infty}\lim_{X\to\infty}\big|\mathscr{E}^{\flat}_{\mathcal{F}}\big(X,M\big)\big|=0\,.
\end{equation}
Inserting \eqref{eq:6.8} into \eqref{eq:6.2}, we deduce from \eqref{eq:6.4} and \eqref{eq:6.9} that:
\begin{equation}\label{eq:6.10}
\begin{split}
\limsup_{X\to\infty}\bigg|\frac{1}{X}\bigintssss\limits_{ X}^{2X}\mathcal{F}\Big(\mathcal{E}_{q}(x)/x^{2q-1}\Big)\textit{d}x&-\bigintssss\limits_{-\infty}^{\infty}\mathcal{F}(\alpha)\mathcal{P}_{q}(\alpha)\textit{d}\alpha\bigg|\leq\\
&\leq\lim_{M\to\infty}\lim_{X\to\infty}\big|\mathscr{E}^{\flat}_{\mathcal{F}}\big(X,M\big)\big|+\lim_{M\to\infty}\limsup_{X\to\infty}\big|\mathscr{E}_{\mathcal{F}}\big(X,M\big)\big|=0\,.
\end{split}
\end{equation}
From \eqref{eq:6.10} we conclude that for any $\mathcal{F}\in C^{\infty}_{0}(\mathbb{R})$ the limit $\mathscr{L}(\mathcal{F})$ exists and is given by:
\begin{equation}\label{eq:6.11}
\lim\limits_{X\to\infty}\frac{1}{X}\bigintssss\limits_{ X}^{2X}\mathcal{F}\Big(\mathcal{E}_{q}(x)/x^{2q-1}\Big)\textit{d}x=\bigintssss\limits_{-\infty}^{\infty}\mathcal{F}(\alpha)\mathcal{P}_{q}(\alpha)\textit{d}\alpha\,.
\end{equation}
The result extends easily to the class $C_{0}(\mathbb{R})$ of continuous functions with compact support. Indeed, let $\omega(y)\geq0$ be a smooth bump function supported in $[-1,1]$ having total mass 1, and for integers $n\geq1$ define $\omega_{n}(y)=n\omega(ny)$. Let $\mathcal{F}\in C_{0}(\mathbb{R})$, and define $\mathcal{F}_{n}=\mathcal{F}\star\omega_{n}\in C^{\infty}_{0}(\mathbb{R})$ where $\star$ denotes the Euclidean convolution operator. We have for any integer $n\geq1$ and any $X>0$: 
\begin{equation}\label{eq:6.12}
\begin{split}
\bigg|\frac{1}{X}\bigintssss\limits_{ X}^{2X}\mathcal{F}\Big(&\mathcal{E}_{q}(x)/x^{2q-1}\Big)\textit{d}x-\bigintssss\limits_{-\infty}^{\infty}\mathcal{F}(\alpha)\mathcal{P}_{q}(\alpha)\textit{d}\alpha\bigg|\leq\\
&\leq\bigg|\frac{1}{X}\bigintssss\limits_{ X}^{2X}\mathcal{F}_{n}\Big(\mathcal{E}_{q}(x)/x^{2q-1}\Big)\textit{d}x-\bigintssss\limits_{-\infty}^{\infty}\mathcal{F}(\alpha)\mathcal{P}_{q}(\alpha)\textit{d}\alpha\bigg|+\underset{y\in\mathbb{R}}{\textit{max}}\,\,\big|\mathcal{F}(y)-\mathcal{F}_{n}(y)\big|\,.
\end{split}
\end{equation}
The claim now follows from \eqref{eq:6.12} Since $\underset{y\in\mathbb{R}}{\textit{max}}\,\,\big|\mathcal{F}(y)-\mathcal{F}_{n}(y)\big|\longrightarrow0$ and:
\begin{equation}\label{eq:6.13}
\lim_{n\to\infty}\lim\limits_{X\to\infty}\frac{1}{X}\bigintssss\limits_{ X}^{2X}\mathcal{F}_{n}\Big(\mathcal{E}_{q}(x)/x^{2q-1}\Big)\textit{d}x=\lim_{n\to\infty}\bigintssss\limits_{-\infty}^{\infty}\mathcal{F}_{n}(\alpha)\mathcal{P}_{q}(\alpha)\textit{d}\alpha=\bigintssss\limits_{-\infty}^{\infty}\mathcal{F}(\alpha)\mathcal{P}_{q}(\alpha)\textit{d}\alpha\,.
\end{equation}
Let us now show that the limit $\mathscr{L}(\mathcal{F})$ exists for all continuous bounded functions. Suppose then that $\mathcal{F}$ is a continuous bounded function. Let $\psi\in C^{\infty}_{0}(\mathbb{R})$ satisfy $0\leq\psi(y)\leq1$ and $\psi(y)=1$ for $|y|\leq1$, and set $\psi_{n}(y)=\psi(y/n)$. Define $\mathcal{F}_{n}(y)=\mathcal{F}(y)\psi_{n}(y)\in C_{0}(\mathbb{R})$, and let $c_{\mathcal{F}}$ be given by $c_{\mathcal{F}}=\underset{y\in\mathbb{R}}{\sup}|\mathcal{F}(y)|$. From the definition of $\psi_{n}$ we have: 
\begin{equation}\label{eq:6.14}
\begin{split}
\frac{1}{X}\bigintssss\limits_{ X}^{2X}\Big|\mathcal{F}\Big(\mathcal{E}_{q}(x)/x^{2q-1}&\Big)-\mathcal{F}_{n}\Big(\mathcal{E}_{q}(x)/x^{2q-1}\Big)\Big|\textit{d}x\leq\\
&\leq c_{\mathcal{F}}\,\frac{1}{X}\bigintssss\limits_{X}^{2X}\Big\{1-\psi_{n}\Big(\mathcal{E}_{q}(x)/x^{2q-1}\Big)\Big\}\textit{d}x\leq\frac{c_{\mathcal{F}}}{n^{2}}\frac{1}{X}\bigintssss\limits_{ X}^{2X}\big|\mathcal{E}_{q}(x)/x^{2q-1}\big|^{2}\textit{d}x\,.
\end{split}
\end{equation}
We now refer to \cite{gath2019analogue}, \textit{Theorem 2} $\S1.3$. For $X>2$ we have:
\begin{equation}\label{eq:6.15}
\begin{split}
\frac{1}{X}\bigintssss\limits_{ X}^{2X}\mathcal{E}^{2}_{q}(x)\textit{d}x=
\mathfrak{D}_{q}X^{2(2q-1)}+O\Big( X^{2(2q-1)-1}\log^{2}{X}\Big)
\end{split}
\end{equation}
where $\mathcal{D}_{q}>0$ is some constant. Hence for any integer $n\geq1$ and any $X>2$:
\begin{equation}\label{eq:6.16}
\begin{split}
\frac{1}{X}\bigintssss\limits_{ X}^{2X}\Big|\mathcal{F}\Big(\mathcal{E}_{q}(x)/x^{2q-1}&\Big)-\mathcal{F}_{n}\Big(\mathcal{E}_{q}(x)/x^{2q-1}\Big)\Big|\textit{d}x\leq\frac{c_{\mathcal{F}}}{n^{2}}\bigg\{\mathfrak{D}_{q}+O\Big( X^{-1}\log^{2}{X}\Big)\bigg\}\,.
\end{split}
\end{equation}
It follows that:
\begin{equation}\label{eq:6.17}
\begin{split}
\lim_{n\to\infty}\limsup_{X\to\infty}\frac{1}{X}\bigintssss\limits_{ X}^{2X}\Big|\mathcal{F}\Big(\mathcal{E}_{q}(x)/x^{2q-1}&\Big)-\mathcal{F}_{n}\Big(\mathcal{E}_{q}(x)/x^{2q-1}\Big)\Big|\textit{d}x=0\,.
\end{split}
\end{equation}
Since \eqref{eq:6.11} holds for the class $C_{0}(\mathbb{R})$, we have by the definition of $\psi_{n}$:
\begin{equation}\label{eq:6.18}
\begin{split}
\lim\limits_{X\to\infty}\frac{1}{X}\bigintssss\limits_{ X}^{2X}\mathcal{F}_{n}\Big(\mathcal{E}_{q}(x)/x^{2q-1}\Big)\textit{d}x=\bigintssss\limits_{|\alpha|\leq n}^{}\mathcal{F}(\alpha)\mathcal{P}_{q}(\alpha)\textit{d}\alpha+\bigintssss\limits_{|\alpha|>n}^{}\mathcal{F}(\alpha)\psi_{n}(\alpha)\mathcal{P}_{q}(\alpha)\textit{d}\alpha\,.
\end{split}
\end{equation}
It then follows from the rapid decay of $\mathcal{P}_{q}(\alpha)$ that:
\begin{equation}\label{eq:6.19}
\begin{split}
\lim_{n\to\infty}\lim\limits_{X\to\infty}\frac{1}{X}\bigintssss\limits_{ X}^{2X}\mathcal{F}_{n}\Big(\mathcal{E}_{q}(x)/x^{2q-1}\Big)\textit{d}x=\bigintssss\limits_{-\infty}^{\infty}\mathcal{F}(\alpha)\mathcal{P}_{q}(\alpha)\textit{d}\alpha\,.
\end{split}
\end{equation}
Since for any integer $n\geq1$ and any $X>0$ we have:
\begin{equation}\label{eq:6.20}
\begin{split}
&\bigg|\frac{1}{X}\bigintssss\limits_{ X}^{2X}\mathcal{F}\Big(\mathcal{E}_{q}(x)/x^{2q-1}\Big)\textit{d}x-\bigintssss\limits_{-\infty}^{\infty}\mathcal{F}(\alpha)\mathcal{P}_{q}(\alpha)\textit{d}\alpha\bigg|\leq\bigg|\frac{1}{X}\bigintssss\limits_{ X}^{2X}\mathcal{F}_{n}\Big(\mathcal{E}_{q}(x)/x^{2q-1}\Big)\textit{d}x-\bigintssss\limits_{-\infty}^{\infty}\mathcal{F}(\alpha)\mathcal{P}_{q}(\alpha)\textit{d}\alpha\bigg|+\\
&+\frac{1}{X}\bigintssss\limits_{ X}^{2X}\Big|\mathcal{F}\Big(\mathcal{E}_{q}(x)/x^{2q-1}\Big)-\mathcal{F}_{n}\Big(\mathcal{E}_{q}(x)/x^{2q-1}\Big)\Big|\textit{d}x
\end{split}
\end{equation}
we conclude from \eqref{eq:6.17} and \eqref{eq:6.19} that for any continuous bounded function $\mathcal{F}$, the limit $\mathscr{L}(\mathcal{F})$ in \eqref{eq:6.1} exists and is given by:
\begin{equation}\label{eq:6.21}
\lim\limits_{X\to\infty}\frac{1}{X}\bigintssss\limits_{ X}^{2X}\mathcal{F}\Big(\mathcal{E}_{q}(x)/x^{2q-1}\Big)\textit{d}x=\bigintssss\limits_{-\infty}^{\infty}\mathcal{F}(\alpha)\mathcal{P}_{q}(\alpha)\textit{d}\alpha\,.
\end{equation}
The extension of \eqref{eq:6.21} to the class of bounded piecewise-continuous functions is now straightforward. It remains to show that $\mathcal{P}_{q}(\alpha)$ defines a probability density. To that end, we note that the LHS of \eqref{eq:6.21} is real and non-negative whenever $\mathcal{F}$ is. By Theorem 6 we know that $\mathcal{P}_{q}(\alpha)$ is an entire function of $\alpha$, and in particular is continuous, so by choosing a suitable test function $\mathcal{F}$ in \eqref{eq:6.21} we conclude that $\mathcal{P}_{q}(\alpha)\geq0$ for real $\alpha$. Taking $\mathcal{F}\equiv1$ in \eqref{eq:6.21} we have $\int_{-\infty}^{\infty}\mathcal{P}_{q}(\alpha)\textit{d}\alpha=1$. The proof of Theorem 1 is therefor complete.
\end{proof}
\noindent
We now proceed to present the proof pf Theorem 2.
\begin{proof}
\textit{(Theorem 2)} For integers $M,j\geq1$ we have: 
\begin{equation}\label{eq:6.22}
\begin{split}
\frac{1}{X}\bigintssss\limits_{ X}^{2X}\bigg(\sum_{m\leq M}\upphi_{q,m}\big(\gamma_{m}x^{2}\big)\bigg)^{j}\textit{d}x=
\sum_{s=1}^{j}\underset{\ell_{1},\ldots,\ell_{s}\geq1}{\sum_{\ell_{1}+\cdots+\ell_{s}=j}}\frac{j!}{\ell_{1}!\cdots\ell_{s}!}\sum_{M\geq m_{s}>\cdots>m_{1}\geq 1}\frac{1}{X}\bigintssss\limits_{X}^{2X}\prod_{i=1}^{s}\upphi^{\ell_{i}}_{q,m_{i}}\big(\gamma_{m}x^{2}\big)\textit{d}x\,.
\end{split}
\end{equation}
It follows from Proposition 5 that:
\begin{equation}\label{eq:6.23}
\begin{split}
\lim_{X\to\infty}\frac{1}{X}\bigintssss\limits_{ X}^{2X}\bigg(\sum_{m\leq M}^{}\upphi_{q,m}\big(\gamma_{m}x^{2}\big)\bigg)^{j}\textit{d}x=\sum_{s=1}^{j}\underset{\ell_{1},\ldots\,,\ell_{s}\geq1}{\sum_{\ell_{1}+\cdots+\ell_{s}=j}}\frac{j!}{\ell_{1}!\cdots\ell_{s}!}\sum_{M\geq m_{s}>\cdots>m_{1}\geq 1}\prod_{i=1}^{s}\mathcal{Q}_{q}(m_{i},\ell_{i})\,.
\end{split}
\end{equation}
By \eqref{eq:5.35} in Lemma 6 we have that $\mathcal{Q}_{q}(m,1)=0$ for any integer $m\geq1$, and so by the upper bound \eqref{eq:5.5} for $\upphi_{q,m}(\cdot)$ in Lemma 4 we may extend the summation over the variables $m_{s}>\cdots>m_{1}\geq1$ all the way to infinity, obtaining:
\begin{equation}\label{eq:6.24}
\begin{split}
&\lim_{X\to\infty}\frac{1}{X}\bigintssss\limits_{ X}^{2X}\bigg(\sum_{m\leq M}^{}\upphi_{q,m}\big(\gamma_{m}x^{2}\big)\bigg)^{j}\textit{d}x=\\
&=\sum_{s=1}^{j}\,\,\underset{\,\,\ell_{1},\,\ldots\,,\ell_{s}\geq1}{\sum_{\ell_{1}+\cdots+\ell_{s}=j}}\,\,\frac{j!}{\ell_{1}!\cdots\ell_{s}!}\underset{\,\,\,m_{s}>\cdots>m_{1}}{\sum_{m_{1},\,\ldots\,,m_{s}=1}^{\infty}}\prod_{i=1}^{s}\mathcal{Q}_{q}(m_{i},\ell_{i})+O_{q,j}\bigg(M^{-1/2}\Big(\log{2M}\Big)^{2^{j}-1}\bigg)
\end{split}
\end{equation}
where the series in \eqref{eq:6.24} converges absolutely. Thus, on letting $M\to\infty$ in \eqref{eq:6.24} we conclude that for any integer $j\geq1$ the limit:
\begin{equation}\label{eq:6.25}
\begin{split}
\textit{L}(j)\overset{\textit{def}}{=}\lim_{M\to\infty}\lim_{X\to\infty}\frac{1}{X}\bigintssss\limits_{ X}^{2X}\bigg(\sum_{m\leq M}\upphi_{q,m}\big(\gamma_{m}x^{2}\big)\bigg)^{j}\textit{d}x=\sum_{s=1}^{j}\underset{\ell_{1},\ldots,\ell_{s}\geq1}{\sum_{\ell_{1}+\cdots+\ell_{s}=j}}\frac{j!}{\ell_{1}!\cdots\ell_{s}!}\underset{m_{s}>\cdots>m_{1}}{\sum_{m_{1},\ldots,m_{s}=1}^{\infty}}\prod_{i=1}^{s}\mathcal{Q}_{q}(m_{i},\ell_{i})
\end{split}
\end{equation}
exists, where the series in \eqref{eq:6.25} converges absolutely. Now, fix the integer $j\geq1$ and let $\psi\in C^{\infty}_{0}(\mathbb{R})$ satisfy $0\leq\psi(y)\leq1$ and $\psi(y)=1$ for $|y|\leq1$. Set $\psi_{n}(y)=\psi(y/n)$, and let $\mathcal{F}_{n}(y)=y^{j}\psi_{n}(y)\in C^{\infty}_{0}(\mathbb{R})$. For $M\geq1$ an integer, we have by the definition of $\psi_{n}$:
\begin{equation}\label{eq:6.26}
\begin{split}
\frac{1}{X}\bigintssss\limits_{ X}^{2X}\Bigg|\bigg(\sum_{m\leq M}\upphi_{q,m}\big(\gamma_{m}x^{2}\big)\bigg)^{j}-\mathcal{F}_{n}\bigg(\sum_{m\leq M}\upphi_{q,m}\big(\gamma_{m}x^{2}\big)\bigg)\bigg|\textit{d}x\leq\frac{1}{n^{j}}\frac{1}{X}\bigintssss\limits_{ X}^{2X}\bigg(\sum_{m=1}^{M}\upphi_{q,m}\big(\gamma_{m}x^{2}\big)\bigg)^{2j}\textit{d}x\,.
\end{split}
\end{equation}
Since $\textit{L}(2j)$ exists, it follows that:
\begin{equation}\label{eq:6.27}
\begin{split}
\lim_{n\to\infty}\limsup_{M\to\infty}\limsup_{X\to\infty}\frac{1}{X}\bigintssss\limits_{ X}^{2X}\Bigg|\bigg(\sum_{m\leq M}\upphi_{q,m}\big(\gamma_{m}x^{2}\big)\bigg)^{j}-\mathcal{F}_{n}\bigg(\sum_{m\leq M}\upphi_{q,m}\big(\gamma_{m}x^{2}\big)\bigg)\bigg|\textit{d}x\leq\textit{L}(2j)\lim_{n\to\infty}n^{-j}=0\,.
\end{split}
\end{equation}
By Theorem 1 and the definition of $\psi_{n}$ we have:
\begin{equation}\label{eq:6.28}
\begin{split}
\lim_{M\to\infty}\lim_{X\to\infty}\frac{1}{X}\bigintssss\limits_{ X}^{2X}\mathcal{F}_{n}\Big(\sum_{m\leq M}\upphi_{q,m}\big(\gamma_{m}x^{2}\big)\Big)\textit{d}x=\bigintssss\limits_{|\alpha|\leq n}^{}\alpha^{j}\mathcal{P}_{q}(\alpha)\textit{d}\alpha+\bigintssss\limits_{|\alpha|> n}^{}\alpha^{j}\psi_{n}(\alpha)\mathcal{P}_{q}(\alpha)\textit{d}\alpha\,.
\end{split}
\end{equation}
It follows from the rapid decay of $\mathcal{P}_{q}(\alpha)$ that:
\begin{equation}\label{eq:6.29}
\begin{split}
\lim_{n\to\infty}\lim_{M\to\infty}\lim_{X\to\infty}\frac{1}{X}\bigintssss\limits_{ X}^{2X}\mathcal{F}_{n}\Big(\sum_{m\leq M}\upphi_{q,m}\big(\gamma_{m}x^{2}\big)\Big)\textit{d}x=\bigintssss\limits_{-\infty}^{\infty}\alpha^{j}\mathcal{P}_{q}(\alpha)\textit{d}\alpha\,.
\end{split}
\end{equation}
Since for any integers $n,M\geq1$ and any $X>0$ we have:
\begin{equation}\label{eq:6.30}
\begin{split}
\bigg|\frac{1}{X}\bigintssss\limits_{ X}^{2X}\bigg(\sum_{m\leq M}\upphi_{q,m}\big(\gamma_{m}x^{2}\big)\bigg)^{j}\textit{d}x-\bigintssss\limits_{-\infty}^{\infty}\alpha^{j}\mathcal{P}_{q}(\alpha)\textit{d}\alpha\bigg|&\leq\bigg|\frac{1}{X}\bigintssss\limits_{ X}^{2X}\mathcal{F}_{n}\bigg(\sum_{m\leq M}\upphi_{q,m}\big(\gamma_{m}x^{2}\big)\bigg)\textit{d}x-\bigintssss\limits_{-\infty}^{\infty}\alpha^{j}\mathcal{P}_{q}(\alpha)\textit{d}\alpha\bigg|+\\
&+\frac{1}{X}\bigintssss\limits_{ X}^{2X}\Bigg|\bigg(\sum_{m\leq M}\upphi_{q,m}\big(\gamma_{m}x^{2}\big)\bigg)^{j}-\mathcal{F}_{n}\bigg(\sum_{m\leq M}\upphi_{q,m}\big(\gamma_{m}x^{2}\big)\bigg)\bigg|\textit{d}x
\end{split}
\end{equation}
we conclude from the above that:
\begin{equation}\label{eq:6.31}
\begin{split}
\bigintssss\limits_{-\infty}^{\infty}\alpha^{j}\mathcal{P}_{q}(\alpha)\textit{d}\alpha=\sum_{s=1}^{j}\,\,\underset{\,\,\ell_{1},\,\ldots\,,\ell_{s}\geq1}{\sum_{\ell_{1}+\cdots+\ell_{s}=j}}\,\,\frac{j!}{\ell_{1}!\cdots\ell_{s}!}\underset{\,\,\,m_{s}>\cdots>m_{1}}{\sum_{m_{1},\,\ldots\,,m_{s}=1}^{\infty}}\prod_{i=1}^{s}\mathcal{Q}_{q}(m_{i},\ell_{i})
\end{split}
\end{equation}
where the series on the RHS of \eqref{eq:6.31} converges absolutely. By \eqref{eq:5.51} in Lemma 7 we obtain the expansion \eqref{eq:2.4} of $\mathcal{Q}_{q}(m,\ell)$ stated in Theorem 2. Finally, by \eqref{eq:5.35} in Lemma 6 and \eqref{eq:5.52} in Lemma 7, it follows from \eqref{eq:6.31} that the $j$-th moment of $\mathcal{P}_{q}(\alpha)$ in the particular case $j=1,3$ satisfies:
\begin{equation}\label{eq:6.32}
\begin{split}
\bigintssss\limits_{-\infty}^{\infty}\alpha\mathcal{P}_{q}(\alpha)\textit{d}\alpha=\sum_{m=1}^{\infty}\mathcal{Q}_{q}(m,1)=0\quad;\quad\bigintssss\limits_{-\infty}^{\infty}\alpha^{3}\mathcal{P}_{q}(\alpha)\textit{d}\alpha=\sum_{m=1}^{\infty}\mathcal{Q}_{q}(m,3)<0\,.
\end{split}
\end{equation}
This concludes the proof.
\end{proof}
\noindent
The proof of Theorem 3 is now straightforward.
\begin{proof}(Theorem 3)
First we consider the case $0<\lambda<2$. To that end, fix some $0<\lambda<2$ and let $\mathcal{F}_{n}(y)=|y|^{\lambda}\psi_{n}(y)\in C_{0}(\mathbb{R})$, where $\psi_{n}(y)=\psi(y/n)$ and $\psi\in C^{\infty}_{0}(\mathbb{R})$ satisfies $0\leq\psi(y)\leq1$, and $\psi(y)=1$ for $|y|\leq1$. By \eqref{eq:6.15} and the definition of $\psi_{n}$ we have for $X>2$:
\begin{equation}\label{eq:6.33}
\begin{split}
\frac{1}{X}\bigintssss\limits_{ X}^{2X}\Big|\,\,\big|\mathcal{E}_{q}(x)/x^{2q-1}\big|^{\lambda}-\mathcal{F}_{n}\Big(\mathcal{E}_{q}(x)/x^{2q-1}\Big)\Big|\textit{d}x\leq\frac{1}{n^{2-\lambda}}\bigg\{\mathfrak{D}_{q}+O\Big( X^{-1}\log^{2}{X}\Big)\bigg\}\,.
\end{split}
\end{equation}
Since $2-\lambda>0$, it follows that:
\begin{equation}\label{eq:6.34}
\begin{split}
\lim_{n\to\infty}\limsup_{X\to\infty}\frac{1}{X}\bigintssss\limits_{ X}^{2X}\Big|\,\,\big|\mathcal{E}_{q}(x)/x^{2q-1}\big|^{\lambda}-\mathcal{F}_{n}\Big(\mathcal{E}_{q}(x)/x^{2q-1}\Big)\Big|\textit{d}x=0\,.
\end{split}
\end{equation}
By Theorem 1 and the definition of $\psi_{n}$ we have:
\begin{equation}\label{eq:6.35}
\begin{split}
\lim_{X\to\infty}\frac{1}{X}\bigintssss\limits_{ X}^{2X}\mathcal{F}_{n}\Big(\mathcal{E}_{q}(x)/x^{2q-1}\Big)\textit{d}x=\bigintssss\limits_{|\alpha|\leq n}^{}|\alpha|^{\lambda}\mathcal{P}_{q}(\alpha)\textit{d}\alpha+\bigintssss\limits_{|\alpha|> n}^{}|\alpha|^{\lambda}\psi_{n}(\alpha)\mathcal{P}_{q}(\alpha)\textit{d}\alpha\,.
\end{split}
\end{equation}
It follows from the rapid decay of $\mathcal{P}_{q}(\alpha)$ that:
\begin{equation}\label{eq:6.36}
\begin{split}
\lim_{n\to\infty}\lim_{X\to\infty}\frac{1}{X}\bigintssss\limits_{ X}^{2X}\mathcal{F}_{n}\Big(\mathcal{E}_{q}(x)/x^{2q-1}\Big)\textit{d}x=\bigintssss\limits_{-\infty}^{\infty}|\alpha|^{\lambda}\mathcal{P}_{q}(\alpha)\textit{d}\alpha\,.
\end{split}
\end{equation}
Since for any integer $n\geq1$ and any $X>0$ we have:
\begin{equation}\label{eq:6.37}
\begin{split}
&\bigg|\frac{1}{X}\bigintssss\limits_{ X}^{2X}\big|\mathcal{E}_{q}(x)/x^{2q-1}\big|^{\lambda}\textit{d}x-\bigintssss\limits_{-\infty}^{\infty}|\alpha|^{\lambda}\mathcal{P}_{q}(\alpha)\textit{d}\alpha\bigg|\leq
\bigg|\frac{1}{X}\bigintssss\limits_{ X}^{2X}\mathcal{F}_{n}\Big(\mathcal{E}_{q}(x)/x^{2q-1}\Big)\textit{d}x-\bigintssss\limits_{-\infty}^{\infty}|\alpha|^{\lambda}\mathcal{P}_{q}(\alpha)\textit{d}\alpha\bigg|+\\
&+\frac{1}{X}\bigintssss\limits_{ X}^{2X}\Big|\,\,\big|\mathcal{E}_{q}(x)/x^{2q-1}\big|^{\lambda}-\mathcal{F}_{n}\Big(\mathcal{E}_{q}(x)/x^{2q-1}\Big)\Big|\textit{d}x
\end{split}
\end{equation}
we conclude from the above that:
\begin{equation}\label{eq:6.38}
\begin{split}
\lim_{X\to\infty}\frac{1}{X}\bigintssss\limits_{ X}^{2X}\big|\mathcal{E}_{q}(x)/x^{2q-1}\big|^{\lambda}\textit{d}x=\bigintssss\limits_{-\infty}^{\infty}|\alpha|^{\lambda}\mathcal{P}_{q}(\alpha)\textit{d}\alpha\,.
\end{split}
\end{equation}
The exact same arguments give in the case $\lambda=1$:
\begin{equation}\label{eq:6.39}
\begin{split}
\lim_{X\to\infty}\frac{1}{X}\bigintssss\limits_{ X}^{2X}\mathcal{E}_{q}(x)/x^{2q-1}\textit{d}x=\bigintssss\limits_{-\infty}^{\infty}\alpha\mathcal{P}_{q}(\alpha)\textit{d}\alpha\,.
\end{split}
\end{equation}
We now treat the case $\lambda=2$. We reexamine in detail the statement of \textit{Theorem 2} of \cite{gath2019analogue}, \textit{eq. (1.6) and  (1.7)} therein. Renormalizing by $x^{2q-1}$ and taking limits, we have in the case where $q\equiv0\,(2)$:
\begin{equation}\label{eq:6.40}
\begin{split}
\lim_{X\to\infty}\frac{1}{X}\bigintssss\limits_{ X}^{2X}\big|\mathcal{E}_{q}(x)/x^{2q-1}\big|^{2}\textit{d}x=\frac{1}{2}
\Bigg(\frac{\pi^{q-1}}{2\Gamma(q)}\Bigg)^{2}\Bigg\{\,\,\underset{\,\,(d,2m)=1}{\sum_{d,m=1}^{\infty}}\frac{r^{2}_{2}\big(m,d;q\big)}{m^{3/2}d^{2q-3}}+2^{2q}\underset{d\equiv0(4)}{\underset{\,\,(d,m)=1}{\sum_{d,m=1}^{\infty}}}\frac{r^{2}_{2}\big(m,d;q\big)}{m^{3/2}d^{2q-3}}\Bigg\}
\end{split}
\end{equation}
and in the case where $q\equiv1\,(2)$:
\begin{equation}\label{eq:6.41}
\begin{split}
\lim_{X\to\infty}\frac{1}{X}\bigintssss\limits_{ X}^{2X}\big|\mathcal{E}_{q}(x)/x^{2q-1}\big|^{2}\textit{d}x=\frac{1}{2}
\Bigg(\frac{\pi^{q-1}}{2\Gamma(q)}\Bigg)^{2}\Bigg\{\,\,\underset{\,\,(d,2m)=1}{\sum_{d,m=1}^{\infty}}\frac{r^{2}_{2}\big(m,d;q\big)}{m^{3/2}d^{2q-3}}+2^{2q}\underset{d\equiv0(4)}{\underset{\,\,(d,m)=1}{\sum_{d,m=1}^{\infty}}}\frac{r^{2}_{2,\chi}\big(m,d;q\big)}{m^{3/2}d^{2q-3}}\Bigg\}\,.
\end{split}
\end{equation}
By \eqref{eq:5.36} and \eqref{eq:5.37} in Lemma 6 it follows from Theorem 2 that for $q\equiv0\,(2)$:
\begin{equation}\label{eq:642}
\begin{split}
&\bigintssss\limits_{-\infty}^{\infty}\alpha^{2}\mathcal{P}_{q}(\alpha)\textit{d}\alpha=\sum_{m=1}^{\infty}\mathcal{Q}_{q}(m,2)=\\
&=\frac{1}{2}\Bigg(\frac{\pi^{q-1}}{2\Gamma(q)}\Bigg)^{2}\Bigg\{\underset{\,\,\,(d,2mk^{2})=1}{\sum_{m,d,k=1}^{\infty}}\frac{r^{2}_{2}\big(mk^{2},d;q\big)}{d^{2q-3}\big(mk^{2}\big)^{3/2}}\mu^{2}(m)+2^{2q}\underset{\,\,d\equiv0\,(4)}{\underset{\,\,(d,mk^{2})=1}{\sum_{m,d,k=1}^{\infty}}}\frac{r^{2}_{2}\big(mk^{2},d;q\big)}{d^{2q-3}\big(mk^{2}\big)^{3/2}}\mu^{2}(m)\Bigg\}=\\
&=\lim_{X\to\infty}\frac{1}{X}\bigintssss\limits_{ X}^{2X}\big|\mathcal{E}_{q}(x)/x^{2q-1}\big|^{2}\textit{d}x
\end{split}
\end{equation}
and for $q\equiv1\,(2)$:
\begin{equation}\label{eq:6.43}
\begin{split}
&\bigintssss\limits_{-\infty}^{\infty}\alpha^{2}\mathcal{P}_{q}(\alpha)\textit{d}\alpha=\sum_{m=1}^{\infty}\mathcal{Q}_{q}(m,2)=\\
&=\frac{1}{2}\Bigg(\frac{\pi^{q-1}}{2\Gamma(q)}\Bigg)^{2}\Bigg\{\underset{\,\,\,(d,2mk^{2})=1}{\sum_{m,d,k=1}^{\infty}}\frac{r^{2}_{2}\big(mk^{2},d;q\big)}{d^{2q-3}\big(mk^{2}\big)^{3/2}}\mu^{2}(m)+2^{2q}\underset{\,\,d\equiv0\,(4)}{\underset{\,\,(d,mk^{2})=1}{\sum_{m,d,k=1}^{\infty}}}\frac{r^{2}_{2,\chi}\big(mk^{2},d;q\big)}{d^{2q-3}\big(mk^{2}\big)^{3/2}}\mu^{2}(m)\Bigg\}=\\
&=\lim_{X\to\infty}\frac{1}{X}\bigintssss\limits_{ X}^{2X}\big|\mathcal{E}_{q}(x)/x^{2q-1}\big|^{2}\textit{d}x\,.
\end{split}
\end{equation}
%
%
%
%
%
%
This concludes the proof.
\end{proof}
\noindent
\textit{\textbf{Acknowledgements.} I would like to express my sincere gratitude to Prof. Amos Nevo for his support throughout the writing of this paper. I would also like to thank Prof. Ze{\'e}v Rudnick for a stimulating discussion on this fascinating subject, which subsequently lead to the writing of this paper.    
}

\end{document}